\documentclass[final,10pt]{elsarticle}
\usepackage{lineno,hyperref}
\usepackage{epsfig,amssymb,amsmath,version,amsthm}
\usepackage{amssymb,version,graphicx,fancybox,mathrsfs,multirow}
\usepackage[notcite,notref]{showkeys}
\usepackage{url,hyperref}
\usepackage{subfigure,epstopdf,bm}
\usepackage{color}
\usepackage[notcite,notref]{showkeys}
\usepackage{url,hyperref}
\usepackage{stmaryrd}
\usepackage{subfigure,epstopdf,arydshln}
\usepackage{color}
\usepackage{cases}
\usepackage{accents}
\usepackage{appendix}
\usepackage{algorithm}
\usepackage{algorithmicx}
\usepackage{algpseudocode}
\usepackage{amsmath}
\usepackage{amsfonts}
\usepackage{amssymb}
\usepackage{graphicx}%
\usepackage{wrapfig}
\modulolinenumbers[5]

\catcode`\@=11 \theoremstyle{plain}
\@addtoreset{equation}{section}
\renewcommand{\theequation}{\arabic{section}.\arabic{equation}}
\@addtoreset{figure}{section}
\renewcommand\thefigure{\thesection.\@arabic\c@figure}
\@addtoreset{table}{section}
\renewcommand\thetable{\thesection.\@arabic\c@table}

\floatname{algorithm}{Algorithm}
\newtheorem{thm}{\bf Theorem}

\newtheorem{proposition}{Proposition}[section]
\newenvironment{theorem}{\begin{thm}} {\end{thm}}
\newtheorem{cor}{\bf Corollary}

\newtheorem{lmm}{\bf Lemma}

\newenvironment{lemma}{\begin{lmm}}{\end{lmm}}
\theoremstyle{remark}
\newtheorem{rem}{Remark}[section]

\def \ri {{\rm i}}

\newcommand{\bs}[1]{\boldsymbol{#1}}
\def \er {{\bs e}_r}
\def \et {{\bs e}_\theta}
\def \ep {{\bs e}_\varphi}
\def \vt {{\bs \Phi}_l^m}

\DeclareSymbolFont{ugmL}{OMX}{mdugm}{m}{n}
\SetSymbolFont{ugmL}{bold}{OMX}{mdugm}{b}{n}

\DeclareMathAccent{\wideparen}{\mathord}{ugmL}{"F3}

\renewcommand \wedge \times

\begin{document}
%\graphicspath{{../figs/}}
\begin{frontmatter}
\title{On time-domain NRBC for Maxwell's equations and its application in accurate  simulation of  electromagnetic invisibility cloaks}

\author[fn1]{Bo Wang}\ead{bowang@hunnu.edu.cn}
\author[fn2]{Zhiguo Yang}\ead{yang1508@purdue.edu}
\author[fn3]{Li-Lian Wang\corref{cor1}}\ead{lilian@ntu.edu.sg}
\author[fn4]{Shidong Jiang}\ead{shidong.jiang@njit.edu}

\cortext[cor1]{Corresponding author}

\address[fn1]{LCSM(MOE), School of Mathematics and Statistics, Hunan Normal University, Changsha, Hunan, 410081, P. R. China.}
\address[fn2]{Department of Mathematics, Purdue University, West Lafayette, Indiana, 47906, USA.}
\address[fn3]{Division of Mathematical Sciences, School of Physical
	and Mathematical Sciences, Nanyang Technological University,
	637371, Singapore.}
	\address[fn4]{Department of Mathematical Sciences, New Jersey Institute of Technology, Newark, New Jersey, 07102, USA.}

%% Group authors per affiliation:
\begin{abstract}
In this paper, we present analytic formulas of the temporal convolution kernel functions involved in the time-domain non-reflecting boundary condition (NRBC) for the electromagnetic scattering problems. Such exact formulas themselves lead to accurate and efficient algorithms for computing the NRBC for  domain reduction of the time-domain Maxwell's system in $\mathbb R^3$.
 A second purpose of this paper is to derive a new time-domain model for the electromagnetic invisibility cloak. Different from the existing models,   it contains only one  unknown field and the seemingly complicated  convolutions can be computed as efficiently as the temporal convolutions in the NRBC.
 The governing equation in the cloaking layer is valid for general geometry, e.g.,  a spherical or  polygonal layer.  Here, we aim at simulating the spherical invisibility cloak. We  take the advantage of radially stratified dispersive media and special geometry, and develop an efficient vector spherical harmonic (VSH)-spectral-element method for its accurate  simulation. Compared with limited results on FDTD simulation,
 the proposed method is optimal in both accuracy and computational cost.  Indeed, the saving in computational time is significant.
 % are derived based on the compact VSH expansion of the solenoidal scattering field. Explicit expressions of the NRBKs in time domain are presented.
%
% umerical results with high accuracy can be obtained in much less computation time and resource compared with the simulations use classic FDTD scheme.
% Based on the derived new NRBC formulations and Drude model for dispersive media in the cloak, a new truncated model problem for spherical invisibility cloaks is formulated. The new model maintain the symmetry of the problem which inspired us to propose an efficient VSH-spectral-element method for numerical simulation. Numerical results with high accuracy can be obtained in much less computation time and resource compared with the simulations use classic FDTD scheme.
\end{abstract}

\begin{keyword} Maxwell's system,
electromagnetic wave scattering, anisotropic and dispersive medium, non-reflecting boundary condition,  convolution, invisibility cloaking.
\end{keyword}

\end{frontmatter}

\section{Introduction}\label{sect::intro}
Numerical simulation of electromagnetic wave propagations in anisotropic and dispersive medium is of fundamental importance  in many scientific applications and engineering designs. The model problem of interest is the time-dependent  three-dimensional  Maxwell's system:
\begin{subequations}\label{orignsystem}
	\begin{numcases}{}
	{\!\!\!}\partial_t \bs D(\bs r, t)-\nabla\wedge \bs H(\bs r,t)=\bs J(\bs r,t) \quad \;  {\rm in} \;\;\; {\mathbb R}^3,\quad t>0,\\[3pt]
	{\!\!\!}{\partial_t  \bs B}(\bs r,t)+\nabla\wedge \bs E(\bs r,t)=\bs 0 \qquad\quad\;\;\;   {\rm in} \;\;\; {\mathbb R}^3,\quad t>0,
	\end{numcases}
\end{subequations}
with the  constitutive relations
\begin{equation}\label{consteqn}
\bs D=\varepsilon_0\bs{\varepsilon}\bs E,\quad \bs B=\mu_0\bs{\mu}\bs H,
\end{equation}
where $\bs r=(x,y,z)\in {\mathbb R}^3$,  $\bs E, \bs H$ are respectively the electric and magnetic fields,   $\bs D, \bs B$ are  the corresponding electric displacement and magnetic induction fields,  and  ${\bs J}$ is the electric current density.  In
\eqref{consteqn},
$\varepsilon_0, \mu_0$ are the electric permittivity and magnetic permeability in vacuum and $\bs{\varepsilon}, \bs\mu$ are the relative permittivity and permeability tensors of the material. Throughout the paper, we denote
$c=1/{\sqrt{\varepsilon_0\mu_0}}$ and $\eta=\sqrt{ \mu_0/\varepsilon_0}. $
%\begin{equation*}\label{constce}
%c=1/{\sqrt{\varepsilon_0\mu_0}}, \quad \eta=\sqrt{ \mu_0/
%	\varepsilon_0}.
%\end{equation*}
Without loss of generality, we assume that the inhomogeneity or dispersity of the medium is confined in a bounded domain $\Omega$ and ${\bs J}$ is compactly supported. As illustrated in Figure \ref{Figmodeldomain},   both $\Omega$ and ${\rm supp}(f)$ are contained in a  ball $\Omega_b$ of radius $b$.  The Maxwell's system \eqref{orignsystem} is supplemented with the initial conditions:
\begin{equation}\label{inicond}
\bs E(\bs r,0) =\bs E_0(\bs r),\quad \bs H(\bs r,0)=\bs H_0(\bs r) \quad  {\rm
	in} \;\;\;  {\mathbb R}^3,
\end{equation}
where $\bs E_0$ and $\bs H_0$ are also assumed to be compactly supported in the ball $\Omega_b.$  As usual,
we impose the far-field Silver-M\"{u}ller radiation boundary condition on the scattering fields: ${\bs E}^{\rm sc}=\bs E-\bs E^{\rm in}$ and ${\bs H}^{\rm sc}=\bs H-\bs H^{\rm in}$  as follows
\begin{equation}\label{EtHSM}
\partial_t {\bs E}^{\rm sc}_T-\eta \,\partial_t {\bs H}^{\rm sc}\times\hat  {\bs r}=o(|\bs r|^{-1})\quad {\rm as}\;\;    |\bs r|\to \infty,\;\;  t>0, \;\;
\end{equation}
where  $\hat {\bs  r}=\bs r/|\bs r|,$ and $\bs E_T^{\rm sc}:= \hat {\bs r}\times  \bs E^{\rm sc}\times \hat{ \bs r}$
is the tangential component of  $\bs E^{\rm sc}.$ Here,  $\bs E^{\rm in}, \bs H^{\rm in}$ are the incident fields.
%\begin{equation}\label{ETdefn}
%\bs E_T^{\rm sc}:= \hat {\bs r}\times  \bs E^{\rm sc}\times \hat{ \bs r}
%\end{equation}
%is the tangential component of the scattering field $\bs E^{\rm sc}.$

Despite its seemly simplicity, the system \eqref{orignsystem}-\eqref{EtHSM} is notoriously difficult to solve numerically. Some of the major numerical issues are (i) unboundedness of the computational domain; % with slowly decaying and oscillatory solution;
(ii) the incompressibility implicitly implied by \eqref{orignsystem} (i.e., ${\rm div}(\bs D)={\rm div}(\bs B)=0$); and (iii) the coefficients $\bs \varepsilon$ and $\bs \mu$ might be singular or frequency-dependent (see \eqref{materials0} and \eqref{DrudeE}).  In this paper, we shall address all these three aspects.

In regards to the first  issue, the method of choice  typically includes the perfectly matched layer (PML) technique \cite{Bere94} or  the artificial  boundary condition \cite{Enguist77,Gro.K95,Hagstrom99}. In particular,  the latter is  known as the absorbing boundary condition  (ABC), if it  leads to a well-posed initial-boundary value problem (IBVP)  and the reflection near the boundary is controllable.
% some ``energy'' can be absorbed  at the boundary.
Ideally, if the solution of the reduced problem coincides with  that of the original problem,  then the underlying  artificial  boundary condition is called  a transparent (or nonreflecting) boundary condition (TBC) (or NRBC).

In this paper,  we resort to  the NRBC  to reduce the problem \eqref{orignsystem}-\eqref{EtHSM}  to an IBVP inside  a spherical bounded domain $\Omega_b:=\{\bs r:|\bs r|<b\}$:
\begin{subequations}\label{reducedsystem}
	\begin{numcases}{}
	\partial_t \bs D-\nabla\wedge \bs H=\bs J; \quad   {\partial_t  \bs B}+\nabla\wedge \bs E=\bs 0 \quad {\rm in} \;\;\; \Omega_b,\quad t>0,\label{reducedsystemeq}\\
	{\bs E}={\bs E}_0,\quad{\bs H}={\bs H}_0\quad   \text{in} \;\;\; \Omega_b,\quad t=0,\label{initA}\\
	\partial_t{\bs E}_T-\eta\partial_t\bs H\times \hat {\bs r}-\mathscr{T}_b[{\bs E}]=\partial_t{\bs E}_T^{\rm in}-\eta\partial_t\bs H^{\rm in}\times \hat {\bs r}-\mathscr{T}_b[{\bs E}^{\rm in}]:=\bs h \quad {\rm at} \;\;\; r=b,
	\label{reducedsystemnrbc}
	\end{numcases}
\end{subequations}
where the NRBC \eqref{reducedsystemnrbc} involves  the capacity operator  ${\mathscr T}_b$  to be specified in Theorem \ref{theo21}.
It is important to point out  that the NRBC is formulated upon the  scattering fields ${\bs E}^{\rm sc},  {\bs H}^{\rm sc}$,
so  $\bs h$ inevitably contains $\mathscr{T}_b[{\bs E}^{\rm in}].$ As such, it  is rather complicated to implement  and computationally time-consuming due to the involvement of the vector spherical harmonics (VSH) expressions of $\bs E^{\rm in}$ and history dependence in time  induced by the temporal convolution (see \eqref{oldEtM} and Remark \ref{rmk::tbch}).	
%\begin{figure}[!ht]
%	\centering
%	\includegraphics[width=0.35\linewidth]{domainfig}
%	\caption{\small An illustration of the geometry}
%	\label{Figmodeldomain}
%\end{figure}
To avoid the serious problem of high costs of computing $\bs h$, we instead solve the total fields in the subdomain $\Omega_{b_0}:=\{\bs r: r<b_0\}$ $\subseteq$ $\Omega_b$ (see Figure \ref{Figmodeldomain}), and compute the outgoing scattering fields in a narrow spherical shell $\Omega_b \setminus \Omega_{b_0}.$   More precisely,  we
reformulate \eqref{reducedsystem} as
%In the new model, we shall have discontinuity across the interior interface $r=R_1$. For the well-posedness of the model, we use the standard transmission condition on the interface. Then the new model can be summarized as follow:
\begin{subequations}\label{reducedsystem1}
	\begin{numcases}{}
	\partial_t \bs D-\nabla\wedge \bs H=\bs J,\quad  {\partial_t  \bs B}+\nabla\wedge \bs E=\bs 0, \quad  r<b_0,\quad t>0,\label{reducedsystem1eq1}\\
	\partial_t \bs D^{\rm sc}-\nabla\wedge \bs H^{\rm sc}=\bs J,\quad  {\partial_t  \bs B}^{\rm sc}+\nabla\wedge \bs E^{\rm sc}=\bs 0, \quad b_0<r<b,\quad t>0,\label{reducedsystem1eq2}\\
	(\bs E-\bs E^{\rm sc})\times\hat {\bs r}=\bs E^{\rm in}\times \hat{\bs r};\quad (\bs H-\bs H^{\rm sc})\times\hat{\bs r}=\bs H^{\rm in}\times \hat {\bs r},\quad \text{at} \;\;\; r=b_0,  \quad t>0,  \label{transcond1}\\
	\partial_t\bs E_T^{\rm sc}-\eta\partial_t\bs H^{\rm sc}\times \hat {\bs r}-\mathscr{T}_b[{\bs E}^{\rm sc}]=\bs 0,\quad   \text{at} \;\;\; r=b,  \quad t>0,\label{reducedsystem1nrbc}\\
	{\bs E}={\bs E}_0,\;\;{\bs H}={\bs H}_0,\quad 0<r<b_0,\;   {\bs E}^{\rm sc}={\bs E}_0^{\rm sc},\;\;{\bs H}={\bs H}_0^{\rm sc}, \quad   b_0<r<b,\; t=0. \label{intialEH}
	%{\bs E}^{\rm sc}={\bs E}_0^{\rm sc},\quad{\bs H}={\bs H}_0^{\rm sc}, \quad   b_0<r<b,\;\; t=0.
	\end{numcases}
\end{subequations}
As a consequence, the NRBC \eqref{reducedsystem1nrbc} depends solely on the scattering fields, which leads to more efficient algorithm. Note that \eqref{transcond1} is obtained from the classic transmission conditions (see,  e.g., \cite[Sec. 1.5]{orfanidis2002electromagnetic} and \cite{Monk03}), that is, the continuity of the tangential components of the total fields $\bs E$ and $\bs H$ at the artificial interface $r=b_0$.

\begin{wrapfigure}{r}{0.41\textwidth}
	\center
	\vspace{-2em}
	\includegraphics[width=0.35\textwidth]{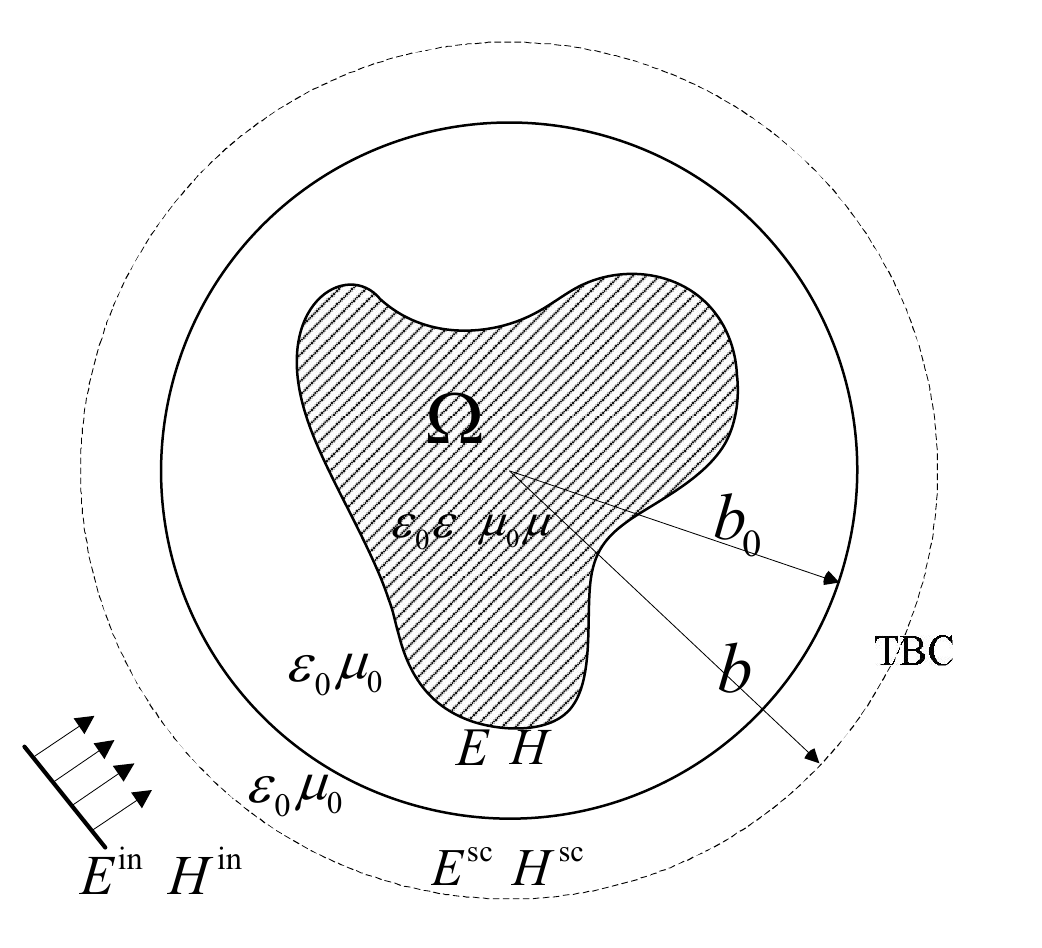}
	\caption{\small An illustration of the geometry}
	\label{Figmodeldomain}
\end{wrapfigure}
One of the main purposes of this paper is devoted to deriving new formulas of the NRBC
by using the compact VSH expansion of the scattering field. For the convolution kernel in the NRBC, an explicit expression in time domain is obtained based on a direct inversion of the Laplace transform (see Theorem \ref{rholthm} below). As shown in \cite{Alpert02,Wang2Zhao12}, the explicit expressions of NRBKs allow for a rapid and accurate evaluation of the convolution in NRBC.

The second main purpose of this paper is to propose an accurate and efficient numerical method for the simulation of the electromagnetic invisibility cloaks by using the new NRBC formula and the compact VSH expansion. Transformation optics originated from the seminal works \cite{pendry.2006,leonhardt2006optical}  offers an effective approach to design  novel and unusual optical devices such as the invisibility cloaks  (see, e.g., \cite{pendry.2006, Greenleaf2009siamreview}), superlens (see, e.g., \cite{yan2008superlens,tsang2008superlens}) and beam splitters (see, e.g., \cite{rahm2008optical}), etc. Numerical simulation plays a crucial role in modelling of the electromagnetic wave interaction with these devices since it serves as a reliable tool to the justification of expensive physical experiments and validation of theoretical predictions. Over the recent years, intensive  simulations and analysis have been devoted to the frequency domain (see, e.g., \cite{cummer2006full,ruan2007ideal, zhai2010finite, zhang2008rainbow, liu2011approximate,kohn2010cloaking, yang_wang_2015, yang2016seamless}). Due to the fact that metamaterials used for manufacturing such kind of devices are unavoidably dispersive (cf. \cite{pendry1996extremely}), i.e., $\bs \varepsilon$ and $\bs \mu$ are frequency-dependent, time-domain mathematical models and simulations of anisotropic and dispersive electromagnetic devices are of fundamental importance. However, only limited works are available for the time-domain simulations including the FDTD \cite{zhao2008full,zhao2009full,okada2012fdtd} and the FETD \cite{li2012time, li2012developing,li2013adaptive, yang2018mathematical, li2019improved}.  Because of the computational complexity, so far the numerical simulation of spherical cloaking structures has only been examined by \cite{zhao2009full} with a parallel implementation of FDTD method. In this paper, we propose a new formulation of the spherical cloak model in the time domain using the Drude dispersion model (cf. \cite{orfanidis2002electromagnetic}). This new formulation allows us to use the symmetry of the problem together with the compact VSH expansions to provide an efficient VSH-spectral-element method for the simulation. Compared with the classic FDTD based algorithm, the VSH-spectral-element method can produce accurate numerical results in much less computational cost.

The rest of the paper is organised as follows. In section \ref{sect::tbc}, we present some new formulas of the NRBC and derive an explicit expression for the underlying convolution kernel. In section \ref{3dcloak}, we first derive a new time-domain model for the spherical dispersive cloaks by using Drude model. Then, an VSH-spectral-element method with Newmark's time integration scheme is proposed for efficient simulation of spherical cloaks. Ample interesting simulations for the spherical dispersive cloaks are presented in section \ref{sect:numer} to show the accuracy and efficiency of the proposed numerical scheme.

%\comm{Should (1.6f) to be 0? should $\bs e_r$ be $\hat {\bs  r}$?}

%
%\begin{rem}\label{newmodelremark}
%It is worthy to emphasize that the sub-domain $B_0$ is not required to be a ball and it can be designed to have other shape for the sake of numerical discretization, e.g., polyhedron for element-based numerical method.
%\end{rem}		

%By eliminating  $\bs H$ from \eqref{reducedsystem}, we obtain
%\begin{subequations}\label{trunsystem}
%\begin{numcases}{}
%{\partial^2_t \bs D}+c^2\,\nabla \times (\bs\mu^{-1}\nabla\wedge (\bs \varepsilon^{-1}\bs D))=\bs F \quad   {\rm in}\;\; B,\;\; t>0,\label{trunsystemeq}\\
%\bs D =\bs D_0,\quad {\partial_t \bs D}=\bs D_1 \quad   {\rm
%	in}\;\; B, \;\; t=0,  \\
%\partial_t {\bs D}_T  +c\,  (\nabla \wedge \bs D) \times \hat {\bs r}-\mathscr T_b[{\bs D}] =\bs 0 \quad{\rm at}\;\; r=b,\;\;  t>0,\label{trunsystemnrbc}
%\end{numcases}
%\end{subequations}
%where
%$$\bs F=\varepsilon^{-1}\partial_t{\bs J},\quad \bs D_1= \big(\nabla\times \bs H_0+\bs J|_{t=0}\big).    $$

%Much of this work is  to develop fast and accurate means for the computation of the  time-domain TBC, which is global in both time and space.

\section{Computation of  time-domain NRBC }\label{sect::tbc}
\setcounter{equation}{0}
\setcounter{lmm}{0}
\setcounter{thm}{0}

In this section, we  present the  formulations of  the capacity  operator ${\mathscr T}_b$  involved in the time-domain NRBC,  and then  derive  some analytically  perspicuous formulas for the associated temporal convolution kernels (dubbed as NRBKs),
which are crucial for efficient and accurate computation of the NRBC, and in return for its seamless integration with the interior solvers.
\subsection{Formulation of  time-domain NRBC}
 Let $L^2(\Omega)$ be the usual space of square integrable functions on $\Omega,$ and denote  $ \bs L^2(\Omega)=(L^2(\Omega))^3.$ We introduce the spaces
\begin{equation*}\label{Hdiv}
\begin{split}
{\mathbb H}({\rm div};\Omega)=\big\{\bs v  \in {\bs L}^2(\Omega) : {\rm div}\bs  v\in L^2(\Omega) \big\};  \;\;
{\mathbb H}({\bf curl};\Omega)=\big\{\bs v\in (L^2(\Omega))^3: \nabla\times {\bs v}\in  {\bs L}^2(\Omega)  \big\},
\end{split}
\end{equation*}
which are equipped with  the graph norms  as defined  in \cite[p. 52]{Monk03}.  We further define
\begin{equation*}\label{zerosps}
{\mathbb H}_0({\rm div};\Omega)=\big\{\bs v\in {\mathbb H}({\rm div};\Omega) : {\rm div} \bs v=0 \big\}.\;
%{\mathbb H}_0({\bf curl};\Omega)=\big\{\bs v\in {\mathbb H}({\bf curl};\Omega)%\cap \big({H}^1(\Omega)\big)^3
%: \bs v\times{\bs e_r}|_{r=a}=\bs 0\big\}.
\end{equation*}

For $0\not=\bs x \in \mathbb R^3,$ let $\er=\bs x/|\bs x|.$ Recall that the VSH
\begin{equation}\label{VSHYpsi}
\big\{ \bs Y_l^m , \bs\Psi_{l}^m,
\bs \Phi_l^m\big\}:=\big\{ Y_l^m \er, \nabla_SY_{l}^m,
\nabla_S Y_l^m \wedge \er\big\}
\end{equation}
 used in the Spherepack \cite{Swa.S00} forms a complete orthogonal basis of ${\bs L}^2(S):=(L^2(S))^3,$ where $\{Y_l^m\}$ are the spherical harmonic basis defined on the unit sphere $S$ as in \cite{Nedelec}. Nevertheless, the following compact form of the VSH expansion of a solenoidal or divergence-free field in \eqref{consept}  can simplify the derivation of NRBC. Moreover, it will lead to more efficient spectral-element algorithm for the 3D spherical cloaking simulation in Section \ref{3dcloak}.
\begin{proposition}\label{EFsinsol}
	For $\bs u\in {\mathbb H}_0({\rm div};\Omega)$, we can  write
	\begin{equation}\label{consept}
	{\bs u}= u_{00}\,\bs Y_0^0 +  \sum_{l=1}^\infty\sum_{|m|=0}^l \Big\{u_{lm}\, \vt
	+ \nabla \times \big(\tilde u_{lm}\,  \vt\big)\Big\},
	\end{equation}
	where $u_{00}$  satisfies
	\begin{equation}\label{eqnu00}
	\Big(\frac{d}{dr}+\frac{2}{r} \Big)u_{00} = 0\;\;\; {\rm or}\;\;\; u_{00}=\frac C {r^2},
	\end{equation}
	for a constant $C$ depends on the average value of the $\bs e_r$ component of $\bs u$ on $S$.
	The expansion \eqref{consept} can be reformulated in terms of the VSH \eqref{VSHYpsi} as follows
	\begin{equation}\label{conseptequiv}
	{\bs u}= u_{00}\,\bs Y_0^0 +  \sum_{l=1}^\infty\sum_{|m|=0}^l \Big\{\frac {\beta_l} r \tilde u_{lm} \,  \bs Y_l^m
	+\hat \partial_r \tilde u_{lm} \,\bs\Psi_l^m +u_{lm}\, \vt \Big\},
	\end{equation}
	where $\beta_l:=l(l+1)$ and %by  the orthogonality  \eqref{orthL3}-\eqref{orthL3S},  % we denote $\beta_l=l(l+1),$ and have
	\begin{equation}\label{uvwexp00}
	\begin{split}
	&    u_{lm}(r)= {\beta_l^{-1}}\big \langle \bs u, \vt  \big\rangle_S,
	\quad   r^{-1} \tilde u_{lm}(r)= \beta_l^{-1}\big \langle \bs u, \bs Y_l^m \big \rangle_S.
	\end{split}
	\end{equation}
\end{proposition}
\begin{proof} We first show that if \eqref{eqnu00} holds, then the expansion \eqref{consept} automatically satisfies
	${\rm div}\, \bs u =0.$  Note that   ${\rm div}(u_{lm}\, \vt)=0$ (cf. \eqref{divformula}). Performing  the divergence operator on \eqref{consept}, and using \eqref{newvect}, we have ${\rm div}\, \bs u=0,$ if $u_{00}$ satisfies the equation in  \eqref{eqnu00} which has explicit solution:  $u_{00}=C/r^2$.
	Thanks to \eqref{vt2yy2}, the expansion \eqref{conseptequiv} follows immediately from \eqref{consept}. Then \eqref{uvwexp00} is a direct consequence of the orthogonality of VSH.
\end{proof}

\begin{rem}\label{y00} For any constant $C,$ the field $ C\, \er/r^2, r>0$ (note: $Y_0^0= 1/{(2\sqrt \pi)}$) is solenoidal. Given a  vector field,  the VSH expansion coefficients in  \eqref{uvwexp00}    can be evaluated accurately and efficiently  by using  discrete VSH-transforms in SpherePack	\cite{Swa.S00}. 	\qed
\end{rem}

The time-domain NRBC to be formulated below  involves  the modified spherical Bessel function {\rm(}cf. \cite{watson}{\rm)} defined by
	\begin{equation}\label{akasympc}
	k_{l}(z)=\sqrt{\frac \pi {2 z}} K_{l+1/2}(z),
	\end{equation}
	with  $K_{l+1/2}(\cdot)$ being the
	modified Bessel function of the second kind of order $l+1/2$, together with  the temporal  convolution and
 inverse Laplace transform:
	$$ (f\ast g)(t)=\int_0^t f(\tau)g(t- \tau)d\tau,\quad {\mathscr L}^{-1}[F](t)= \frac 1 {2\pi \ri}\int_{\gamma-\infty \ri}^{\gamma+\infty\ri}F(s) e^{s t} ds,$$
	where $F(s)$ is the Laplace transform of $f(t),$ and  the integration is done along the vertical line $\Re (s) = \gamma$ in the complex plane such that $\gamma$ is greater than the real part of all singularities of $F(s)$.
%By using the VSH expansions \eqref{EFsinsol} and \eqref{conseptequiv}, we can derive the time-domain NRBC.

The formulation of the capacity  operator can be found in e.g.,  \cite{Col.K98,Nedelec,Monk03}, but the notation and normalisation are very different.  Here, we feel compelled to sketch its derivation.
\begin{theorem}\label{theo21}
	The time-domain capacity operator $\mathscr T_b$ takes the form
	\begin{equation}\label{oldEtM}
	{\mathscr T}_b[\bs E^{\rm sc}]:=\frac c b \sum_{l=1}^\infty\sum_{m=-l}^l\Big\{
	\big(\rho_l\ast {\psi}^{(1)}_{lm}\big)\,{\bs\Psi}_l^m+\big(\sigma_l\ast {\psi}^{(2)}_{lm}\big)\,\vt\Big\},
	\end{equation}
	where the convolution kernels are given by the inverse Laplace transforms:
	\begin{equation}\label{arhol}
	\begin{split}
	\rho_l(t)&={\mathscr L}^{-1}\bigg[z\,\bigg(\frac{z\, k_l(z)}{k_l(z)+z\, k_l'(z)}+1\bigg) \bigg](t), \quad \sigma_l(t) ={\mathscr L}^{-1}\bigg[1+z+z\frac{k_l'(z)}{k_l(z)}\bigg](t),
	\end{split}
	\end{equation}
with	$z=\frac{sb}{c}$, and where $\{{\psi}^{r}_{lm}, {\psi}^{(1)}_{lm}, {\psi}^{(2)}_{lm}\}$ are   the VSH expansion coefficients of $\bs E^{\rm sc}$ at the spherical surface $r=b$, i.e.,
	\begin{equation}\label{vshexpscatfield}
	\bs E^{\rm sc}=\psi_{00}\,\bs Y_0^0 + \sum_{l=1}^\infty\sum_{|m|=0}^l\Big\{{\psi}^{r}_{lm}\,  \bs Y_l^m\, +{\psi}_{lm}^{(1)}\bs \Psi_{l}^m+{\psi}_{lm}^{(2)}\, \vt\Big\}.
	\end{equation}
%	according to definition \eqref{Eexpansion}.   and the involved convolution is defined as usual: $(f\ast g)(t)=\int_0^t f(\tau)g(t- \tau)d\tau$.
	%with $\delta$ being the Dirac delta function.
\end{theorem}
\begin{proof}
	Consider  the Maxwell's equations exterior to the artificial ball $B:=\{\bs r: |\bs r|<b\}:$
	\begin{subequations}\label{exteriorsystem}
		\begin{numcases}{}
		\varepsilon_0 \partial_t \bs E^{\rm e}-\nabla\wedge \bs H^{\rm e}=\bs 0,\quad
		\mu_0 {\partial_t  \bs H}^{\rm e}+\nabla\wedge \bs E^{\rm e}=\bs 0,\;\;  {\rm in}\;\;
		{\mathbb R}^3\setminus \bar B,\;\; t>0,\label{Heqnnew}\\
		{\bs E}^{\rm e}\wedge \hat {\bs r} =\bs \lambda,  \quad r=b,\;\;\; t>0\label{Estime}\\
		{\partial_t \bs E_T^{\rm e}}- \eta \,
		{\partial_t \bs H^{\rm e}}\times  \hat{\bs r}
		=o(r^{-1}),\quad  |\bs r|\to \infty, \;\; t>0, \label{bcondnew}\\
		\bs E^{\rm e} =\bs H^{\rm e}=\bs 0, \quad  {\rm
			in}\;\; {\mathbb R}^3\setminus \bar B,\;\;\; t=0, \label{inicondnew}
		\end{numcases}
	\end{subequations}
	where $\bs \lambda$ is a given field. % in ${\bs L}^2_T(S).$
	 It is known that this system can be solved analytically by using Laplace transform in time and separation of variables in space. For this purpose,  we denote by  $\breve{\bs E}^{\rm e}, \breve{\bs H}^{\rm e}$ and $\breve{\bs \lambda}$ the Laplace transforms of $\bs {E}^{\rm e}, \bs {H}^{\rm e}$ and ${\bs \lambda}$ with respect to $t$, respectively.  As  $\breve{\bs E}^{\rm e}$ is a divergence-free vector field, we have from Proposition \ref{EFsinsol} that
	\begin{equation}
	\label{divfreevshexp}
	\breve{\bs E}^{\rm e}=\breve u_{00}\bs Y_0^0+\sum_{l=1}^\infty\sum_{|m|=1}^l\big\{\breve{u}_{lm} \bs\Phi_l^m+\nabla\times\big(\breve{v}_{lm} \bs\Phi_l^m\big)\big\}.
	\end{equation}
	 According to \cite[Chap. \!\!5]{Nedelec},  the Laplace transformed system of  \eqref{exteriorsystem}   in $s$-domain
	 \begin{equation}\label{tranformms}
	 \varepsilon_0 s 	\breve{\bs E}^{\rm e} -\nabla\times 	\breve{\bs H}^{\rm e}=\bs 0,\quad  \mu_0 s \breve{\bs H}^{\rm e} +\nabla\times 	\breve{\bs E}^{\rm e}=\bs 0;\quad \breve {\bs E}^{\rm e}\wedge \hat {\bs r} =\breve {\bs \lambda} \;\; {\rm at}\;\; r=b,
	 \end{equation}
	 has the  exact  solution \eqref{divfreevshexp} with $\breve u_{00}=0$ and
	\begin{equation}\label{solutionLap}
	\breve{u}_{lm}(r)=-\frac{k_l(sr/c)}{k_l(sb/c)} \breve{\lambda}_{lm}^{(1)},\quad \breve{v}_{lm}=\frac{k_l(sr/c)}{\hat{\partial}_rk_l(sb/c)}\breve{\lambda}_{lm}^{(2)},\quad l\ge 1,
	\end{equation}
	where $\hat \partial_r=\frac d {dr}+\frac 1 r$, $\{\breve{\lambda}_{lm}^{(1)}, \breve{\lambda}_{lm}^{(2)}\}$ are the VSH expansion coefficients of $\breve{\bs{\lambda}}$ (involving only the tangential components).
Moreover,  we can derive the electric-to-magnetic Calderon (EtMC)  operator that maps the data $\breve{\bs  \lambda}$   to $\breve{\bs H}^{\rm e}\times \hat {\bs r}$ (cf. \cite{Col.K98,Monk03}) as follows
	\begin{equation}\label{sdomainbc1}
	\breve{\bs H}^{\rm e} \wedge \hat {\bs r}
	=-\frac{1}{s\mu_0}\sum_{l=1}^\infty\sum_{m=-l}^l\bigg\{\frac{ s^2k_l(sr/c)}{c^2\hat \partial_r k_l(sb/c)}\breve{\lambda}_{lm}^{(2)}\bs\Psi_l^m
	-\frac{\hat{\partial}_rk_l(sr/c)}{k_l(sb/c)}\breve{\lambda}_{lm}^{(1)}\bs\Phi_l^m\bigg\},
	\end{equation}
	which can be derived from the second equation in \eqref{divfreevshexp} and the properties of VSH in Appendix \ref{spherehar}.
%	using the property(cf.  \cite{Alpert02}):  {\color{red} What does this mean?}
%	$$\mathcal L_l[k_l(sr/c)]=-\frac{s^2}{c^2}k_l(sr/c),$$
By requiring the scattering  fields  $\{\breve{\bs E}^{\rm sc}, \breve{\bs H}^{\rm sc}\}$  to be identical to the exterior fields $\{\breve{\bs E}^{\rm e}, \breve{\bs H}^{\rm e}\}$ across  the
	artificial boundary $r=b,$ and setting $\breve{\bs
	E}^{sc}\times \hat{\bs r}|_{r=b}=\breve{ \bs \lambda},$ we obtain
	\begin{equation}\label{sdomaintotdomain}
	\begin{split}
	\mu_0 s\breve{\bs H} ^{\rm sc}\wedge \hat {\bs r}\big|_{r=b}&=-\sum_{l=1}^\infty\sum_{m=-l}^l\Big\{\frac{s^2 k_l(sb/c)}{c^2\hat{\partial}_rk_l(sb/c)}\breve{\psi}^{(1)}_{lm}\bs\Psi_l^m+\frac{\hat{\partial}_rk_l(sb/c)}{k_l(sb/c)}\breve{\psi}^{(2)}_{lm}\bs\Phi_l^m\Big\}
	\\&=-\sum_{l=1}^\infty\sum_{m=-l}^l\Big\{\Big(\frac{1}{b}\mathscr{L}[\rho_l]-\frac{s}{c}\Big)\breve{\psi}^{(1)}_{lm}\bs\Psi_l^m+\Big(\frac{1}{b}\mathscr{L}[\sigma_l]-\frac{s}{c}\Big)\breve{\psi}^{(2)}_{lm}\bs\Phi_l^m\Big\}\\
	&=-\frac{1}{b}\sum_{l=1}^\infty\sum_{m=-l}^l\Big\{\mathscr{L}[\rho_l]\breve{\psi}^{(1)}_{lm}\bs\Psi_l^m
	+\mathscr{L}[\sigma_l]\breve{\psi}^{(2)}_{lm}\bs\Phi_l^m\Big\}+\frac{s}{c}\breve{\bs E}_{T}^{\rm sc},
	\end{split}
	\end{equation}
	where $\sigma_l(t)$ and $\rho_l(t)$ are defined in \eqref{arhol}.
	Transforming the relation  \eqref{sdomaintotdomain} back to  $t$-domain leads to the time-domain NRBC
	\begin{equation}\label{nrbconscatteringfield}
	\partial_t{\bs E}^{\rm sc}_T-\eta\partial_t\bs H^{\rm sc}\times \hat {\bs r}-\mathscr{T}_b[{\bs E}^{\rm sc}]=\bs 0 \quad {\rm at}\;\;\; r=b,
	\end{equation}
	and the capacity operator  ${\mathscr T}_b[\bs E^{\rm sc}]$ given by \eqref{oldEtM}.
	%We also refer to  \cite{zhaowangACES} for a very different derivation.
\end{proof}

\begin{rem}\label{rmk::tbch}  Note that the exact NRBC in  \cite{Hagstrom07} expressed as a system of  $\bs E^{\rm sc}$ and $\bs H^{\rm sc}$, is actually equivalent to the formulation \eqref{reducedsystem1nrbc} by using the VSH  $\big\{ \bs Y_l^m , \bs\Psi_{l}^m, \bs \Phi_l^m\big\}$.
	From the NRBC \eqref{reducedsystem1nrbc} on scattering fields, the NRBC \eqref{reducedsystemnrbc} on total field $\{ \bs E, \bs H\}$ can be derived straightforwardly as follows
	\begin{equation}\label{nrbcontotalfield}
	\partial_t{\bs E}_T-\eta\partial_t\bs H\times \hat {\bs r}-\mathscr{T}_b[{\bs E}]=\partial_t{\bs E}_T^{\rm in}-\eta\partial_t\bs H^{\rm in}\times \hat {\bs r}-\mathscr{T}_b[{\bs E}^{\rm in}]:=\bs h, \quad {\rm at}\;\;\; r=b.
	\end{equation}
	We shall see  from \eqref{oldEtM}  and \eqref{nrbcontotalfield} that the source term $\bs h$ needs to be precomputed from $\{ \bs E^{\rm in}, \bs H^{\rm in} \}$ if the model \eqref{reducedsystem} is used for solving the total field $\{ \bs E, \bs H\}$ in the whole computational domain $B$. However, the involved term $\mathscr{T}_b[{\bs E}^{\rm in}]$ is computationally costly due to: (i) The VSH expansion coefficients of the incident wave are necessary for the computation; (ii) in the time direction, two convolutions need to be calculated numerically for each group of VSH expansion coefficients; (iii) the numerical scheme for the convolution needs to be more accurate than the time discretization scheme for the model problem due to sensitivity of the operator $\mathscr{T}_b$ on the error. Thus, a great number of time consuming VSH expansion need to be performed.	\qed
\end{rem}

It is seen   from \eqref{oldEtM} that %and \eqref{newEtM}
to compute ${\mathscr T}_b[\bs E^{\rm sc}]$ at  $t>0,$ it requires
\begin{itemize}
\item[(i)] accurate  evaluation of the convolution kernel functions:  $\rho_l(t)$ and $\sigma_l(t)$ defined in  \eqref{arhol};
\item[(ii)] fast computation of the temporal convolutions:  $\rho_l\ast {\psi}^{(1)}_{lm}$ and $\sigma_l\ast {\psi}^{(2)}_{lm}.$
\end{itemize}
With these, we  can compute ${\mathscr T}_b[\bs E^{\rm sc}]$ on the sphere $r=b$  by using the VSH transform in  the Spherepack \cite{Swa.S00}.

%exact NRBC \eqref{reducedsystemnrbc} is global in both time and space, due to the involvement of
%the VSH expansion coefficients and temporal convolutions. Moreover, it requires to compute the inverse Laplace transform of the ratio of modified Bessel functions.

\medskip

%\subsubsection{Formulas of }
We now deal with the first issue. Interestingly,
the kernel function  $\sigma_l(t)$ coincides with  the NRBC kernel of  the transient wave equation, which admits the following explicit formula  %derived from the residue theorem
 (cf. \cite{Sofronov1998,Hagstrom07,Wang2Zhao12}).
%\comm{Write down as a theorem for the formulas of $\sigma_l$ and $\rho_l$!}
\begin{proposition}\label{sigmalthm} Let $\sigma_{l}(t)$ be the kernel function defined in \eqref{arhol}.
 %, and $z=sb/c$.
	Then we have
	\begin{equation}\label{kernela}
	\sigma_{l}(t)=\frac{c} b \sum_{j=1}^{l}z_j^l e^{\frac{c}{b} z_j^lt},\quad l\ge 1,\;\; t\ge 0,
	\end{equation}
		where $\{z_j^l\}_{j=1}^{l}$ are the zeros of $K_{l+1/2}(z)$ with $l\ge 1$.
	%Here, {\color{red} summarize the properties of the zeros here as with Lemma 2.1! Also $0$ is a note a pole!}
%	and
%	\begin{equation}\label{kernelomega}
%	\omega_{l}(t)=\frac{c} b \sum_{j=1}^{l}\big(z_j^l\big)^2 e^{\frac{c}{b} z_j^lt}+\delta(t)\sum_{j=1}^{l}z_j^l,\quad l\ge 1,\;\; t\ge 0.
%	\end{equation}
\end{proposition}
We remark that according to \cite[p. 511]{watson},  $K_{l+1/2}(z)$ has exactly $l$ zeros in conjugate pairs which are simple and lie in the second and third quadrants.  The interested readers might  refer to \cite[Lemma 2.1]{Wang2Zhao12} for more details.

 \smallskip

%\comm{We can write some remarks on the distribution of the zeros of $K_{l+1}(z)$! Then it is more convenient for writing below!}

Remarkably, we can derive  a very similar analytical  formula for the kernel function $\rho_l(t).$  Our starting point is to
rewrite the ratio of the modified Bessel functions in  \eqref{arhol} by using  \eqref{akasympc} as follows
\begin{equation}\label{defnf}
\begin{split}
&z\bigg(\frac{zk_l(z)}{k_l(z)+zk'_l(z)}+1\bigg)
%\\&=z\bigg(\frac{z\sqrt{\frac{2}{\pi z}}K_{l+1/2}(z)}{\sqrt{\frac{2}{\pi z}}K_{l+1/2}(z)+
%z\Big(\sqrt{\frac{2}{\pi z}}K'_{l+1/2}(z)-\frac{1}{2z}\sqrt{\frac{2}{\pi z}}K_{l+1/2}(z)\Big)}+1\bigg)
%\\&
=z\bigg(\frac{zK_{l+1/2}(z)}{\frac{1}{2}K_{l+1/2}(z)+zK'_{l+1/2}(z)}+1\bigg).
\end{split}
\end{equation}
We are interested in the poles of the ratio, i.e.,  zeros of $\frac{1}{2}K_{l+1/2}(z)+zK'_{l+1/2}(z)$. It is  indeed very
fortunate to find the following results in \cite[Lemmas 1-2]{tokita1972} (for more general combination of this form).
\begin{figure}[!ht]
	\subfigure{
		\begin{minipage}[t]{0.4\textwidth}
			\centering
			\rotatebox[origin=cc]{-0}{\includegraphics[width=1\textwidth,height=0.88\textwidth]{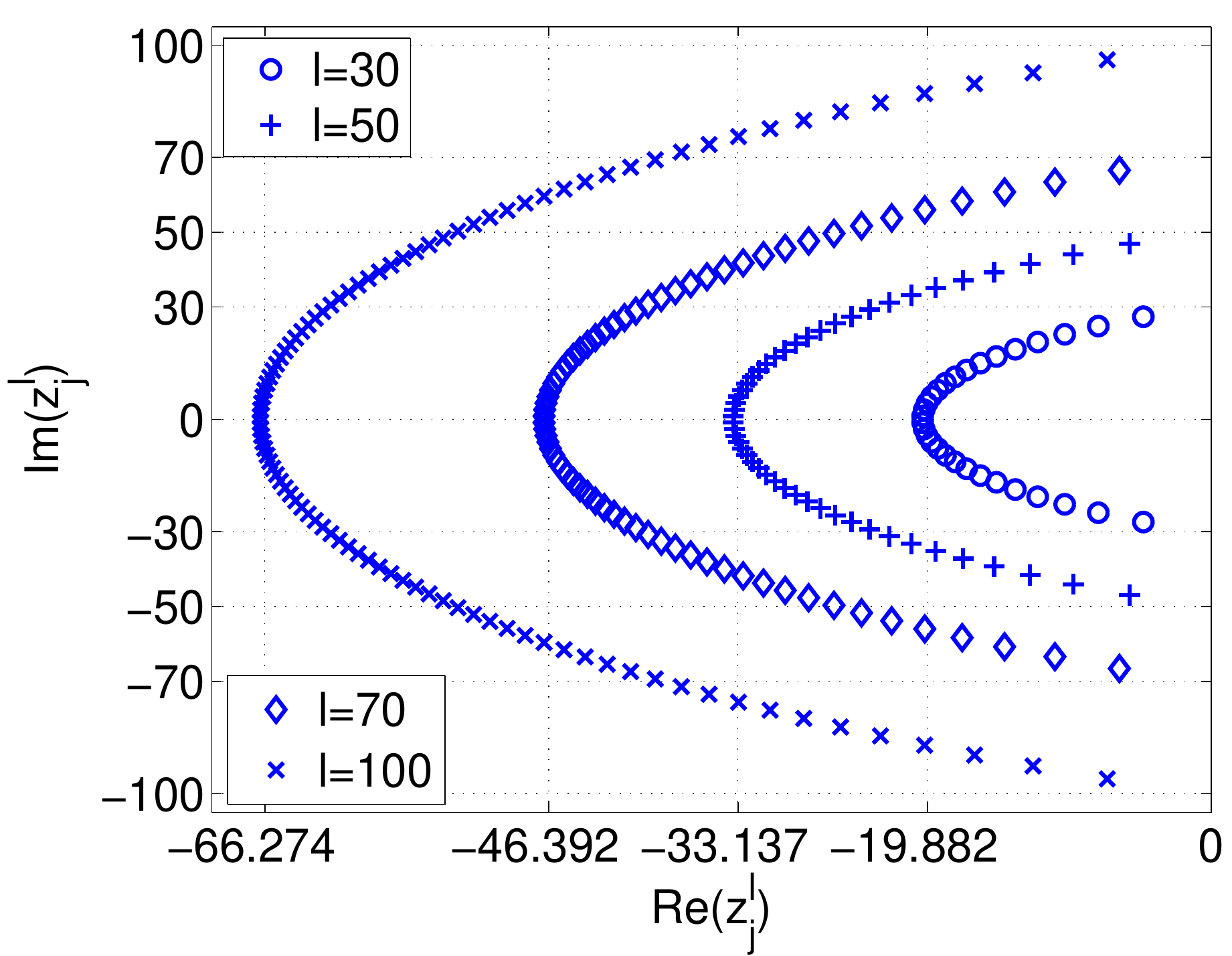}}
	\end{minipage}}\qquad \;\;
	\subfigure{
		\begin{minipage}[t]{0.4\textwidth}
			\centering
			\rotatebox[origin=cc]{-0}{\includegraphics[width=1\textwidth,height=0.88\textwidth]{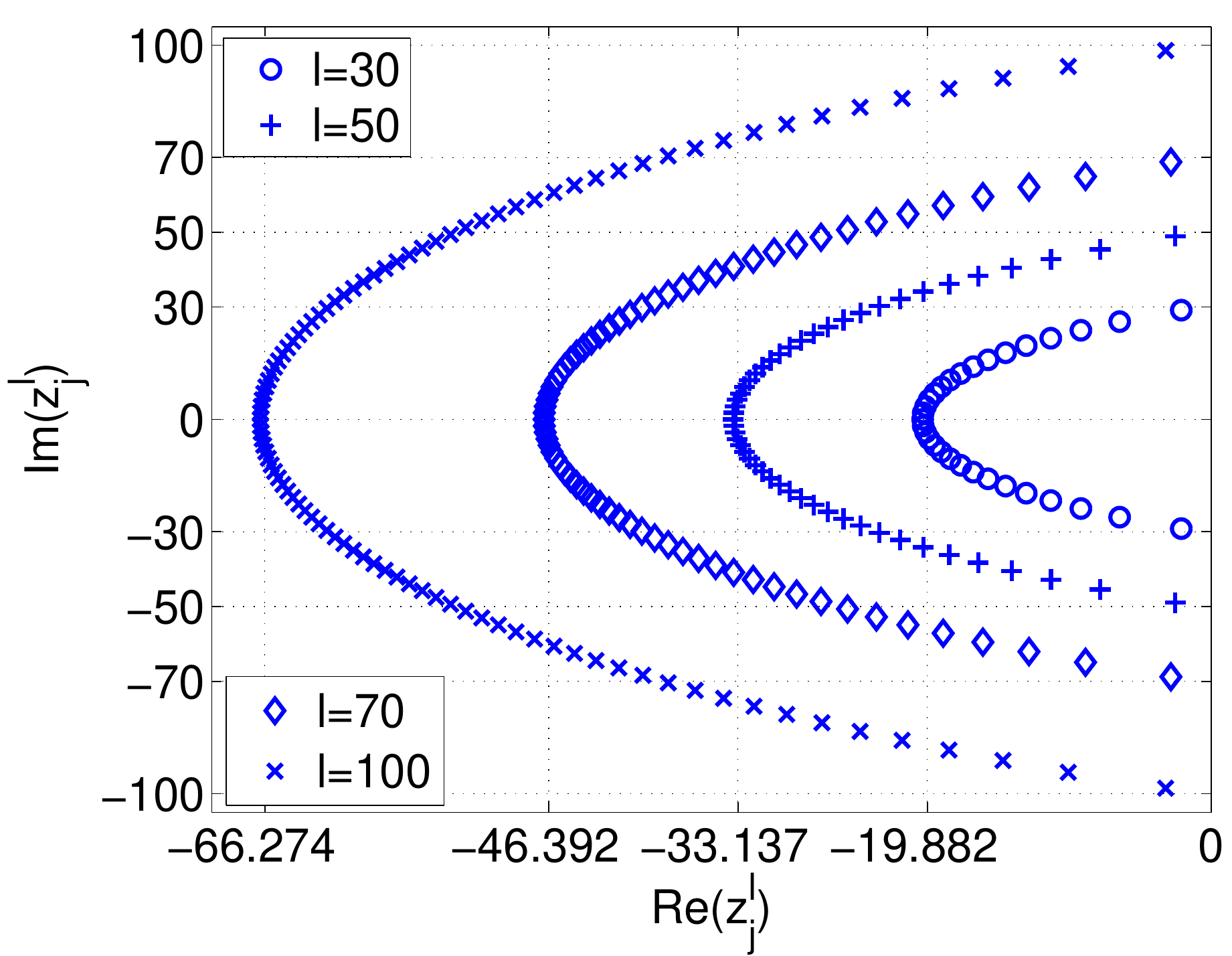}}
	\end{minipage}}
	\vspace*{-10pt} \caption{\small Distributions of the zeros of
		$K_{l+1/2}(z)$ (left) and $\frac{1}{2}K_{l+1/2}(z)+zK'_{l+1/2}(z)$ (right).}\label{Figzerodistr}
\end{figure}
%In order to use the residue theorem to evaluate $\rho_l(t)$ given in  \eqref{arhol}, %\eqref{defnf},
%it is essential to  understand the behavior of the poles, i.e., the zeros of  $\frac{1}{2}K_{l+1/2}(z)+zK'_{l+1/2}(z)$.
%The following properties can be found in .
\begin{lemma}\label{zerosprop}
	Let $l$ be a nonnegative    integer. Then we have the following properties.
	\begin{itemize}
		\item[{\rm (a)}] $\frac{1}{2}K_{l+1/2}(z)+zK'_{l+1/2}(z)$ has exactly  $l+1$ zeros.
		\item[{\rm (b)}] If $z_*$ is a zero of $\frac{1}{2}K_{l+1/2}(z)+zK'_{l+1/2}(z)$, then its complex conjugate $\bar z_*$ is also a zero. % As a result, $z=0$ is a zero if $l$ is even.
		\item[{\rm (c)}] All zeros of $\frac{1}{2}K_{l+1/2}(z)+zK'_{l+1/2}(z)$ are simple and have  negative real parts, so  they lie in
		the left half of the complex plane.
		% \item[{\rm (iii)}]  All zeros of $\frac{1}{2}K_{l+1/2}(z)+zK'_{l+1/2}(z)$ are
		%simple, and lie  approximately along the left half of the boundary
		%of an eye-shaped domain around $z=0$ {\rm(}see Figure
		%{\rm\ref{Figzerodistr}(right))}.
	\end{itemize}
\end{lemma}
%\begin{proof}
%The properties (i) and (ii) can be found in \cite[Lemma 1 and Lemma 2]{tokita1972}.
%\end{proof}
\begin{wrapfigure}{r}{0.41\textwidth}
	\center
	\vspace{-2em}
	\includegraphics[scale=0.63]{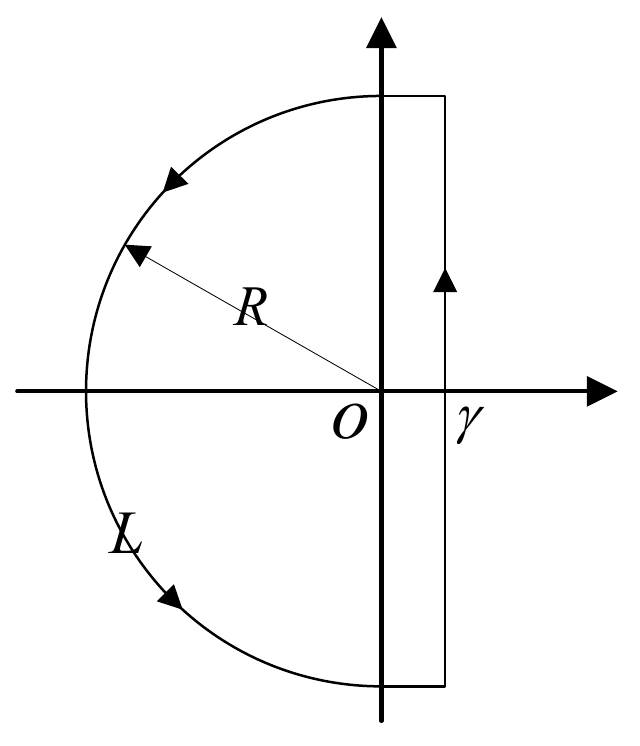}
	\vspace{-1em}
	\caption{\small \small  Contour $L$ for the inverse Laplace transform.}
	\vspace{-2em}
	\label{contourL}
\end{wrapfigure}
We illustrate  in Figure \ref{Figzerodistr} the distribution  of  zeros of
$K_{l+1/2}(z)$ (left) and $\frac{1}{2}K_{l+1/2}(z)+zK'_{l+1/2}(z)$ (right) for various $l.$ We observe that for
a given $l$, the zeros of $\frac{1}{2}K_{l+1/2}(z)+zK'_{l+1/2}(z)$ have a  distribution very similar to  those of $K_{l+1/2}(z)$, that is,  sitting on  the left half boundary of an
eye-shaped domain  that intersects the imaginary axis approximately
at $\pm  l\ri,$ and the negative real axis at $-la$ with $a\approx
0.66274$ (see the vertical  dashed  coordinate grids). Such a behaviour is very similar to that of $K_{l+1/2}(z)$ (cf.
\cite{Wang2Zhao12}).
% We refer to \cite{Sofronov1998,Hagstrom07,Wang2Zhao12} for more details for the behavors of  the zeros of $K_{l+1/2}(z)$, which are valid for  $\frac{1}{2}K_{l+1/2}(z)+zK'_{l+1/2}(z)$ in general.
% as predicted by  Lemma
%\ref{zerosprop} (iii).

%\begin{figure}[h!]
%	\centering
%	{~}\hspace*{-16pt}
%	\subfigure[]{ \includegraphics[scale=.63]{contour}}
%	\caption{\small \small  Contour $L$ for the inverse Laplace transform.}
%	\label{contourL}
%\end{figure}

\smallskip
With the above understanding, we now ready to  present the analytical formula for
the convolution kernel function $\rho_l(t)$.
\begin{theorem}\label{rholthm}
	Let $\{\tilde z_j^l\}_{j=1}^{l+1}$ be the zeros of $\frac{1}{2}K_{l+1/2}(z)+zK'_{l+1/2}(z)$ with integer $l\ge 0$. %, and $z=sb/c$.
	Then we can compute $\rho_l(t)$ in \eqref{arhol} via
	%\begin{equation}\label{kerneltrho}
	%\tilde\rho_{l}(t):=\mathcal L^{-1}\bigg[\frac{zk_l(z)}{k_l(z)+zk'_l(z)}+1\bigg](t)=\frac{c} b \sum_{j=1}^{l+1}\frac{\big(\tilde z_j^l\big)^2}{\big(\tilde z_j^l\big)^2+l(l+1)} e^{ct \tilde z_j^l /b},
	%\end{equation}
	%and
	\begin{equation}\label{kernelrho}
	\begin{split}
	\rho_{l}(t)&%=\mathcal L^{-1}\bigg[z\Big(\frac{zk_l(z)}{k_l(z)+zk'_l(z)}+1\Big)\bigg](t)=\frac{b}{c}\Big({\tilde\rho}'_{l}(t)+\tilde\rho_{l}(0)\delta(t)\Big)
	%\\&
	=\frac{c} b \sum_{j=1}^{l+1}\frac{(\tilde z_j^l)^3}{l(l+1)+(\tilde z_j^l)^2}\, e^{ct \tilde z_j^l /b}+
	\delta(t)\sum_{j=1}^{l+1}\frac{(\tilde z_j^l)^2}{l(l+1)+(\tilde z_j^l)^2},
	\end{split}
	\end{equation}
	where  $\delta(t)$ is the Dirac delta function.
\end{theorem}
\begin{proof}
 Using  the property
 \begin{equation}\label{invLapf}
 \mathscr L^{-1}[sf(s)](t)=f'(t)+f(0)\delta(t),
 \end{equation}
 we obtain from \eqref{arhol} and \eqref{defnf} that
	\begin{equation}\label{rhol}
	\rho_l(t)=\mathscr L^{-1}\bigg[z\,\bigg(\frac{z\, k_l(z)}{k_l(z)+z\, k_l'(z)}+1\bigg)\bigg](t)=\frac{b}{c}\big(\tilde\rho'_{l}(t)+\tilde\rho_{l}(0)\delta(t)\big),
	\end{equation}
	where $z={sb}/{c},$ and with a change of variable $s=cz/b,$ we have
	\begin{equation}\label{trhol}
	\begin{split}
	\tilde\rho_{l}(t)&=\frac{1}{2\pi \ri}\frac{c}{b}\int_{\gamma-\infty \ri}^{\gamma+\infty\ri}\bigg[\frac{zk_l(z)}{k_l(z)+zk'_l(z)}+1\bigg] e^{czt/b} dz
	:=\frac{1}{2\pi \ri}\frac{c}{b}\int_{\gamma-\infty \ri}^{\gamma+\infty\ri}F_{l}(z) e^{czt/b} dz.
	\end{split}
	\end{equation}
	In view of the formula
	(see \cite[10.49.12]{olver2010nist}):
	\begin{eqnarray}
	k_{l}(z)=\frac{\pi}{2}\sum\limits_{k=0}^l\frac{(l+k)!e^{-z}}{2^kk!(l-k)!z^{k+1}},\quad
 l\ge 0,
	\end{eqnarray}
	and Lemma \ref{zerosprop}, we conclude that  $F_l(z)$ is a meromorphic function.
	%	\begin{figure}[!ht]
	%		\centering
	%		\includegraphics[width=0.28\textwidth]{contour.pdf}
	%		\caption{\small Contour $L$ for the inverse Laplace tranform.} \label{contour}
	%	\end{figure}
	We introduce the closed contour $L$ as depicted in  Figure \ref{contourL}. Using the residue theorem and
	Jordan's Lemma, we have
	\begin{equation*}
	2\pi \ri \sum\limits_{j=1}^{l+1}{\rm Res}\,\big[F_{l}(z)
	e^{czt/b}; \tilde z_{j}^{l}\big]=\lim_{R\to+\infty}\oint_L F_{l}(z)e^{czt/b}dz
	=\int_{\gamma-\infty \ri}^{\gamma+\infty\ri}F_{l}(z) e^{czt/b} dz.
	\end{equation*}
From \eqref{akasympc}, we calculate that
	\begin{equation}\label{Fznp}
	\begin{split}
	\frac{b}{c}\tilde \rho_l(t)&= \sum\limits_{j=1}^{l+1}{\rm Res}\,\big[F_{l}(z)
	e^{czt/b}; \tilde z_{j}^{l}\big]\\
	&=\sum\limits_{j=1}^{l+1}\lim_{z\rightarrow\tilde z_{j}^{l}}
	\bigg\{\big(z-\tilde z_{j}^{l}\big) e^{ctz/b}\bigg[\frac{zK_{l+1/2}(z)}{\frac{1}{2}K_{l+1/2}(z)+zK'_{l+1/2}(z)}+1\bigg]\bigg\}
	\\&=\sum\limits_{j=1}^{l+1}
	\frac{e^{ct\tilde z_{j}^{l}/b}\tilde z_{j}^{l}K_{l+1/2}(\tilde z_{j}^{l})}
	{\frac{3}{2}K'_{l+1/2}(\tilde z_{j}^{l})+\tilde z_{j}^{l}K''_{l+1/2}(\tilde z_{j}^{l})}.
	\end{split}
	\end{equation}
	Since $K_{l+1/2}(z)$ satisfies  the equation (cf. \cite{watson})
	\[
	z^2\frac{d^2w}{dz^2}+z\frac{dw}{dz}-\Big(z^2+(l+1/2)^2\Big)w=0,
	\]
	we have
	\[
	\begin{split}
	z^2K''_{l+1/2}(z)+\frac{3}{2}zK'_{l+1/2}(z)&=\big(z^2+(l+1/2)^2\big)K_{l+1/2}(z)-zK'_{l+1/2}(z)+\frac{3}{2}zK'_{l+1/2}(z)
	\\&=\big(z^2+l(l+1)\big)K_{l+1/2}(z)+\frac{1}{2}\bigg(\frac{1}{2}K_{l+1/2}(z)+zK'_{l+1/2}(z)\bigg).
	\end{split}
	\]
	A combination of  \eqref{Fznp} and the fact that $\{\tilde z_j^l\}_{j=1}^{l+1}$ are zeros of $\frac{1}{2}K_{l+1/2}(z)+zK'_{l+1/2}(z)$ yields
	\begin{equation}\label{Fznp1}
	\begin{split}
	\tilde\rho_{l}(t)&%=\frac{1}{2\pi \ri }\frac{c}{b}\int_{\gamma-\infty \ri}^{\gamma+\infty
	%\ri}F_{l}(z) e^{czt/b} dz
	=\frac{c}{b}\sum\limits_{j=1}^{l+1}
	\frac{e^{ct\tilde z_{j}^{l}/b}\big(\tilde z_{j}^{l}\big)^2K_{l+1/2}(\tilde z_{j}^{l})}
	{\big[\big(\tilde z_{j}^{l}\big)^2+l(l+1)\big]K_{l+1/2}(\tilde z_{j}^{l})}
	=\frac{c}{b}\sum\limits_{j=1}^{l+1}
	\frac{e^{ct\tilde z_{j}^{l}/b}\big(\tilde z_{j}^{l}\big)^2}
	{\big(\tilde z_{j}^{l}\big)^2+l(l+1)}.
	\end{split}
	\end{equation}
	%By using the property of the inverse Laplace transform $\mathcal L^{-1}[sf(s)](t)=f'(t)+f(t)\delta(t)$ and \eqref{arhol}, we have
	%\[
	%\rho_l(t)=\frac{b}{c}\mathcal L^{-1}\big[sF_{l}(z)\big](t)=\frac{b}{c}\big[\tilde\rho'(t)+\tilde\rho(0)\delta(t)\big].
	%\]
	Inserting \eqref{Fznp1} into \eqref{rhol} leads to the expression of $\rho_l(t)$ in \eqref{kernelrho}.
\end{proof}

Having addressed the first issue on how to  compute  the convolution kernel functions, we now introduce an efficient technique to alleviate the historical burden of  temporal convolutions involved in the capacity operator \eqref{oldEtM}.
%\begin{rem}\label{rmk::tempconvolution}
	Observe from  \eqref{kernela} and \eqref{kernelrho} that  the time variable  $t$ only
	presents  in the complex exponentials. As a result,  the temporal convolutions can be evaluated recursively  as shown  in e.g., \cite{Alpert02,Wang2Zhao12}. More precisely, given a continuous function $g(t)$, we define
	\begin{equation}\label{ftzasy}
	f(t;z):=e^{ctz/b}\ast g(t)=\int_0^t e^{c(t-\tau)z/b}g(\tau)d\tau.
	\end{equation}
	Then by \eqref{kernela} and \eqref{kernelrho},  %we have  %\eqref{oldEtM}
	\begin{subequations}\label{recurdelta7}
		\begin{align}
		(\sigma_l\ast g)(t)& =\frac c {b} \sum_{j=1}^{l}z_j^l f(t;z_j^l);
		%\quad[\omega_l\ast g](t)=\frac c {b} \sum_{j=1}^{l}(z_j^l)^2 f(t;z_j^l)+g(t)\sum_{j=1}^lz_j^l,
		\label{recurdeltomg}\\
		(\rho_l\ast g)(t)&=\frac{c} b \sum_{j=1}^{l+1}\frac{\big(\tilde z_j^l\big)^3 f(t;\tilde z_j^l)}{\big(\tilde z_j^l\big)^2+l(l+1)}+g(t)\sum_{j=1}^{l+1}\frac{\big(\tilde z_j^l\big)^2}{\big(\tilde z_j^l\big)^2+l(l+1)}.\label{recurdeltrho}
		\end{align}
	\end{subequations}
	One verifies readily  that
	\begin{equation}\label{fastconvformula}
	f(t+\Delta t;z)=e^{c\Delta t\,z/b}f(t; z)+\int_t^{t+\Delta t}e^{c(t+\Delta t-\tau)z/b}g(\tau)\,d\tau,
	\end{equation}
	so $f(t; z)$ can march in $t$ with step size $\Delta t$ recursively.
	As a result, the temporal convolution in the NRBC can be computed efficiently with the  explicit expressions \eqref{kernela} and \eqref{kernelrho}, and with the above fast  recursive algorithm.
%\end{rem}

\smallskip
Next, we provide some numerical results to demonstrate the high accuracy in computing   $\rho_l(t)$ and the related convolution in  \eqref{kernelrho}. Let  $\phi(t), t\ge 0$ be a given differentiable function such that $\phi(0)=0.$ As shown in \cite{Wang2Zhao12},  $(\sigma_l\ast \phi)(t)$ can be computed very accurately (i.e., using Proposition \ref{sigmalthm} and \eqref{recurdeltomg}).  Let $\psi_l(t)$ be a function associated with $\phi(t)$ through
\begin{equation}\label{psiphi}
\mathscr{L}[\psi_l](s)=\frac{k_l(z)+zk'_l(z)}{k_l(z)}\mathscr{L}[\phi](s)
=\left(\!1+z+z\frac{k'_l(z)}{k_l(z)}\right)\!\mathscr{L}[\phi](s)-z\mathscr{L}[\phi](s),
\end{equation}
where $z=sb/c.$ Applying the inverse Laplace transform and  using the definition of $\sigma_l(t)$ in \eqref{arhol}, we obtain
from \eqref{invLapf}, $\phi(0)=0$ and Proposition  \ref{sigmalthm}  that
\begin{equation}\label{psi}
\psi_l(t)=(\sigma_l\ast \phi)(t)-\frac{b}{c}\phi'(t)=\frac{c} b\sum_{j=1}^l (z_j^l)^2 e^{ct z_j^l /b}\ast\phi(t)+\phi(t)\sum_{j=1}^l z_j^l.
%+\phi(0)\delta(t)\big].
\end{equation}

We next present two  ways to compute $(\rho_l\ast\psi_l)(t),$ where the first one only requires the use of the formula for $\sigma_l(t).$ Indeed,  we have from \eqref{arhol}, \eqref{invLapf} and \eqref{psiphi} that
\begin{equation}\label{identityA}
\begin{split}
(\rho_{l}\ast\psi_l)(t)&=\mathscr L^{-1}\bigg[z\Big(\frac{z\, k_l(z)}{k_l(z)+z\, k_l'(z)}+1\Big)\mathscr{L}[\psi_l]\bigg]
=\mathscr L^{-1}\big[z^2\mathscr{L}[\phi] +z\mathscr{L}[\psi_l]\big]\\
&=\mathscr L^{-1}\bigg[\Big(1+z+z\frac{k'_l(z)}{k_l(z)}\Big)z\mathscr{L}[\phi]\bigg]=\sigma_l\ast
\mathscr L^{-1}\big[z\mathscr{L}[\phi]\big]= (\sigma_l\ast \phi')(t).
\end{split}
\end{equation}
Then, by Proposition \ref{sigmalthm} and integration by parts, we find
\begin{equation}\label{conv2}
\begin{split}
(\rho_{l}\ast\psi_l)(t)
&= \sum_{j=1}^l z_j^l e^{ct z_j^l /b}\ast \phi'(t)
=\frac{c} b\sum_{j=1}^l (z_j^l)^2 e^{ct z_j^l /b}\ast\phi(t)+\phi(t)\sum_{j=1}^l z_j^l.
\end{split}
\end{equation}

On the other hand, using Theorem \ref{rholthm} and the relation \eqref{psi}  leads to
\begin{equation}\label{conv1}
\begin{split}
(\rho_{l}\ast\psi_l)(t)&=\frac{c} b \sum_{j=1}^{l+1}\frac{\big(\tilde z_j^l\big)^3}{\big(\tilde z_j^l\big)^2+l(l+1)} e^{ct \tilde z_j^l /b}\ast\psi_l(t)+
\psi_l(t)\sum_{j=1}^{l+1}\frac{\big(\tilde z_j^l\big)^2}{\big(\tilde z_j^l\big)^2+l(l+1)},
\end{split}
\end{equation}
where $\psi_l$ is computed by \eqref{psi}. It is evident that the convolutions in \eqref{conv2}-\eqref{conv1}    can be evaluated by using \eqref{fastconvformula}.
%and  Proposition  \ref{sigmalthm} as follows
%\begin{equation}
%\psi_l(t)=\frac{c} b \sum_{j=1}^{l}z_j^l e^{\frac{c}{b} z_j^lt}\ast\phi(t) -\frac{b}{c}\phi'(t).
%\end{equation}

%
%
%By the equality $z\breve\Psi(s)=z\big(1+z+zk'_l(z)/k_l(z)\big)\breve\Phi(s)-z^2\breve\Phi(s)$
%
%
%
%%we denote $\breve{\Phi}(s)=\mathscr L[\phi(t)]$, and define another function $\psi(t)$ such that
%%\[
%%\mathscr{L}[\psi(t)]=\breve\Psi(s):=\frac{k_l(z)+zk'_l(z)}{k_l(z)}\breve\Phi(s)
%%=\Big[1+z+z\frac{k'_l(z)}{k_l(z)}\Big]\breve\Phi(s)-z\breve\Phi(s).
%%\]
%Apparently, $\psi(t)$ in the physical domain can be obtained by taking inverse Laplace transform and
%We shall use two different ways to compute $\rho_{l}(t)\ast\psi(t)$ and compare the results to validate the formula \eqref{kernelrho}.  and the formula \eqref{kernela}, we arrive at
%
It is seen that \eqref{conv2} and \eqref{conv1} are equivalent, but the former solely involves $\sigma_l(t),$ which can used as a reference to check the accuracy of  $\rho_l(t).$   We take $b=3, c=5$ and $\phi(t)=\sin^6(8t)$,  so we can compute the two functions and convolutions with exponential functions  in  \eqref{conv2}-\eqref{conv1} exactly.
We tabulate  in Table \ref{tb1} the relative errors
% between \eqref{conv1}-\eqref{conv2}.
%We know that $f_l(t)$ and $\tilde f_l(t)$ should be equal.
\[
e_l(t)=\frac{|f_l(t)-\tilde f_l(t)|}{|\tilde f_l(t)|},
\]
where $f_l(t)$ and $\tilde f_l(t)$ denote the convolution computed by  \eqref{conv1} and \eqref{conv2}, respectively.
We can see that the relative errors are of machine accuracy,  which validate the formula \eqref{kernelrho}.
\begin{table}[!ht]
{\small
		\begin{center}
			\caption{\normalsize The relative error $e_l(t)$ for different $l$ and $t$.}\label{tb1}
			\begin{tabular}{|c|c|c|c|c|}
				\hline \multicolumn{1}{|c|}{$l$}&
				\multicolumn{1}{c|}{$t=1$}&
				\multicolumn{1}{c|}{$t=2$}&
				\multicolumn{1}{c|}{$t=4$}&
				\multicolumn{1}{c|}{$t=10$}
				\\
				\cline{1-5}
				1 &   3.0366e-015 & 1.9953e-016 & 2.6091e-015 & 2.5178e-015    \\
				\hline
				5   &    2.0593e-015 & 7.5062e-015 & 3.8667e-015 & 1.7176e-014 \\
				\hline
				10   &  8.2695e-015 & 1.7839e-014 & 3.5474e-014 & 1.5205e-016 \\
				\hline
				15  &   8.8691e-015 & 3.5764e-014 & 1.1748e-014 & 1.5469e-015 \\
				\hline
				30   &  4.4403e-015 & 3.0733e-015 & 8.0434e-015 & 2.8513e-015 \\
				\hline
				50  &  5.0412e-015 & 3.0602e-015 & 1.1924e-016 & 2.9790e-015\\
				\hline
			\end{tabular}
		\end{center}}
\end{table}

%\comm{This part needs some further revision! Or just briefly describe it!}

\begin{rem}\label{rmk::poleapprox}
	It is clear that the number of zeros to be used is determined by the truncation of the expansion \eqref{Eexpansion}.  If $l$ is large, the pole compression algorithm (cf. \cite{Alpert, jiang2008efficient}) can be adopted to obtain approximations for the NRBKs: $\sigma_l(t)$ and $\rho_l(t)$. The approximated kernels have the same form as in \eqref{kernela} and \eqref{kernelrho} while the number of poles in the summation has been reduced significantly.
\end{rem}

%\comm{\color{red} We might put this part to Section 3 where it is used in algorithm!}
\subsection{An alternative formulation of the capacity operator}\label{sectnewNRBC}
It is seen from \eqref{oldEtM} that the capacity operator $\mathscr{T}_b[\bs E^{\rm sc}]$ only involves the tangential component
%(in $ \bs L^2_T(S)$)
   of the vector field $\bs E^{\rm sc}\in {\mathbb H}_0({\rm div};\Omega).$   In fact, as shown in Proposition \ref{EFsinsol},
   the  VSH expansion coefficients for a divergence-free field satisfy some relation that allows us to derive the following alternative representation of the capacity operator in Theorem \ref{theo21}. We find that it has certain advantage in the application in the forthcoming section.
\begin{theorem}\label{newthmA0} The time-domain capacity operator ${\mathscr T}_b$ in Theorem \ref{theo21}  can also be formulated as
\begin{equation}\label{newEtM00}
{\mathscr T}_b[\bs E^{\rm sc}]=\frac c b \sum_{l=1}^\infty\sum_{|m|=0}^l\bigg\{
\frac{\omega_l\ast {\psi}_{lm}^{r}}{l(l+1)}\,{\bs \Psi}_l^m+\big(\sigma_l\ast {\psi}_{lm}^{(2)}\big)\,\vt\bigg\},
\end{equation}
where $\sigma_l$ is given in \eqref{arhol} and
\begin{equation}\label{kernelomega00}
\begin{split}
\omega_l(t)=\frac b c\big(\sigma_l'(t)+\sigma_l(0)\delta(t)\big)=\frac{c} b \sum_{j=1}^{l}\big(z_j^l\big)^2 e^{\frac{c}{b} z_j^lt}+\delta(t)\sum_{j=1}^{l}z_j^l.
\end{split}
\end{equation}
Here, the expression \eqref{newEtM00} involves two of the VSH expansion coefficients $\{{\psi}^{r}_{lm}, {\psi}^{(1)}_{lm}, {\psi}^{(2)}_{lm}\}$  of $\bs E^{\rm sc}$ in \eqref{vshexpscatfield}.
\end{theorem}
\begin{proof}
 In view of  Proposition \ref{EFsinsol},
we can reformulate  \eqref{divfreevshexp} as
%\begin{equation}\label{conseptequiv00}
%	{\bs u}= u_{00}\,\bs Y_0^0 +  \sum_{l=1}^\infty\sum_{|m|=0}^l \Big\{\frac {\beta_l} r \breve{v}_{l}^m  \,  \bs Y_l^m
%	+\hat \partial_r \breve{v}_{l}^m  \,\bs\Psi_l^m +u_{lm}\, \vt \Big\},
%	\end{equation}
	\begin{equation}\label{divfreevshexp00}
	\breve{\bs E}^{\rm e}=\breve u_{00}\bs Y_0^0+\sum_{l=1}^\infty\sum_{|m|=0}^l\Big\{\breve{u}_{lm} \bs\Phi_l^m+\frac {\beta_l} r \breve{v}_{lm}  \,  \bs Y_l^m
	+\hat \partial_r \breve{v}_{lm}  \,\bs\Psi_l^m \Big\},
	\end{equation}
	which is a solution of the exterior problem \eqref{tranformms}.
	Let $\big\{\breve \psi^r_{lm}, {\breve \psi}^{(1)}_{lm},{\breve \psi}^{(2)}_{lm}\big\}$  be the Laplace transforms of $\big\{{\psi^r_{lm}, \psi}^{(1)}_{lm},{\psi}^{(2)}_{lm}\big\}$, which is the VSH expansion coefficients of the scattering field $\bs E^{\rm sc}$ in \eqref{vshexpscatfield}.
        Note that $\breve{\bs E}^{\rm e}=\breve {\bs E}^{\rm sc}$ is the solution of the exterior problem \eqref{tranformms} %\eqref{exteriorsystem}
         with boundary data $\breve{\bs\lambda}=\breve{\bs E}^{\rm sc}\times\hat{\bs r}$ on the artificial boundary $r=b$. Then by \eqref{vshexpscatfield},  \eqref{solutionLap} and \eqref{divfreevshexp00}, we arrive at the VSH expansion coefficients of Laplace transformed scattering field $\breve{\bs E}^{\rm sc}:$
\begin{equation}\label{eq: psir1}
\breve{\psi}^{r}_{lm}=\frac{\beta_l}{r}\breve{v}_{lm}=\frac{\beta_l} r \frac{ k_l(sr/c)}{\hat{\partial}_rk_l(sb/c)}\breve{\lambda}_{lm}^{(2)},\quad\breve \psi^{(1)}_{lm}=\hat{\partial}_r\breve{v}_{lm} =\frac{\hat{\partial}_rk_l(sr/c)}{\hat{\partial}_rk_l(sb/c)}\breve{\lambda}_{lm}^{(2)},
\end{equation}
for $r\geq b$.  This implies
\begin{equation*}\label{identicoef}
\breve \psi^r_{lm}= \frac{\beta_l} r  \frac{ k_l(sr/c)}{\hat{\partial}_rk_l(sr/c)}  \breve \psi^{(1)}_{lm}
= \frac{\beta_l} r  \frac{ k_l(sr/c)}{ \frac s c k_l'(sr/c)+\frac 1  r  k_l(sr/c)}  \breve \psi^{(1)}_{lm},
\end{equation*}
so at $r=b,$ we have %Combine these two equations in \eqref{eq: psir1} together leads to the following relation:
\begin{equation}\label{conffs}
\frac{\breve \psi^r_{lm}}{\beta_l}=\frac{ k_l(z)}{k_l(z)+z\, k_l'(z)}
\breve \psi_{lm}^{(1)},\quad z=\frac{sb}{c}.
\end{equation}
%for all $|m|\leq l$, $l\geq 1$.
Multiplying both sides of \eqref{conffs} by $(1+z)k_l(z)+zk_l'(z)$ yields
\begin{equation}\label{zeqnation}
z\,\bigg(\frac{z\, k_l(z)}{k_l(z)+z\, k_l'(z)}+1\bigg) \breve \psi_{lm}^{(1)}=z\,\bigg(1+z+z
\frac{k_l'(z)}{k_l(z)}\bigg)\frac{\breve \psi^r_{lm}}{\beta_l}.
\end{equation}
In view of the definitions of $\rho_l$ and $\sigma_l$ in \eqref{arhol} and using the property \eqref{invLapf}, we take the  inverse Laplace transform on both sides of  \eqref{zeqnation} and find
\begin{equation}\label{kernelrelation}
\rho_l\ast\psi_{lm}^{(1)}=\frac{\omega_l\ast {\psi}_{lm}^{r}}{\beta_l},
\end{equation}
where  the kernel function %we have used the definitions of the kernel functions , the property \eqref{invLapf} and defined $\omega_l$ as
\begin{equation}\label{kernel}
\begin{split}
\omega_l(t)&={\mathscr L}^{-1}\bigg[z\,\bigg(1+z+z
\frac{k_l'(z)}{k_l(z)}\bigg)\bigg](t)=\frac b c\big(\sigma_l'(t)+\sigma_l(0)\delta(t)\big). % \quad l\ge 1,\;\; t\ge 0.
\end{split}
\end{equation}
%Here, the expression of the NRBK $\sigma_{l}(t)$ in \eqref{arhol} and the derivative property of Laplace transform is used to obtain the expressions for $\omega_l(t)$, $l\geq 1$.
Then, substituting \eqref{kernelrelation} into \eqref{oldEtM} leads to   the capacity operator $\mathscr{T}_b[\bs E^{sc}]$ in
\eqref{newEtM00}.

Finally, the last formula in \eqref{kernelomega00} can be obtained from  \eqref{sigmalthm} directly.
\end{proof}

%Both formulas given in \eqref{oldEtM} and \eqref{newEtM00} could be adopted in different scenarios for computation convenience. In the application considered in the next section, it is natural to use the formula \eqref{newEtMdivfree}.
%\comm{\color{red} Q-1. ~ Will the following part be used somewhere below?}

\section{Simulation of three-dimensional dispersive invisibility cloak} \label{3dcloak}
\setcounter{equation}{0}
%\setcounter{lemma}{0}
%\setcounter{theorem}{0}
%{ \color{red}
%\begin{itemize}
%\item This part should have a bigger picture, maybe not from cloak, but cloak as an example
%\item Provide models for general symmetric dispersive medium
%\item Try to connect with the problem statement part.
%\end{itemize}
%}
As already mentioned in the introductory section, the invisibility cloak is one of the most appealing examples in the field of  transformation optics \cite{pendry.2006, green2003anis}. In this section, we focus on the time-domain modelling and efficient simulation of the electromagentic invisibility cloak first proposed in  \cite{pendry.2006}.  Indeed, there exist very limited works in three-dimensional cloak simulations. Our contributions are twofold.  (i) We shall derive a new mathematical formulation of the time-domain dispersive cloak.  Different from the existing models  based on some mixed forms of both $\bs E$ and $\bs H$ (see, e.g.,  \cite{Hao2008fdtd, zhao2009full, li2012developing, Li2014wellpose}), the new formulation  only involves one unknown field $\bs D$ (see Theorem \ref{governingeqincloak}), where the seemingly  complicated temporal convolutions in the form of \eqref{ftzasy} can be evaluated efficiently as shown in \eqref{fastconvformula}. Moreover, the proposed governing equation in the cloaking layer with special dispersive media is valid for other geometries (e.g., the polygonal layer) other than the spherical shell.
%together with the
%and {\color{red} leads to a symmetric variational form}.
(ii) To simulate the time-domain spherical cloaking designed in \cite{pendry.2006},   we shall develop a very efficient VSH-spectral-element method for solving the reduced  problem truncated by the NRBC \eqref{nrbconscatteringfield} using  the alternative formulation of the  capacity operator in Theorem \ref{newthmA0}. The implementation of the new algorithm can run thousands of time steps in a few hours on a desktop with intel i7 CPU,  while the parallel implementation of the classic FDTD method running on a cluster with 100 processors and 220 GB memory takes 45 hours for 13000 time steps (cf. \cite[pp. 7307]{zhao2009full}). %\eqref{newEtMdivfree}.

\subsection{Dispersive modelling of 3D invisibility  cloaks}
\setcounter{equation}{0}
\setcounter{lmm}{0}
\setcounter{thm}{0}

%Typically, time-domain modelling of invisibility cloaks is  based on the Maxwell's equations % of Faraday's law and Ampere's law
%\begin{subequations}\label{orignsystem00}
%\begin{numcases}{}
%{\!\!\!}\partial_t \bs D(\bs r, t)-\nabla\wedge \bs H(\bs r,t)=\bs 0 \qquad  {\rm in}\;\; \Omega_{\rm cl},\;\; t>0,\\[3pt]
%{\!\!\!}{\partial_t  \bs B}(\bs r,t)+\nabla\wedge \bs E(\bs r,t)=\bs 0 \qquad \;  {\rm in}\;\; \Omega_{\rm cl},\;\; t>0,
%\end{numcases}
%\end{subequations}
%with two appropriate {\em constitutive relations}  that can close this system of twelve unknowns but  eight equations.

%\comm{In the figure, no $\Omega_{\rm cl}$?}

  The key to the design of invisibility cloak is to fill the cloaking layer, denoted by  $\Omega_{\rm cl}$ (see, e.g., $\Omega_1$ in  Figure \ref{cloakregion}),  with specially designed metamaterials, which can steer electromagnetic waves from penetrating into the enclosed region, and thereby render the interior ``invisible'' to the outside observer.  According to  the pioneering work by Pendry et al. \cite{pendry.2006},  the cloaking parameters disperse with frequency and therefore can only be fully effective at a single frequency. To investigate this interesting phenomena, it is necessary to simulate the full wave and consider  % model and simulate
the non-monochromatic waves passing through  such frequency-dependent materials in time domain.

%Transformation optics originated from the seminal works \cite{pendry.2006,Lea}  offers an effective to design such materials.
%Intensive  simulations and analysis have been devoted to the frequency domain (see, e.g., \cite{cummer2006full,ruan2007ideal, zhai2010finite, zhang2008rainbow, liu2011approximate,kohn2010cloaking}).
%However, only limited works are available for the time-domain simulations including the FDTD \cite{zhao2008full,zhao2009full,okada2012fdtd} and the FETD \cite{li2012developing,li2013adaptive}.  Indeed, the invisibility devices in frequency-domain design are usually for a single frequency or a narrow band, so it is advisable to consider the full-wave scenarios.

The central issue for time-domain modelling is to formulate  the {\em constitutive relations}.
%Like \cite{zhang2008rainbow,okada2012fdtd}, we introduce the   dispersive medium in $\Omega_{\rm cl}$ through the Drude model. Based on a quite different derivation, we obtain a new
%model with only an unknown vector field $\bs D$.
%To fix the idea, we restrict our attentions to the spherical cloak.
The material parameters of an ideal spherical cloak are given by (cf. \cite{pendry.2006}):
\begin{equation}\label{materials0}
{\bs \varepsilon}={\bs \mu}={\rm diag}\big(\varepsilon(r), \epsilon, \epsilon\big),\quad
\varepsilon(r)=\frac{R_2}{R_2-R_1}\Big(\frac{r-R_1}{r}\Big)^2,\quad \epsilon=\frac{R_2}{R_2-R_1},
\end{equation}
in the cloaking layer  $\Omega_{\rm cl}=\{R_1<r<R_2\}.$ %, where $R_1$ and $R_2$ are the inner and outer radii of the cloaking layer
Since $\varepsilon(r) \in [0,1)$, same as in the left-handed materials (LHMs), the material parameters are often mapped by dispersive medium models, e.g., Drude model, Lorentz model \cite{Hao2008fdtd}. Here, we map $\varepsilon(r)$ (to frequency $\omega$-dependent medium)  via the Drude dispersion model:
%as in \cite{zhang2008rainbow,okada2012fdtd}:
\begin{equation}\label{drudei}
\varepsilon_k(r,\omega)=1-\frac{\omega_{p,k}^2(r)}{\omega(\omega-\ri \gamma_k)}, \quad \omega_{p,k}(r)
=\sqrt{\omega_c(\omega_c-\ri \gamma_k)(1-\varepsilon(r))},
\end{equation}
for $k=1,2,$ leading  to  the dispersive media in $\Omega_{\rm cl}:$
\begin{equation}\label{DrudeE}
\hat{\bs \varepsilon}(r,\omega)={\rm diag} \big(\varepsilon_1(r, \omega), \epsilon,\epsilon \big),\quad  \hat{\bs \mu}(r,\omega)={\rm diag}\big(\varepsilon_2(r,\omega), \epsilon,\epsilon \big).
\end{equation}
In the above expressions,   $\omega$ is the wave frequency, $\{\omega_{p,k}\}$ are the plasma frequencies, $\{\gamma_k\}$ are damping terms called collision frequencies and $\omega_c>0$ is the operating frequency of the cloak. Indeed,  if $\omega=\omega_c,$  then \eqref{DrudeE} reduces to \eqref{materials0}. Although the ideal lossless case, i.e., $\gamma_1=\gamma_2=0$ has been adopted in \cite{zhao2008full, zhao2009full,li2012developing}, it is physically more reasonable to include the loss effect of the medium in the modeling. Hereafter, we assume that  $\gamma_1\neq 0, \gamma_2\neq 0$.
%\begin{figure}[!ht]
%	\centering
%	\includegraphics[width=0.8\linewidth]{cloakregion}
%	\caption{\small Left: A sketch of the cross section of a general 3D cloak; Right: A sketch of the cross section of a spherical cloak.}
%	\label{cloakregion}
%\end{figure}

Denote  by $\hat f$ the Fourier transform of a generic function $f(t)$, i.e.,
\begin{equation*}
\hat f(\omega)=\mathscr{F}[f(t)](\omega)=\int_{-\infty}^{+\infty}f(t)e^{-\ri \omega t}dt.
\end{equation*}
Let  ${\bs v}=(v_r,v_{\theta}, v_{\phi})^t$ be a generic vector field, where $v_r$, $v_{\theta}$ and $v_{\phi}$ are the components of
${\bs v}$ in the coordinate units  ${\bs e}_r$, ${\bs e}_{\theta}$ and ${\bs e}_{\phi}$, respectively.
%In what follows, we denote a generic vector field ${\bs v}=(v_r,v_{\theta}, v_{\phi})^t$, where $v_r$, $v_{\theta}$ and $v_{\phi}$ are the components of
%${\bs v}$ regard to ${\bs e}_r$, ${\bs e}_{\theta}$ and ${\bs e}_{\phi}$, respectively.

Given \eqref{DrudeE}, the constitutive relation in the cloaking layer $\Omega_{\rm cl}$ in the frequency domain reads
\begin{align}
&\widehat {\bs D}=\big(\widehat D_r,\widehat D_{\theta}, \widehat D_{\phi}\big)^t=\varepsilon_0 \hat{\bs \varepsilon}(r,\omega)\widehat {\bs E}=\varepsilon_0\big(\varepsilon_1(r,\omega) \widehat E_r,\epsilon \widehat E_{\theta}, \epsilon \widehat E_{\phi}\big)^t, \label{Const1}\\
&\widehat {\bs B}=\big(\widehat B_r,\widehat B_{\theta}, \widehat B_{\phi}\big)^t=\mu_0 \hat{\bs \mu}(r,\omega)\widehat {\bs H}=\mu_0\big(\varepsilon_2(r,\omega) \widehat H_r,\epsilon \widehat H_{\theta}, \epsilon \widehat H_{\phi}\big)^t. \label{Const2}
\end{align}
%\comm{\color{red} I moved forward this remark. Is this ok?}
\begin{rem}\label{Rmk:cond} {The time-domain constitutive equations extensively used in \cite{zhao2008full, zhao2009full,li2007error, Li2014wellpose} reads
		\begin{align}
		&\frac{\partial^2 D_r}{\partial t^2}+\gamma_1\frac{\partial D_r}{\partial t}=\varepsilon_0\Big(\frac{\partial^2 E_r}{\partial t^2}+\gamma_1\frac{\partial E_r}{\partial t}+\omega_{p,1}^2E_r\Big), \quad (D_{\theta}, D_{\phi})=\varepsilon_0 \epsilon (E_{\theta}, E_{\phi}), \label{AlterConst1}\\
		&\frac{\partial^2 B_r}{\partial t^2}+\gamma_2\frac{\partial B_r}{\partial t}=\mu_0\Big(\frac{\partial^2 H_r}{\partial t^2}+\gamma_2\frac{\partial H_r}{\partial t}+\omega_{p,2}^2H_r\Big), \quad (B_{\theta}, B_{\phi})=\mu_0 \epsilon (H_{\theta}, H_{\phi}), \label{AlterConst2}
		\end{align}
		which can be obtained by simply applying the inverse Fourier transform to \eqref{Const1}-\eqref{Const2}.
		\qed}
\end{rem}

Different from the existing models, we use \eqref{Const1}-\eqref{Const2} to  derive the following relations in  time domain.
%where  $\widehat f$   is  the Fourier transform of $f$ given by
%\begin{equation*}
%\widehat f(\omega)=\mathscr{F}[f(t)](\omega)=\int_{-\infty}^{\infty}f(t)e^{-\ri \omega t}dt.
%\end{equation*}
%is the Fourier transform of $f(t)$.

\begin{lemma}\label{prop1} We have the constitutive relations in time domain of the form
	\begin{align}
	&{\bs E}=\varepsilon_0^{-1}\mathscr{D}_1[{\bs D}],\quad {\bs H}=\mu_0^{-1}\mathscr{D}_2[{\bs B}], \label{Constfinal}
	\end{align}
	where for $k=1,2,$ the operators
	\begin{equation}\label{DispOperator1}
	\mathscr{D}_k[\bs D]:=\Big(D_r+\int_{0}^{t} \vartheta_k(r, t-\tau)D_r(\cdot,\tau)d\tau,\, \epsilon^{-1}D_{\theta},\, \epsilon^{-1}D_{\phi}\Big)^t,
	\end{equation}
	with kernel functions given by
	\begin{equation}\label{kernD}
	\vartheta_k(r, t)=\frac{\ri\,\omega_{p,k}^2(r)}{\zeta^0_k(r)-\zeta^1_k(r)}\big(e^{\ri \zeta^0_k(r) t}-e^{\ri \zeta^1_k(r) t}  \big).
	\end{equation}
	Here,  $\{\zeta^0_k(r), \zeta^1_k(r)\}_{k=1}^2$ are the roots of  the quadratic equation: $z^2-\ri\gamma_kz-\omega_{p,k}^2=0$ given by
	\begin{equation}\label{solformu}
	\begin{split}
	& \zeta_k^0(r)=-\frac 1{\sqrt 2} \sqrt{\sqrt{\xi_k^2+\eta_k^2} + {\xi_k} }+\ri \bigg(\frac{\gamma_k} 2 +\frac 1{\sqrt 2} \sqrt{  \sqrt{\xi_k^2+\eta_k^2} -{\xi_k} }\bigg),\\
	& \zeta_k^1(r)=\frac 1{\sqrt 2} \sqrt{\sqrt{\xi_k^2+\eta_k^2} + {\xi_k} }+\ri \bigg(\frac{\gamma_k} 2 -\frac 1{\sqrt 2} \sqrt{  \sqrt{\xi_k^2+\eta_k^2} -{\xi_k} }\bigg),
	\end{split}
	\end{equation}
	where
	\begin{equation}\label{notereal}
	\xi_k= \omega_c^2\,(1-\varepsilon(r))-\frac {\gamma^2_k} 4,\quad \eta_k=-\gamma_k\omega_c\,(1-\varepsilon(r)),
	\end{equation}
	are the real and imaginary parts of $\omega_{p,k}^2-\frac {\gamma^2_k} 4$  for $k=1,2.$
\end{lemma}
\begin{proof}
	Using  the definition of $\varepsilon_k(r,\omega)$ in \eqref{DrudeE}, we derive from \eqref{Const1}-\eqref{Const2}  that
	\begin{align}
	& \widehat E_r=\varepsilon_0^{-1} \Big(1+\frac{\omega_{p,1}^2(r)}{\omega^2-\ri\gamma_1\omega-\omega_{p,1}^2(r) }\Big) \widehat D_r,\quad (\widehat E_{\theta},\widehat E_{\phi})^t=(\varepsilon_0 \epsilon)^{-1}(\widehat D_{\theta},\widehat D_{\phi})^t\;, \label{Const3}\\
	&\widehat H_r=\mu_0^{-1}\Big(1+\frac{\omega_{p,2}^2(r)}{ \omega^2-\ri\gamma_2\omega-\omega_{p,2}^2(r) } \Big)\widehat B_r ,\quad (\widehat H_{\theta},\widehat H_{\phi})^t=(\mu_0 \epsilon)^{-1}(\widehat B_{\theta},\widehat B_{\phi})^t.\;  \label{Const4}
	\end{align}
	%Transforming  back to time domain, we obtain the time domain constitutive relations:
	Applying the inverse Fourier transform to \eqref{Const3}-\eqref{Const4} leads to
	\begin{equation}\label{InvforA}
	\begin{split}
	&E_r=\varepsilon_0^{-1}D_r+\varepsilon_0^{-1}\omega_{p,1}^2(r)\mathscr{F}^{-1}\Big[\frac{1}{\omega^2-\ri\gamma_1\omega-\omega_{p,1}^2(r)}\Big]\ast D_r,\;\; %(E_{\theta}, E_{\phi})^t=(\varepsilon_0 \epsilon)^{-1}(D_{\theta}, D_{\phi})^t,
	\\
	&H_r=\mu_0^{-1}B_r+\mu_0^{-1}\omega_{p,2}^2(r)\mathscr{F}^{-1}\Big[\frac{1}{\omega^2-\ri\gamma_2\omega-\omega_{p,2}^2(r)}\Big]\ast B_r, \\[3pt]
	&(E_{\theta}, E_{\phi})^t=(\varepsilon_0 \epsilon)^{-1}(D_{\theta}, D_{\phi})^t,\quad  (H_{\theta}, H_{\phi})^t=(\mu_0 \epsilon)^{-1}(B_{\theta}, B_{\phi})^t,
	\end{split}
	\end{equation}
	where ``\ $\ast$\," is  the usual convolution as before.
	%	involved convolution is defined as
	%	$$\big(f*g\big)(t):=\int_{-\infty}^{\infty}f(t-\tau)g(\tau)d\tau.$$

	The rest of the derivation is to explicitly evaluate two inverse Fourier transforms.   Let $\zeta^0_k, \zeta^1_k$ be two roots of  $z^2-\ri\gamma_kz-\omega_{p,k}^2=0.$ Then
	we immediately have $\zeta_k^0+\zeta_k^1=\ri \gamma_k$ and $\zeta_k^0\zeta_k^1=-\omega_{p,k}^2,$ so   we can write
	\begin{equation}\label{splA}
	\frac{1}{\omega^2-\ri\gamma_k\omega-\omega_{p,k}^2(r)}=\frac{1}{\zeta_k^0-\zeta_k^1}\Big(\frac{1}{\omega-\zeta^0_k}-\frac{1}{\omega-\zeta^1_k}\Big).
	\end{equation}
	Recall  that (cf. \cite{arfken1999mathematical}):
	\begin{equation}
	\label{fourierformula}
	\mathscr{F}^{-1}\Big[\frac{1}{\ri\omega+a}\Big]=-\ri\,	\mathscr{F}^{-1}\Big[\frac{1}{\omega-a\ri}\Big]=e^{-at}H(t), \quad \hbox{if}\;\;\mathfrak{R}\{a\}>0,
	\end{equation}
	where $H(t)$ is the Heaviside function.
	Suppose that we can show
	\begin{equation}\label{InvFormula0}	
	\mathfrak{Im}\{\zeta^0_k\}>0,\quad \mathfrak{Im}\{\zeta^1_k\}>0.
	\end{equation}  	
	Then by \eqref{splA}-\eqref{fourierformula},
	\begin{equation}\label{InvFormula1}
	\mathscr{F}^{-1}\Big[\frac{1}{\omega^2-\ri\gamma_k\omega-\omega_{p,k}^2(r)}\Big]=\frac{\ri}{\zeta^0_k-\zeta^1_k}\big(e^{\ri \zeta^0_k t}-e^{\ri \zeta^1_k t}   \big)H(t).
	\end{equation}
	Consequently, we derive \eqref{DispOperator1}-\eqref{kernD} from  \eqref{InvforA} and \eqref{InvFormula1}.
	
	It remains to verify  \eqref{solformu} and  \eqref{InvFormula0}. It is evident that the quadratic equation has the roots:
	\begin{equation}\label{exactsolu}
	z=\frac{\gamma_k} 2 \ri  \pm \sqrt{\omega_{p,k}^2-\frac{\gamma_k^2} 4}=\frac{\gamma_k} 2 \ri\pm \sqrt{\xi_k+ \ri  \eta_k}\,.
	\end{equation}
	Setting  $\alpha_k+\ri \beta_k=\sqrt{\xi_k+ \ri  \eta_k}$,   we find
	$\alpha_k^2-\beta_k^2=\xi_k$ and $2\alpha_k\beta_k=\eta_k.$ Solving this system yields
	\begin{equation}\label{alphabetak}
	\alpha^2_k=\frac{\sqrt{\xi_k^2+\eta_k^2}+\xi_k}{2},\quad 	\beta^2_k=\frac{\sqrt{\xi_k^2+\eta_k^2}-\xi_k}{2}.
	\end{equation}
	Noting that $\alpha_k\beta_k<0, $ we can determine $\alpha_k,\beta_k,$ and  obtain \eqref{solformu} from
	\eqref{exactsolu}.   By  \eqref{solformu},  $\mathfrak{Im}\{\zeta^0_k\}>\mathfrak{Im}\{\zeta^1_k\},$ so  we next show that
	$\mathfrak{Im}\{\zeta^1_k\}>0,$ that is,
	$$
	\frac{\gamma_k} 2 >\frac 1{\sqrt 2} \sqrt{  \sqrt{\xi_k^2+\eta_k^2} -{\xi_k} }\;\;\;  {\rm i.e.,}\;\;\;   \gamma_k^4+4\gamma_k^2 \xi_k-4\eta_k^2>0.
	$$
	Direct calculation from \eqref{notereal} leads to
	$$
	\gamma_k^4+4\gamma_k^2 \xi_k-4\eta_k^2=4\gamma_k^2 \omega_c^2\, \varepsilon(r)\, (1-\varepsilon(r))>0,
	$$
	as $\gamma_k\neq 0$, $\omega_c>0$ and $0<\varepsilon(r)<1$ (cf. \eqref{materials0}).  This verifies \eqref{InvFormula0}	 and completes the proof.
\end{proof}

%Different from \eqref{AlterConst1}-\eqref{AlterConst2},
 With the constitutive relations \eqref{Constfinal}-\eqref{DispOperator1} at our disposal, we represent $\bs E,\bs H$ in terms of $\bs D, \bs B$ and then eliminate $\bs B$, leading to  the following equation in $\Omega_{\rm cl}.$
\begin{theorem}\label{governingeqincloak}
	Assume that the source term and initial fields vanish in the cloaking layer $\Omega_{\rm cl}$. Then the governing equation in the cloaking layer  takes the form
	\begin{equation}\label{secondordergovern}
	% \frac{\partial^2 {\bs D}}{\partial t^2}
	\partial^2_{t}{\bs D}+c^2 \nabla \times \big( \mathscr{D}_2[ \nabla \times (\mathscr{D}_1[{\bs D}])] \big)=\bs 0\quad {\rm in}\;\; \Omega_{\rm cl}.
	\end{equation}
	%	 where $c=1/{\sqrt{\varepsilon_0\mu_0}}.$
\end{theorem}
\begin{proof}
	First, we show that given the homogeneous initial condition $\bs B(\bs r, 0)=\bs 0$, we have  $\partial_t \mathscr{D}_2[\bs B]=\mathscr{D}_2[\partial_t \bs B]$, that is, the operators $\partial_t$ and $\mathscr{D}_2$ are commutable.
	Indeed, by
	\begin{align*}
	\int_{0}^{t} \frac{\partial}{\partial t}\vartheta_2(r,t-\tau)  B_r(\bs r,\tau)d\tau
	=&-\vartheta_2( r,t-\tau)B_r(\bs r,\tau)\big|_0^t+\int_{0}^{t} \vartheta_2(r,t-\tau) \frac{\partial B_r(\bs r,\tau)}{\partial \tau}  d\tau\\
	=&\int_{0}^{t} \vartheta_2(r,t-\tau) \frac{\partial B_r(\bs r,\tau)}{\partial \tau}  d\tau,
	\end{align*}
	and \eqref{Constfinal}-\eqref{DispOperator1}, we verify that $\partial_t \mathscr{D}_2[\bs B]=\mathscr{D}_2[\partial_t \bs B]$.
	Thus, taking time derivative on both sides of the second equation in \eqref{Constfinal}, we obtain
	\begin{equation}\label{pHpt}
	\partial_t {\bs H}=\mu_0^{-1} {\partial_t(\mathscr{D}_2[\bs B])}=\mu_0^{-1} {\mathscr{D}_2[\partial_t \bs B]}.
	\end{equation}
	By substituting the constitutive relation \eqref{Constfinal} in the second equation in \eqref{reducedsystem1eq1}, we derive
	\begin{equation}\label{B2D}
	{\partial_t {\bs B}}=-\varepsilon_0^{-1}\nabla \times \big(\mathscr{D}_1[{\bs D}]\big).
	\end{equation}
	Then, taking time derivative on the first equation in \eqref{reducedsystem1eq1} and utilizing \eqref{pHpt}-\eqref{B2D} to eliminate ${\bs H}$ leads to \eqref{secondordergovern}, which ends the proof.
\end{proof}

\begin{rem}
	It is worthwhile to note that the mathematical model \eqref{sphericalcloakmodel} is not limited to spherical  dispersive cloaks. It is applicable to the modelling of many electromagnetic devices with symmetric non-diagonal $\bs \varepsilon$ and $\bs \mu$ made from metamaterials.
Following the procedure in \cite{okada2012fdtd}, we start with diagonalising the symmetric matrices $\bs\varepsilon$ and $\bs \mu,$ i.e.,
	\begin{equation}\label{constgeneral}
	\bs{\varepsilon}={\bs P} {\bs \Lambda_1} {\bs P}^t,\quad \bs \mu={\bs Q} {\bs \Lambda_2} {\bs Q}^t,\quad   {\bs \Lambda_i}={\rm diag}(\lambda_{i1}, \lambda_{i2},\lambda_{i3}),\quad i=1, 2,
	\end{equation}
	and $\{\bs P, \bs Q  \}=\{ P_{ij}, Q_{ij} \}_{1\leq i,j \leq 3} $ are orthonormal matrices. Then, we use the Drude model to map $\{\lambda_{ij}(\bs r)\}$ less than $1$  to $\{ \lambda_{ij}(\bs r, \omega) \}$ similar with \eqref{drudei} and take inverse Fourier transform to \eqref{constgeneral} with replaced $\{ \lambda_{ij}(\bs r, \omega) \}.$ As a result, we obtain the same constitutive relations  as  \eqref{Constfinal}
	$${\bs E}=\varepsilon_0^{-1}\mathscr{D}_1[{\bs D}],\quad {\bs H}=\mu_0^{-1}\mathscr{D}_2[{\bs B}]$$
	with more complicated forms of $\mathscr D_1$ and $\mathscr D_2$ :
	\begin{equation}\label{DispOperator1nondiagonal}
	\begin{split}
	\mathscr{D}_1[\bs D]:={\bs P} \widetilde{\bs \Lambda}^{-1}_1 {\bs P}^t \bs D+\int_0^t {\bs P} {\bs \Theta}_1(\bs r, t-\tau){\bs P}^t \bs D(\bs r, \tau)d\tau,\\
	\mathscr{D}_2[\bs B]:={\bs Q} \widetilde{{\bs \Lambda}}^{-1}_2 {\bs Q}^t \bs B+\int_0^t {\bs Q} {\bs \Theta}_2(\bs r, t-\tau) {\bs Q}^t \bs B(\bs r, \tau)d\tau,
	\end{split}
	\end{equation}
	where $\widetilde{\bs \Lambda}_i={\rm diag}\big(\tilde{\lambda}_{i1}(\bs r),\tilde{\lambda}_{i2}(\bs r),\tilde{\lambda}_{i3}(\bs r) \big)$, $\bs \Theta_i={\rm diag}(\vartheta_{i1}, \vartheta_{i2}, \vartheta_{i3})$ are diagonal matrices with
	\begin{equation*}
	\begin{split}
	& \tilde\lambda_{ij}=\begin{cases}
	\displaystyle 1& {\rm if}\;\; \lambda_{ij}(\bs r)\in (0, 1),\\
	\displaystyle \lambda_{ij}(\bs r) & {\rm if}\;\; \lambda_{ij}(\bs r)\in [1,\infty),
	\end{cases}\quad
	 \vartheta_{ij}=\begin{cases}
	\displaystyle \frac{\ri(\omega_{p,i}^j(\bs r))^2\big(e^{\ri \zeta^0_{ij} t}-e^{\ri \zeta^1_{ij} t}   \big)}{\zeta^0_{ij}-\zeta^1_{ij}}& {\rm if}\;\; \lambda_{ij}(\bs r)\in(0, 1),\\
	\displaystyle 0 & {\rm if}\;\; \lambda_{ij}(\bs r)\in [1,\infty),
	\end{cases}
	\end{split}
	\end{equation*}
	$\omega_{p,i}^j(\bs r)$ has a similar expression
	$$\omega_{p,i}^j(\bs r)	=\sqrt{\omega_c(\omega_c-\ri \gamma_i)(1-\lambda_{ij}(\bs r))},\quad i=1, 2,$$
	and complex pairs $\{\zeta^0_{ij}, \zeta^1_{ij}\}$ are the roots of quadratic equations $\omega^2-\ri\gamma_i\omega-(\omega_{p,i}^j(\bs r))^2=0$, $i=1, 2$, respectively.
	\qed
\end{rem}

\begin{wrapfigure}{r}{0.5\textwidth}
	\center
	\vspace{-2em}
	\includegraphics[scale=.3]{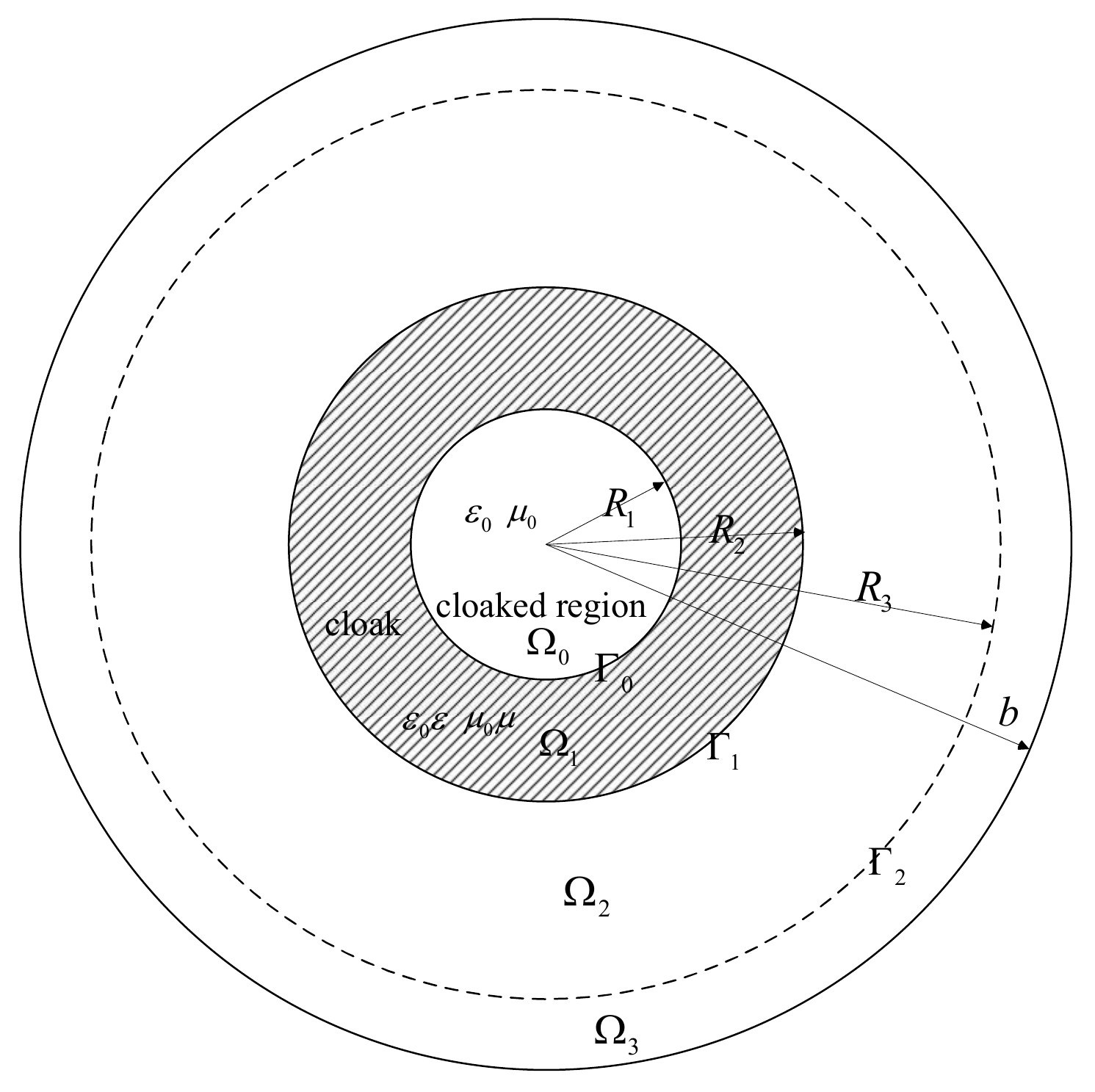}
	\vspace{-1em}
	\caption{\small \small Sketch of the cross section.}
	\label{cloakregion}
\end{wrapfigure}
\subsection{Simulation of the spherical invisibility cloaks}

In what follows, we focus on the  simulation of the spherical cloaks. We first present the full model with reduction of the unbounded domain by using the NRBC in Section \ref{sect::tbc}.  As sketched in Figure  \ref{cloakregion}, we denote
\begin{equation*}
\begin{split}
R_0=0;\;\; \Omega_i=\{R_i<r<R_{i+1}\},\;\; i=0,1; \\
\Omega_2=\Omega_{b_0}\setminus(\Omega_0\cup\Omega_1), \;\;  \Omega_3=\Omega_b\setminus \Omega_{b_0}.
\end{split}
\end{equation*}
Correspondingly, we further denote
$$\bs F(\bs r, t)=\partial_t\bs J(\bs r, t),$$
and
\begin{equation}
\begin{split}
&\Gamma_i=\bar \Omega_i \cap \bar \Omega_{i+1},\quad\{   \bs E^i, \bs H^i, \bs D^i, \bs B^i,\bs F^i \}=\{ \bs E, \bs H, \bs D,\bs B,\bs F  \}|_{\Omega_i}, \;\;\;i=0,1,2;\\[2pt]
&\{ \bs E^3, \bs H^3, \bs B^3, \bs D^3, \bs F^3\}=\{  \bs E^{\rm sc} , \bs H^{\rm sc}, \bs B^{\rm sc}, \bs D^{\rm sc}, \bs F  \}|_{\Omega_3}.
\end{split}
\end{equation}
We summarise below the  assumptions (for usual scattering problems):
\begin{itemize}
	\item[(i)] ${\bs \varepsilon}={\bs \mu}={\bs I}_3 \;\;{\rm in}\;\; \Omega_b \setminus \Omega_1$;
	\item[(ii)] There is no wave in the truncated domain $\Omega_b$ at time $t=0$, that is, we shall have homogeneous initial condition;
	\item[(iii)] The source term $\bs J$ is compactly supported in $\Omega_2.$
\end{itemize}

%\begin{figure}[h!]
%	\centering
%	{~}\hspace*{-16pt}
%	%\subfigure[]{ \includegraphics[scale=.63]{contour}}\qquad\qquad
%	\subfigure[]{\includegraphics[scale=.3]{cloakregion}}
%	\caption{\small \small Sketch of the cross section of a spherical cloak.}
%	\label{cloakregion}
%\end{figure}

\begin{proposition}
	The full  model for 3D cloak takes the form
	\begin{subequations}\label{sphericalcloakmodel}
		\begin{numcases}{}
\partial_{t}^2 {\bs D}^i+c^2 \nabla \times  \nabla \times\bs D^i=\bs F^i \quad{\rm in}\;\; \Omega_i,\;\;  i=0, 2, 3, \label{HomoEq0}\\[4pt]
		\partial_{t}^2  {\bs D}^1+c^2 \nabla \times \big( \mathscr{D}_2[ \nabla \times (\mathscr{D}_1[{\bs D}^1])] \big)=\bs 0\quad {\rm in} \;\; \Omega_1,  \label{CloakingEq}\\[4pt]
		({\bs D^{0} -\mathscr{D}_1[{\bs D}^{1}}])\times  \hat {\bs r}=\bs 0,\;\;\big(\nabla\times\bs D^{0} -\mathscr{D}_2\big[\nabla\times\mathscr{D}_1[{\bs D}^{1}]\big]\big)\times  \hat {\bs r}=\bs 0\;\; {\rm on}\;\; \Gamma_0, \label{Interface}\\[4pt]
		(\mathscr{D}_1[{\bs D}^{1}]-\bs D^{2})\times  \hat {\bs r}=\bs 0,\;\;\big(\mathscr{D}_2\big[\nabla\times\mathscr{D}_1[{\bs D}^{1}]\big]-\nabla\times\bs D^{2}\big)\times  \hat {\bs r}=\bs 0\;\; {\rm on}\;\; \Gamma_1, \label{Interface1}\\[4pt]
		({\bs D}^{2}-\bs D^{3})\times \hat {\bs r}=\bs D^{\rm in}\times \hat {\bs r},\;\;\nabla\times({\bs D}^{2}-\bs D^{3})\times  \hat {\bs r}=\nabla\times\bs D^{\rm in}\times \hat {\bs r}\;\; {\rm on}\;\; \Gamma_2, \label{Interface2}\\[4pt]
		\partial_t {\bs D}_T^3+c\big( \nabla \times {\bs D}^3  \big)\times \hat{\bs r}-\mathscr{T}_b[{\bs D}^3]=\bs 0\;\;\; {\rm at}\;\; r=b, \label{DtN}\\[4pt]
		{\bs D}(\bs r, 0)=\bs 0,\quad \partial_t {\bs D}(\bs r, 0)=\bs 0\;\;\; {\rm in}\;\; \Omega_b, \label{InitialEq}
		\end{numcases}
	\end{subequations}
	where $\bs D_T^{3}:= \hat {\bs r}\times  \bs D^{3}\times \hat{ \bs r}$
	is the tangential component of  $\bs D^{3}$ on the boundary $r=b$.
\end{proposition}
\begin{proof}
	Note that \eqref{HomoEq0} is a direct consequence of \eqref{reducedsystem1eq1}, \eqref{CloakingEq} is proved in Theorem \ref{governingeqincloak}, and \eqref{Interface2}-\eqref{InitialEq} are direct consequences  of \eqref{transcond1}-\eqref{intialEH} and the above assumption (i).  For the jump conditions  \eqref{Interface}-\eqref{Interface1}, we recall the standard transmission conditions
	\begin{equation}\label{continuEH}
	\bs E^i\times \hat{ \bs r}=\bs E^{i+1}\times \hat{ \bs r},\quad \bs H^i\times \hat{ \bs r}=\bs H^{i+1}\times\hat{ \bs r}\quad{\rm at}\;\;\; \Gamma_i, \quad i=0, 1.
	\end{equation}
	The first jump conditions in \eqref{Interface}-\eqref{Interface1} can be obtained by directly applying  the first constitutive relation between $\bs E$ and $\bs D$ to the above transmission condition on $\bs E$. Therefore, we focus on the first jump conditions in \eqref{Interface}-\eqref{Interface1}. Inserting the constitutive relations \eqref{Constfinal} into \eqref{continuEH} directly leads to \eqref{Interface} and %the fact
	%Recall the constitutive relations \eqref{Constfinal}, the tangential continuity of $\bs H$ implies that
	\begin{equation}
	\bs B^0\times \hat{ \bs r}=\mathscr{D}_2[\bs B^1]\times \hat{ \bs r} \quad{\rm at}\;\;\;\Gamma_1;\quad \bs B^2\times\hat{ \bs r}=\mathscr{D}_2[\bs B^1]\times\hat{ \bs r} \quad{\rm at}\;\;\; \Gamma_2.
	\end{equation}
From \eqref{pHpt}, we derive
	\begin{equation}
	{\partial_t \bs B^0}\times \hat{ \bs r}=\mathscr{D}_2\big[{\partial_t \bs B^1}\big]\times \hat{ \bs r} \quad{\rm at}\;\;\;\Gamma_1;\quad {\partial_t \bs B^2}\times \hat{ \bs r}=\mathscr{D}_2\big[{\partial_t \bs B^1}\big]\times \hat{ \bs r} \quad{\rm at}\;\;\; \Gamma_2,
	\end{equation}
	which,  together with \eqref{reducedsystem1eq1} and \eqref{B2D}, yields  the second jump conditions in \eqref{Interface}-\eqref{Interface1}. This ends the derivation.
\end{proof}	

\subsubsection{VSH-spectral-element discretization}\label{sect::sem}
In view of the spherical geometry and radially stratified dispersive media, we can fully exploit these advantages to develop  an efficient and accurate VSH-spectral-element solver for the Maxwell's system \eqref{sphericalcloakmodel}. Needless to say,  it is optimal compared with the FDTD simulation in  \cite[pp. 7307]{zhao2009full} for the time-domain Pendry's spherical cloak.

%The dispersive modeling in the last section shows that the special geometry of the spherical cloak results to an important fact that the integral operators in $\mathscr D_i, i=1, 2$ are independent of $\theta$ and $\phi$ and only act on the radial component of $\bs D$. Hence, the separation variable technique with VSH expasion can be used to reduce the governing equations in model problem \eqref{sphericalcloakmodel} to two series of partial differential equations w.r.t. $r$. The idea is used to produce a highly efficient VSH spectral element method for the simulation of 3D spherical cloak in this section.
%The implementation of the new algorithm is much more efficient than the algorithm based on the classic FDTD method (cf. \cite{zhao2009full}). \comm{\color{red} Maybe we can quote the precise claim therein!}

The key is to employ the divergence-free VSH expansion of the fields and reduce the governing equations into two sequences of decoupled one-dimensional problems.
By proposition \ref{EFsinsol}, the solenoidal fields $\bs D^i$, $\bs F^i$ and $\bs D^{in}$ can have VSH expansions
\begin{equation}\label{unknownfieldsexp}
\{\bs D^i, \bs F^i\}=\{u_{00}^i, f_{00}^i\}\,\bs Y_0^0 +  \sum_{l=1}^\infty\sum_{|m|=0}^l \Big\{\{u_{lm}^i,f_{1,l}^{i,m}\}\, \bs\Phi_l^m + \nabla \times (\{v_{lm}^i,f_{2,l}^{i,m}\}\,  \bs\Phi_l^m\big)\Big\},
\end{equation}
and
\begin{equation}
\bs D^{\rm in}=g_{00}\,\bs Y_0^0 +  \sum_{l=1}^\infty\sum_{|m|=0}^l \Big\{g_{lm}\, \bs\Phi_l^m
+ \nabla \times (h_{lm}\,  \bs\Phi_l^m\big)\Big\}.
\end{equation}
It is worthy of pointing out that the capacity operator in \eqref{DtN} has two alternative expressions \eqref{oldEtM} and \eqref{newEtM00}. Both use the usual VSH expansion coefficients. For example, we have
\begin{equation}\label{EtMD}
\mathscr T_b[\bs D^{3}]=\frac c b \sum_{l=1}^\infty\sum_{|m|=0}^l\bigg\{
\frac{\omega_l\ast D_{lm}^{r}}{l(l+1)}\,{\bs \Psi}_l^m+\big(\sigma_l\ast D_{lm}^{(2)}\big)\,\vt\bigg\},
\end{equation}
according to \eqref{newEtM00}, where $\{D_{lm}^r, D_{lm}^{(1)}, D_{lm}^{(2)}\}$ are the coefficients in the VSH expansion
\begin{equation}
\bs D^{3}=D_{00}\bs Y_0^0+\sum_{l=1}^\infty\sum_{|m|=0}^l\bigg\{
D_{lm}^{r}\,\bs Y_{lm}+D_{lm}^{(1)}\,{\bs \Psi}_l^m+D_{lm}^{(2)}\,\vt\bigg\}.
\end{equation}
In order to do dimension reduction using expansion \eqref{unknownfieldsexp}, we re-express the formulation \eqref{EtMD} using coefficients $\{u_{lm}^3, v_{lm}^3\}$. From the Proposition \ref{EFsinsol}, we have relations
\begin{equation}
D_{lm}^{r}=\frac{l(l+1)}{r}v_{lm}^{3},\quad D_{lm}^{(1)}=\hat{\partial}_rv_{lm}^{3}, \quad D_{lm}^{(2)}=u_{lm}^{3}.
\end{equation}
A simple substitution in \eqref{EtMD} gives
%\begin{equation}\label{oldEtMdivfree}
%{\mathscr T}_b[\bs E^{\rm sc}]:=\frac c b \sum_{l=1}^\infty\sum_{m=-l}^l\Big\{
%\big(\rho_l\ast \hat{\partial}_r\tilde E_{lm}^{sc}\big)\,{\bs\Psi}_l^m+\big(\sigma_l\ast E_{lm}^{sc}\big)\,\vt\Big\},
%\end{equation}
%or alternatively
\begin{equation}\label{newEtMdivfree}
{\mathscr T}_b[\bs D^{3}]=\frac c b \sum_{l=1}^\infty\sum_{m=-l}^l\bigg\{
b^{-1}(\omega_l\ast v_{lm}^{3})\,{\bs \Psi}_l^m+\big(\sigma_l\ast u_{lm}^{3}\big)\,\vt\bigg\}.
\end{equation}

\begin{proposition}\label{newprob}  For $l\geq 1$, $|m|\leq l$ and $i=0, 1, 2, 3,$ denote
	\begin{equation}\label{uvfh}
	g=g_{lm},\;\;  h=h_{lm},\;\;  u^i=u_{lm}^i, \;\;   v^i=v_{lm}^i, \;\; f_1^i=f_{1,l}^{i,m},\;\;  f_2^i=f_{2,l}^{i,m}, \;\; I_i:=(R_i, R_{i+1}).
	\end{equation}
	With the simple variable substitution
	\begin{equation}\label{varisub}
	\tilde{u}^0=\epsilon u^0,\quad \tilde{u}^1=u^1,\quad \tilde{u}^2=\epsilon u^2, \quad \tilde{u}^3=\epsilon u^3, \quad \tilde{g}=\epsilon g,
	\end{equation}
	the Maxwell system \eqref{sphericalcloakmodel} reduced to the following two sequences of one-dimensional problem for $v$ and $\tilde u$, respectively, for $l\geq 1$, $|m|\leq l$:
	\begin{subequations}\label{veqsys}
		\begin{numcases}{}
		\frac{\partial^2 v^i}{\partial t^2}- \frac{c^2}{r^2}\frac{\partial}{\partial r}\Big( r^2 \frac{\partial v^i}{\partial r}  \Big)+\frac{c^2\beta_l}{r^2}v^i=f_2^i, \quad r\in I_i,\quad i=0, 2, 3,\label{vhomo0}\\[4pt]
		\frac{\partial^2 v^1}{\partial t^2}-\frac{c^2}{\epsilon^2r^2}\frac{\partial}{\partial r}\Big( r^2 \frac{\partial v^1}{\partial r}  \Big)+\frac{c^2}{\epsilon} \frac{\beta_l}{r^2}v^1+\frac{c^2}{\epsilon } \frac{\beta_l}{r^2} \vartheta_1\ast v^1(r,t)=0,\quad r\in I_1,\label{vcloak}\\[4pt]
		v^0=v^1,\quad \partial_r v_1=\epsilon \partial_r v^0+(\epsilon-1)r^{-1}v^0\quad {\rm at}\;\;\; r=R_1,\label{interfacev}\\[4pt]
		v^2=v^1,\quad \partial_r v_1=\epsilon \partial_r v^2+(\epsilon-1)r^{-1}v^2\quad  {\rm at}\;\;\;r=R_2,\label{interfacev1}\\[4pt]
		v^2-v^3=h,\quad \partial_r v^2-\partial_rv^3=\partial_rh\quad {\rm at}\;\;\;r=R_3,\label{interfacev2}\\[4pt]
		\frac{1}{c}\partial_t v^3+\frac{\partial v^3 }{\partial r}+\frac{1}{b}v^3-\frac{1}{b}\sigma_l*v^3=0\quad {\rm at}\;\;\;r=b,\label{DtNv}\\[4pt]
		v|_{t=0}=\partial_t v|_{t=0}=0.\label{initialv}
		\end{numcases}
	\end{subequations}
	while $\tilde u$ satisfies the same equations in \eqref{veqsys} with $\tilde u,$ $f_1,$ $\tilde g$ and $\vartheta_2$, in place of $v,$ $f_2,$ $h$ and $\vartheta_1$, respectively, and for $l=m=0,$
	\begin{equation}
	\partial^2_t u_{00}^i=f_{00}^i,\quad u^i_{00}|_{t=0}=\partial_tu^i_{00}|_{t=0}=0,\quad r\in I_i.
	\end{equation}
\end{proposition}
\begin{proof}
	We postpone the detailed derivation in \ref{proofnewprob}.
\end{proof}

The above proposition shows that  $\tilde u$ and $v$ can be obtained by solving \eqref{veqsys} with different input data. Therefore, we only need to focus on the one dimensional problems \eqref{veqsys}. Note that the solution of \eqref{veqsys} has a jump at $r=R_3$. We introduce
\begin{equation}
\tilde v(r, t)=\begin{cases}
\displaystyle v(r, t), \quad 0\leq r\leq R_3,\quad t\geq 0,\\[4pt]
\displaystyle v(r, t)+h(R_3, t)\frac{b-r}{b-R_3},\quad R_3<r\leq b,\quad t\geq 0,
\end{cases}
\end{equation}
and
\begin{equation*}
\tilde f_2(r, t)=\begin{cases}
\displaystyle f_2(r, t), \quad 0< r< R_3,\;\; t>0,\\[4pt]
\displaystyle f_2(r, t)+\bigg\{\frac{\partial^2h(R_3, t)}{\partial t}+\Big(\frac{2bc^2}{r(b-r)}+\frac{c^2\beta_{l}}{r^2}\bigg)h(R_3, t)\bigg\}\frac{b-r}{b-R_3},\;\; R_3<r< b,\;\; t>0,
\end{cases}
\end{equation*}
to rewrite \eqref{veqsys} into
\begin{subequations}\label{veqsyscontinuous}
	\begin{numcases}{}
	\frac{\partial^2 \tilde v^i}{\partial t^2}- \frac{c^2}{r^2}\frac{\partial}{\partial r}\Big( r^2 \frac{\partial \tilde v^i}{\partial r}  \Big)+\frac{c^2\beta_l}{r^2}\tilde v^i=\tilde f_2^i, \quad r\in I_i, \;\; i=0, 2, 3,\label{vgeqoutcloak}\\[4pt]
	\frac{\partial^2 \tilde v^1}{\partial t^2}-\frac{c^2}{\epsilon^2r^2}\frac{\partial}{\partial r}\Big( r^2 \frac{\partial \tilde v^1}{\partial r}  \Big)+\frac{c^2}{\epsilon} \frac{\beta_l}{r^2}\tilde v^1+\frac{c^2}{\epsilon } \frac{\beta_l}{r^2} \vartheta_1\ast \tilde v^1(r,t)=0,\quad r\in I_1,\label{vgeqincloak}\\[4pt]
	\tilde v^0=\tilde v^1,\quad \partial_r \tilde v_1=\epsilon \partial_r \tilde v^0+(\epsilon-1)r^{-1}\tilde v^0\quad {\rm at}\;\;\; r=R_1,\\[3pt]
	\tilde v^2=\tilde v^1,\quad \partial_r \tilde v_1=\epsilon \partial_r \tilde v^2+(\epsilon-1)r^{-1}\tilde v^2\quad  {\rm at}\;\;\;r=R_2,\\
	\tilde v^2=\tilde v^3,\quad \partial_r \tilde v^2-\partial_r\tilde v^3=\partial_rh+\frac{1}{b-R_3}h(R_3, t)\quad {\rm at}\;\;\;r=R_3,\\
	\frac{1}{c}\partial_t\tilde v^3+\frac{\partial \tilde v^3 }{\partial r}+\frac{1}{b}\tilde v^3-\frac{1}{b}\sigma_l*\tilde v^3=-\frac{1}{b-R_3}h(R_3, t)\quad {\rm at}\;\;\;r=b,\\
	\tilde v|_{t=0}=h(R_3, 0)\frac{b-r}{b-R_3}\chi_{[R_3, b]},\quad \partial_t \tilde v|_{t=0}=\frac{\partial h(R_3, 0)}{\partial t}\frac{b-r}{b-R_3}\chi_{[R_3, b]}.
	\end{numcases}
\end{subequations}
Here $\chi_{[R_3, b]}$ is the indicator function which is equal to $1$ inside the interval $[R_3,b]$ and vanish outside.

Obviously, $\tilde v(r, t)$ is continuous in $I$. Multiplying \eqref{vgeqoutcloak} and \eqref{vgeqincloak} by test function $r^2\phi$ and $\epsilon r^2\phi$ respectively for $\phi\in H^1(I)$, using integration by parts and summing up the resulted equations, then applying the interface conditions and boundary condition we obtain the variational problem: Find $\tilde v(\cdot, t)\in H^1(I)$, s.t.
\begin{equation}\label{variationalprob}
\mathcal{B}(\tilde v, \phi)=(\tilde f_2, \phi)+\Big(\partial_rh(R_3, t)+\frac{h(R_3,t)}{b-R_3}\Big)c^2R_3^2\phi(R_3)-\frac{h(R_3,t)}{b-R_3}c^2b^2\phi(b),
\end{equation}
for all $\phi\in H^1(I)$, where
\begin{equation}\label{Avw1}
\begin{split}
\mathcal B(\tilde v,\phi):=&\int_{I\backslash I_1}(r^2\partial_{tt}\tilde v\phi+c^2r^2\partial_r\tilde v\partial_r\phi)dr+\int_{I_1}\Big(\epsilon r^2\partial_{tt}\tilde v\phi+\frac{c^2r^2}{\epsilon}\partial_r\tilde v\partial_r\phi\Big)dr\\
+&\beta_lc^2\Big(\int_{I}\tilde v\phi dr+\int_{I_1}\vartheta_1\ast \tilde v(r,t)\phi(r)dr\Big)+c^2(\epsilon-1)R_1\tilde v(R_1,t)\phi(R_1)\\
-&c^2(\epsilon-1)R_2\tilde v(R_2,t)\phi(R_2)+\big\{cb^2\partial_t\tilde v(b,t)+c^2b(\tilde v(b,t)-\sigma\ast \tilde v(b, t))\big\}\phi(b).
\end{split}
\end{equation}

%Noting that the solution $v$ and test function $\phi$ are discontinuous at point $r=R_3$, the notations $\phi(R_3-0)$ and $\phi(R_3+0)$ stand for the left and right limit of $\phi$ at $r=R_3$. In the variational problem, we do not handle the standard interface conditions \eqref{varprobintercond}. Since we shall use interface conforming mesh in the discretization, it finds that it is very easy to deal with them in the numerical scheme. As we will show later, it only requires a small change on the right hand side of the standard spectral element discretization.

Based on the variational problem \eqref{variationalprob}, we introduce the spectral-element discretization.
Let $\mathcal{I}_h:0=r_0<r_1<\cdots<r_E=b$ be an interface conforming mesh of the interval $I$ and denote the element by $\{K_e=(r_{e-1}, r_e)\}_{e=1}^E$. Here, the interface conforming mesh means that the points $r=R_1, R_2, R_3$ are mesh points, see Figure \ref{mesh}. Let $\mathcal{P}_N(K_e)$ be the set of all complex valued polynomials of degree at most $N$ in each interval $K_e$ and define the spectral element approximation space as
\begin{equation}\label{XnXn}
\mathscr{X}_N(\mathcal{I}_h):=\big\{u\in H^1(I): u|_{K_e}\in \mathcal{P}_N(K_e)\big\}.
\end{equation}
The spectral element discretization of \eqref{veqsyscontinuous} is to find $\tilde v_N(r,t)\in {\mathscr X}_N(\mathcal{I}_h)$ for all $t>0$, such that
\begin{equation}\label{uN1dprob}
\mathcal{B}(\tilde v_N, \phi)=(\tilde f_2, \phi)+\Big(\partial_rh(R_3, t)+\frac{h(R_3,t)}{b-R_3}\Big)c^2R_3^2\phi(R_3)-\frac{h(R_3,t)}{b-R_3}c^2b^2\phi(b),
\end{equation}
for all $\phi\in {\mathscr X}_N(\mathcal{I}_h)$.
\begin{figure}[!ht]
	\centering
	\includegraphics[width=0.7\linewidth]{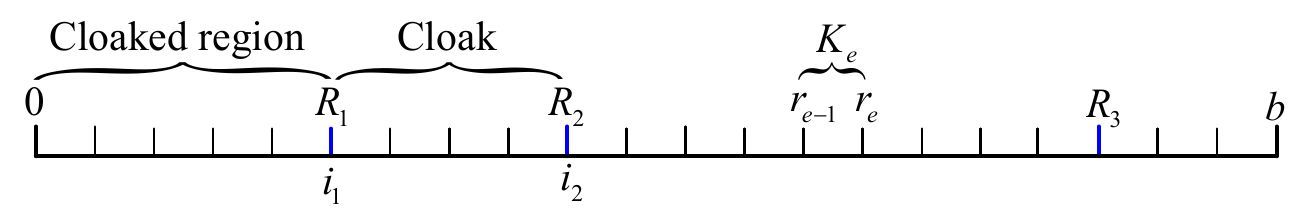}
	\caption{\small Interface conforming mesh used by the spectral element discretization.}
	\label{mesh}
\end{figure}

This spectral element discretization for $\tilde v$ leads to the following integral differential system
\begin{equation}
\label{diffinteq2}
\begin{split}
&\mathbb M\ddot{\bs V}+\mathbb B\dot{\bs V}+{\mathbb C}{\bs V}+\bs G-\frac{c}{b}\mathbb B(\sigma\ast{\bs V})={\bs F}, \quad \bs V(0)=\bs V_0,\quad \dot{\bs V}(0)=\bs V_1,
\end{split}
\end{equation}
where
$$\mathbb M=(m_{ij})_{\mathcal{N}\times \mathcal{N}},\quad \mathbb A=(a_{ij})_{\mathcal{N}\times \mathcal{N}},\quad \bs G=(G_i)_{\mathcal{N}},$$
$$\mathbb B=cb^2\mathbb E_{\mathcal{N}\mathcal{N}},\quad\mathbb C=\mathbb A+c^2b\mathbb E_{\mathcal{N}\mathcal{N}}+c^2(\epsilon-1)(R_1\mathbb E_{i_1i_1}-R_2\mathbb E_{i_2i_2}),$$
are matrices with entries given by
\begin{equation*}
\begin{split}
& m_{ij}=\int_{I\backslash I_1}r^2\phi_i\phi_jdr+\epsilon\int_{I_1}r^2\phi_i\phi_jdr,\quad G_i=\beta_lc^2\int_{I_1}{\vartheta}_1\ast \tilde v_N(r, t)\phi_i(r)dr,\\
& a_{ij}=c^2\int_{I\backslash I_1}(r^2\partial_r\phi_j\partial_r\phi_i+\beta_l\phi_j\phi_i)dr+c^2\int_{I_1}\Big(\frac{r^2}{\epsilon}\partial_r\phi_j\partial_r\phi_i+\beta_l\phi_j\phi_i\Big)dr,\\
&F_i=\int_{I}\tilde f_2\phi_i(r)dr+\Big(\partial_rh(R_3, t)+\frac{h(R_3,t)}{b-R_3}\Big)c^2R_3^2\phi_i(R_3)-\frac{h(R_3,t)}{b-R_3}c^2b^2\phi_i(b).
\end{split}
\end{equation*}
Here, $i_1$, $i_2$ denote the global index of the freedom at $r=R_1, R_2$ (see Figure \ref{mesh} for illustration), respectively, $\mathcal{N}$ is the degree of freedom and also the global index of the freedom attached to mesh point $r=b$, and $\mathbb E_{mn}=(E_{ij})_{\mathcal{N}\times \mathcal{N}}$ is the matrix with only one non-zero entry $E_{mn}=1$.

\subsubsection{Newmark's scheme for time discretization}
The spectral element discretization leads to the integral differential system \eqref{diffinteq2} w.r.t $t$. Noting that all the involved time integrations are actually convolutions of exponential functions with the unknown functions, fast algorithm based on formula \eqref{fastconvformula} can be used. Let us first discuss the discretization of the convolution ${\vartheta}_1\ast \tilde v_N(r, t)$. Define
\begin{equation}
\tilde{\vartheta}_1^{\alpha}(r, t)=e^{\ri \zeta^0_1(r) t},\quad \tilde{\vartheta}_1^{\beta}(r, t)=e^{\ri \zeta^1_1(r) t}.
\end{equation}
Then
\begin{equation}\label{thetaconvsplit}
{\vartheta}_1\ast \tilde v_N(r, t)=\frac{\ri (\omega_{p,1}(r))^2}{\zeta^0_1(r)-\zeta^1_1(r)}\big(\tilde{\vartheta}_1^{\alpha}\ast \tilde v_N(r, t)-\tilde{\vartheta}_1^{\beta}\ast \tilde v_N(r,t)\big).
\end{equation}
By using the trapezoidal rule and  \eqref{fastconvformula}, we have the second-order approximations
\begin{equation}\label{kernelastv}
\begin{split}
&\tilde{\vartheta}_1^{\alpha}\ast \tilde v_N(r,t_{n+1})\approx\lambda_0(r)\tilde{\vartheta}_1^{\alpha}\ast \tilde v_N(r, t_{n})+\frac{\Delta t}{2}(\tilde v_N(r, t_{n+1})+\lambda_0(r)\tilde v_N(r, t_n)),\\
&\tilde{\vartheta}_1^{\beta}\ast \tilde v_N(r, t_{n+1})\approx\lambda_1(r)\tilde{\vartheta}_1^{\beta}\ast \tilde v_N(r, t_{n})+\frac{\Delta t}{2}(\tilde v_N(r, t_{n+1})+\lambda_1(r)\tilde v_N(r,t_n)),
\end{split}
\end{equation}
where
\begin{equation}
\lambda_0(r)=e^{\ri \zeta^0_1(r)\Delta t},\quad \lambda_1(r)=e^{\ri \zeta^1_1(r)\Delta t}.
\end{equation}
Substituting \eqref{kernelastv} into \eqref{thetaconvsplit}, we obtain % approximation
\begin{equation}
{\vartheta}_1\ast \tilde v_N(r, t_{n+1})\approx \frac{\ri(\omega_{p,1}(r))^2}{\zeta^0_1(r)-\zeta^1_1(r)}\big(\tilde v^c_{N}(r,t_n)+\frac{\Delta t}{2}(\lambda_0(r)-\lambda_1(r))\tilde v_N(r,t_n)\big),
\end{equation}
where
$$\tilde v^c_N(r, t_n):=(\lambda_0(r)\tilde{\vartheta}_1^{\alpha}-\lambda_1(r)\tilde{\vartheta}_1^{\beta})\ast \tilde v_N(r, t_n).$$
Thus, we get the discretization for $\bs G(t_{n+1})$ given by $\bs G^{n+1}:=(G_i^{n+1})$ with
\begin{equation}\label{Gtimediscrete}
G_i^{n+1}:=\frac{\ri\beta_lc^2(\omega_{p,1}(r))^2}{\zeta^0_1(r)-\zeta^1_1(r)}\int_{I_1}(v^c_{N}(r,t_n)
+\frac{\Delta t}{2}(\lambda_0(r)-\lambda_1(r))v_N(r,t_n))\phi_i(r)\,dr.
\end{equation}
It is important to point out that $\bs G^{n+1}$ is a vector obtained by using the solution before the current time step thus can be moved to the right hand side in the fully discretization scheme. We denote the new right hand side vector by $\widetilde{\bs F}^{n}=\bs F^n-\bs G^n$.

Next, we consider the discretization of the convolution term $(\sigma_l\ast\bs V)(t)$.  For this purpose, we define
\begin{equation}
{\bs V}_j(t):=\int_0^te^{c(t-\tau)z_j^l/b}\bs V(\tau)\,d\tau.
\end{equation}
By using the trapezoidal rule and  \eqref{fastconvformula} again, we obtain the second order approximations
\begin{equation}
{\bs V}^{0}_j=\bs 0, \quad {\bs V}^{n+1}_j= e^{c\Delta tz_j^l/b}{\bs V}^{n}_j+\frac{\Delta t}{2}\bs V^{n+1}+\frac{\Delta t}{2}e^{c\Delta tz_j^l/b}\bs V^n,\;\;
\end{equation}
of ${\bs V}_j(t_{n+1})$ for $j=1, 2, \cdots, l$.
Accordingly, we have
\begin{equation}\label{nrbctimediscrete}
(\sigma\ast{\bs V})^0=\bs 0,\quad (\sigma\ast{\bs V})^{n+1}=\frac{\Delta t}{2}\alpha_1\bs V^{n+1}+\frac{\Delta t}{2}\alpha_2\bs V^{n}+\sum_{j=1}^{l}\alpha_2^j{\bs V}^n_j,
\end{equation}
with
$$\alpha_1=\frac{c}{b} \sum_{j=1}^{l}z_j^l, \quad \alpha_2^j=\frac{c}{b}z_j^le^{c\Delta tz_j^l/b},\quad \alpha_2=\sum_{j=1}^{l}\alpha_2^j,$$
is a second order discretization of the convolution term $(\sigma\ast{\bs V})(t_{n+1})$.

For the dicretization of  time derivatives, we adopt the new marks scheme (cf. \cite{Wang2Zhao12}). The key idea is to use the approximations:
\begin{align}
{\bs V}^{n+1}&={\bs V}^{n}+\Delta t\dot{\bs V}^{n}+\frac{\Delta t^2}{2}(1-2\beta)\ddot{\bs V}^{n}+\beta\Delta t^2\ddot{\bs V}^{n+1},\label{newmarkskey1}\\
\dot{\bs V}^{n+1}&=\dot{\bs V}^{n}+(1-\gamma)\Delta t\ddot{\bs V}^{n}+\gamma\Delta t\ddot{\bs V}^{n+1},\label{newmarkskey2}
\end{align}
where $\beta$ and $\gamma$ are given parameters. Using the approximations \eqref{Gtimediscrete} and \eqref{nrbctimediscrete}, we can formulate the time discretization of the system \eqref{diffinteq2} at $t_{n+1}$ as
\begin{equation}
\label{diffeqtlvn}
\begin{split}
\mathbb M\ddot{\bs V}^{n+1}+\mathbb B\dot{\bs V}^{n+1}+\mathbb C{\bs V}^{n+1}
-\frac{c}{b}\mathbb B\bigg\{\frac{\alpha_1\Delta t}{2}\bs V^{n+1}+\frac{\alpha_2\Delta t}{2}\bs V^{n}+\sum\limits_{j=1}^l\alpha_2^j{\bs V}^n_j\bigg\}=\widetilde{\bs F}^{n+1}.
\end{split}
\end{equation}
Inserting \eqref{newmarkskey2} into \eqref{diffeqtlvn} to eliminate $\dot{\bs V}^{n+1}$ leads to
\begin{eqnarray}
\big(\mathbb M+\gamma\Delta t\mathbb B\big)\ddot{\bs V}^{n+1}+\Big\{\mathbb C-\frac{\alpha_1c\Delta t}{2b}\mathbb B\Big\}{\bs V}^{n+1}=\widetilde{\bs F}^{n+1}-\bs W^n,\label{newmarkeqn}
\end{eqnarray}
where
\begin{equation*}
\bs W^{n}=(1-\gamma)\Delta t\mathbb B\ddot{\bs V}^n+\mathbb B\dot{\bs V}^{n}-\frac{\alpha_2c\Delta t}{2b}\mathbb B\bs V^{n}-\frac{c}{b}\mathbb B\sum\limits_{j=1}^l\alpha_2^j{\bs V}^n_j.
\end{equation*}
From \eqref{newmarkskey1}, we have
\begin{equation}\label{vddotapp}
\beta \Delta t^2\ddot{\bs V}^{n+1}=\bs V^{n+1}-\Big\{\bs V^n+\Delta t\dot{\bs V}^n+\Delta t^2\Big(\frac{1}{2}-\beta\Big)\ddot{\bs V}^n\Big\}:=\bs V^{n+1}-\widetilde{\bs V}^n.
\end{equation}
Using \eqref{vddotapp} in \eqref{newmarkeqn} we arrive the fully discretization scheme
\begin{equation}
\bigg\{\mathbb M+\Big(\gamma\Delta t-\frac{\alpha_1c\beta\Delta t^3}{2b}\Big)\mathbb B+\beta\Delta t^2\mathbb C\bigg\}{\bs V}^{n+1}
=\beta\Delta t^2\big(\widetilde{\bs F}^{n+1}-\bs W^n\big)+(\mathbb M+\gamma\Delta t\mathbb B)\widetilde{\bs V}^n.
\end{equation}
It is known that in general, the Newmark's scheme is of second-order and unconditionally stable, if the parameter satisfy $\gamma\geq\frac{1}{2}$ and $\beta\geq\frac{1}{4}(\frac{1}{2}+\gamma)^2$.

\section{Numerical experiments}\label{sect:numer}
In this section, we shall validate the feasibility and accuracy of the methodology for the simulation of 3D spherical cloaks via some numerical experiments. In all the experiments, we set $\varepsilon_0=\mu_0=1$, $c=1/\sqrt{\varepsilon_0\mu_0}=1$, $R_3=0.95$, $b=1$, $E=20$, $N=20$, $\triangle t=1.0e-3$. All VSH expansions are truncated at $L=40$ and the parameters in Newmark's time discretization are set to $\gamma=0.5, \beta=0.25$.
\subsection{Monochromatic incident wave}
Set
\begin{equation}\label{inc1}
\bs D^{\rm in}(\bs r, t)=(1-e^{-10t})\cos(kx-\omega t)\bs A,\quad \bs A:=\begin{bmatrix}
0 & 0 & A
\end{bmatrix}^{\rm T},
\end{equation}
where $k=\omega=40$ and $A=1$. Note that it gets close to a monochromatic wave very quickly as $t$ increases, e.g. $t>3, (1-e^{-10t})\geq 0.999999999999906$ due to the exponential term. The parameters of the cloaking device are set $\omega_c=40$, $\gamma_1=\gamma_2=0.001$, $R_1=0.15$, $R_2=0.35$. We plot the contours of $D_z$ at different time in Figure \ref{example2fig}. It shows that the spherical domain $|\bs r|<R_1$ is perfectly cloaked from monochromatic wave with angular frequency $\omega=40$. No waves are propagating inside the cloaked region.
\begin{figure}[h!]
	{~}\hspace*{-16pt}
	\centering
	\subfigure[t=1]{\includegraphics[scale=.2]{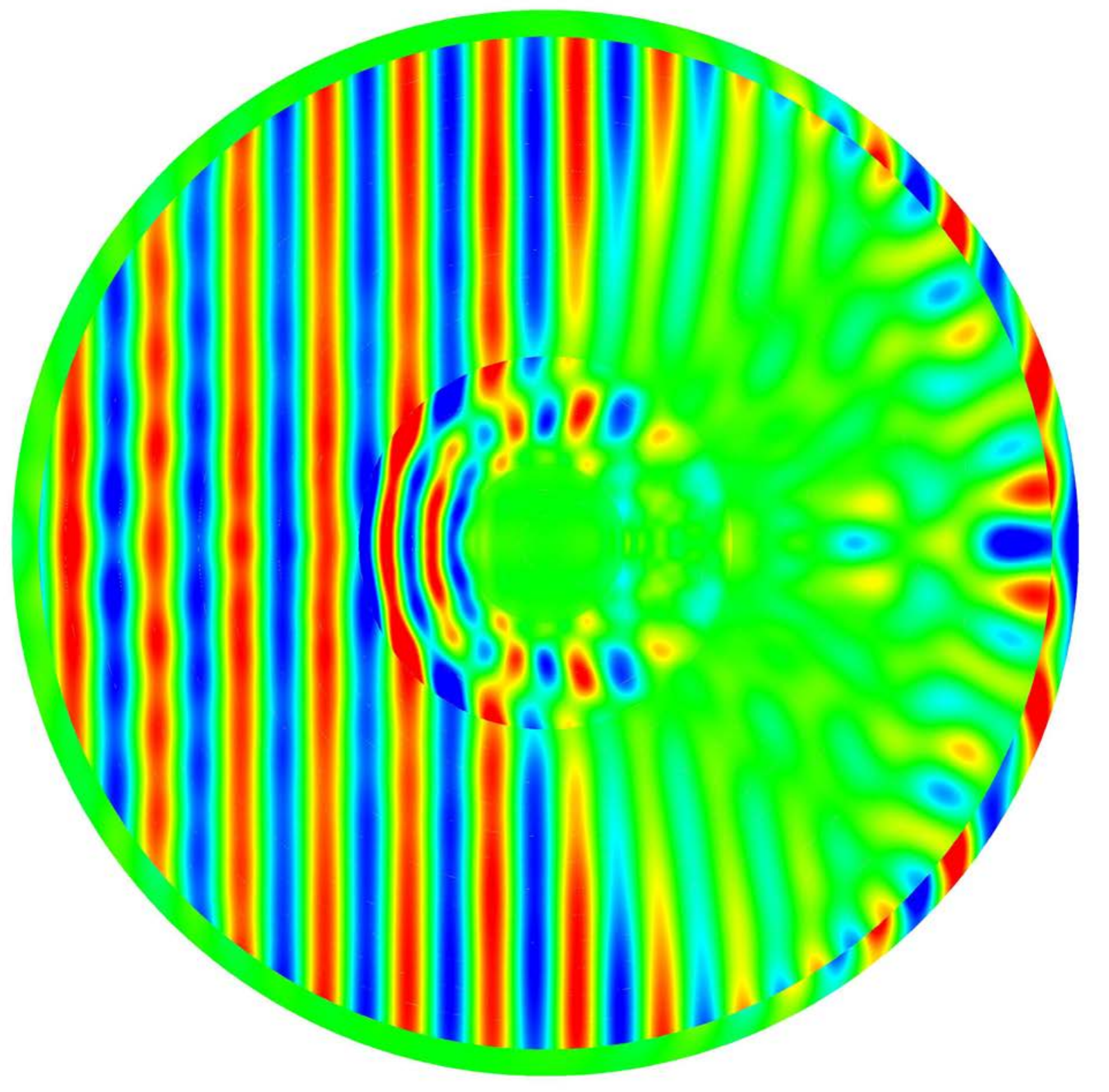}}
	\subfigure[t=3]{ \includegraphics[scale=.2]{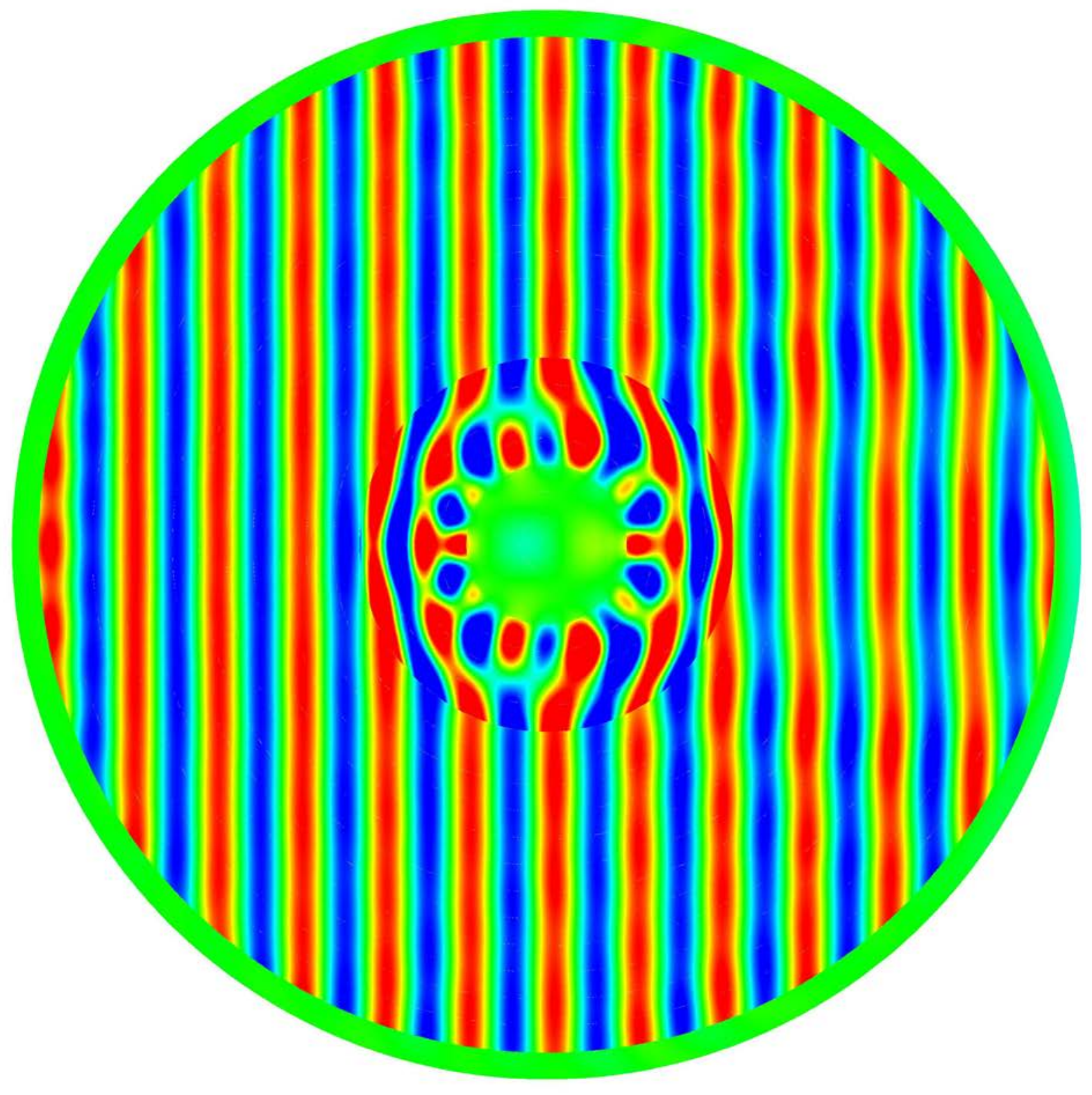}}
	\subfigure[t=5]{ \includegraphics[scale=.2]{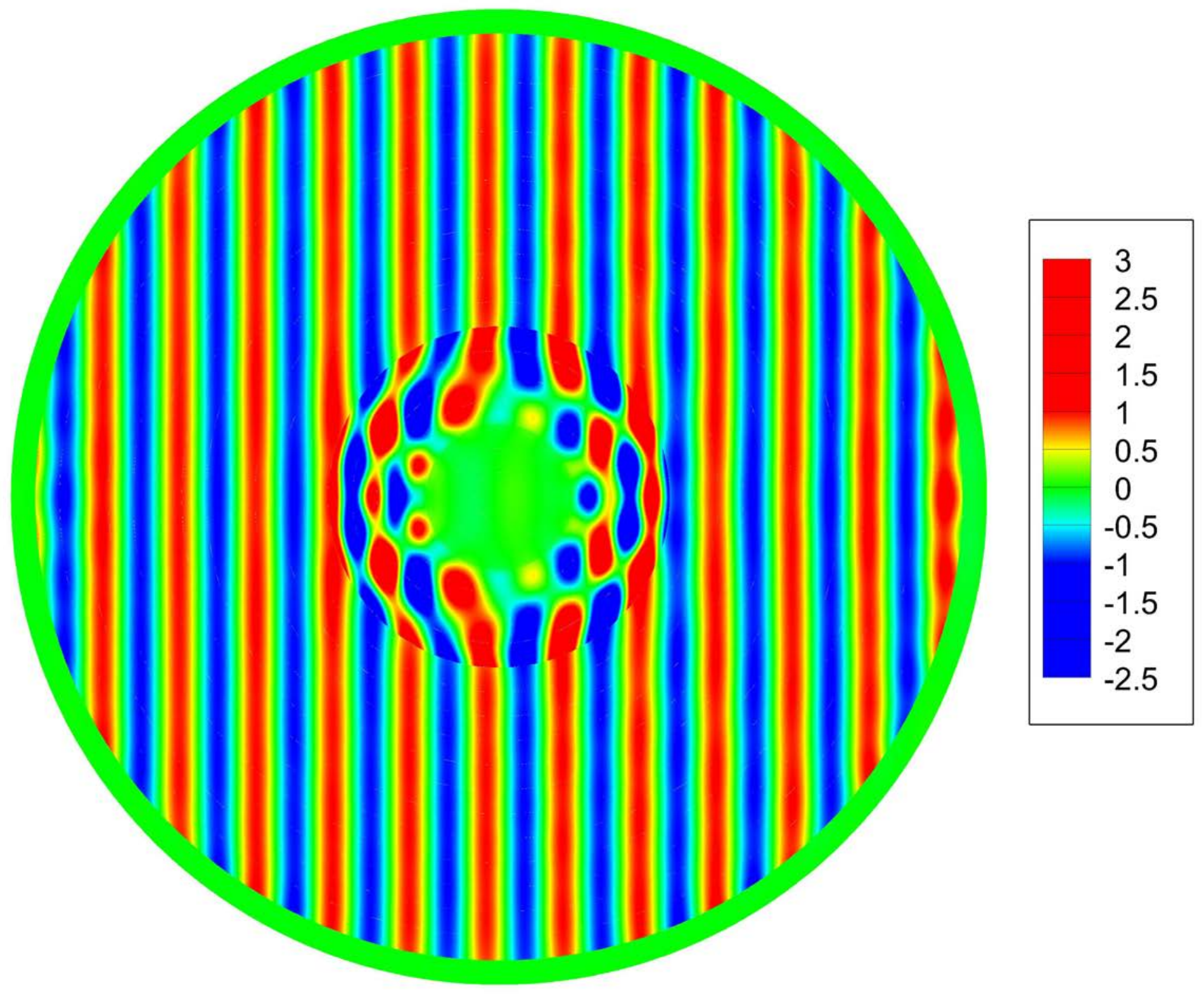}}\\
	\subfigure[t=7]{ \includegraphics[scale=.2]{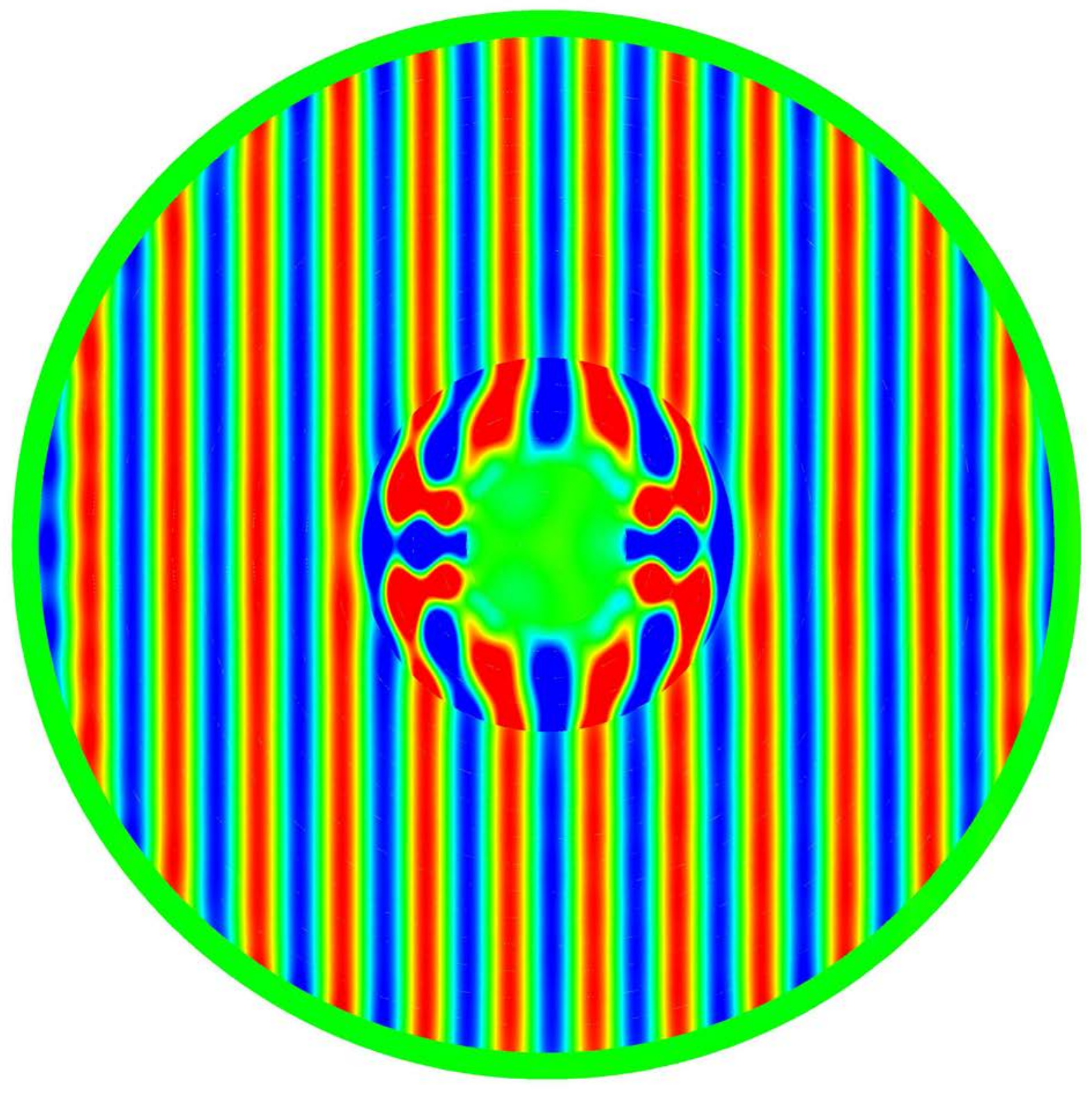}}
	\subfigure[t=9]{ \includegraphics[scale=.2]{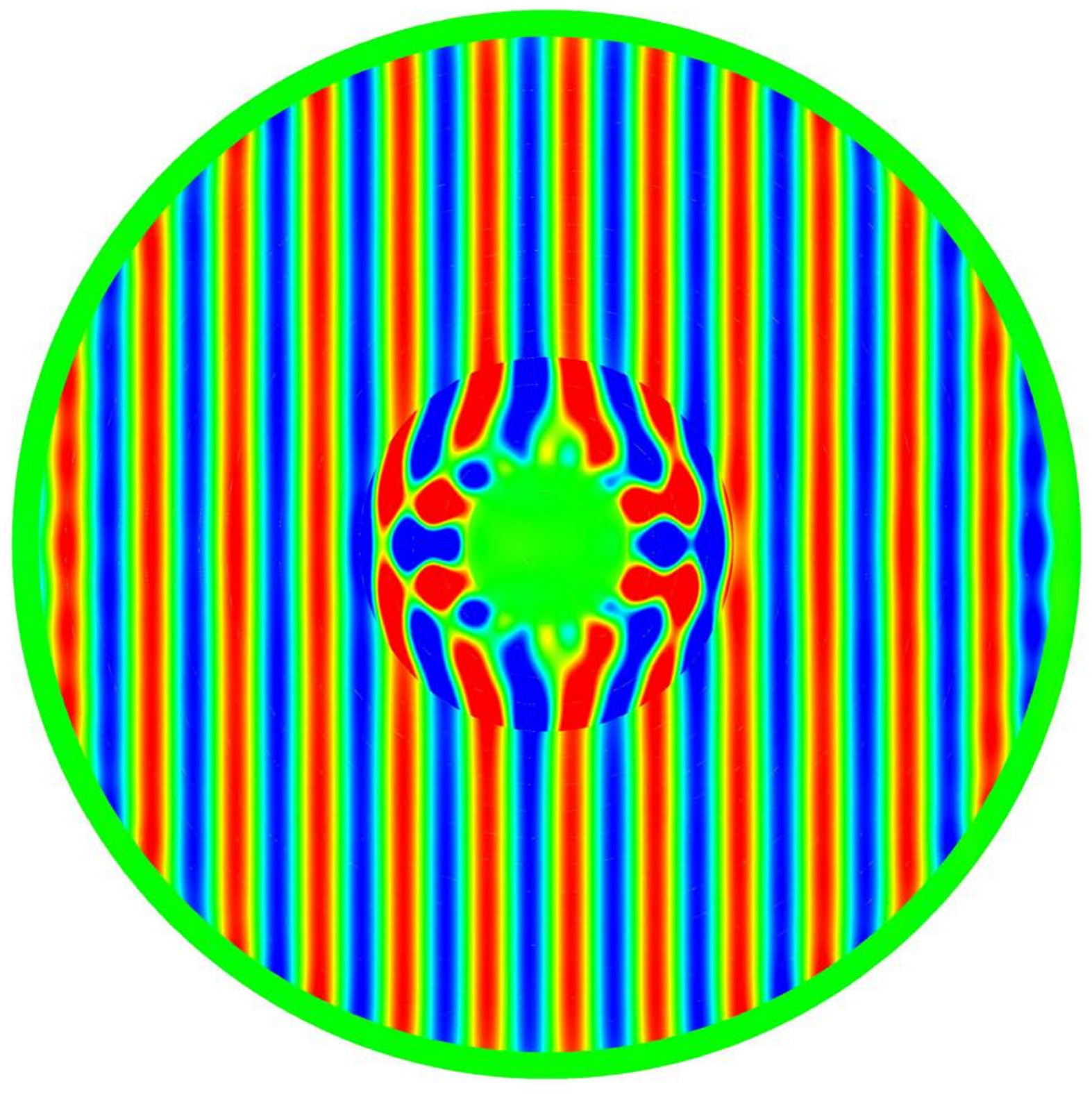}}
	\subfigure[t=11]{ \includegraphics[scale=.2]{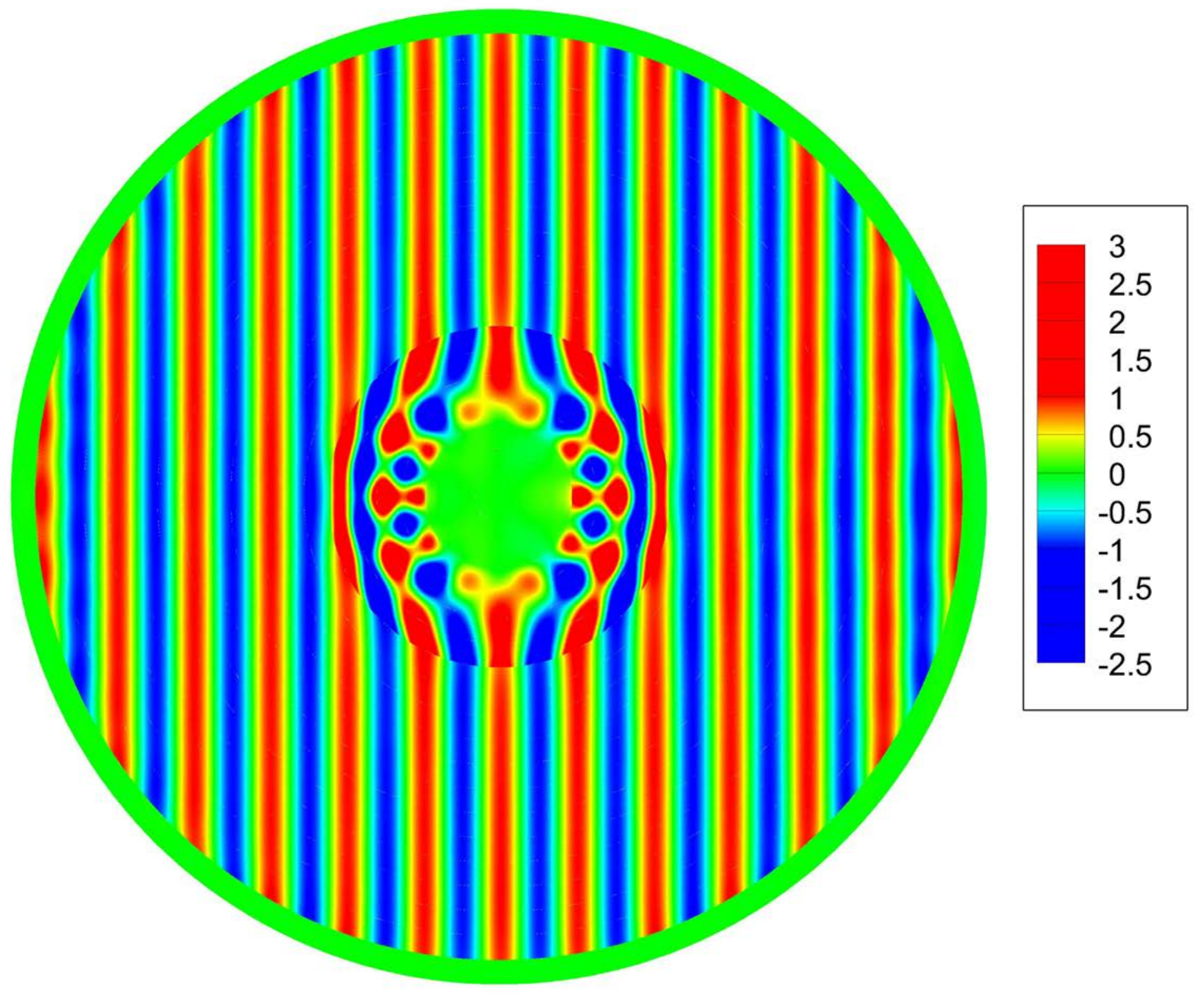}}
	\caption{\small Contours of the approximated $D_z$ in the $XY$ plane at different time steps with $R_2>2R_1$.}
	\label{example2fig}
\end{figure}
%\begin{figure}[h!]
%	{~}\hspace*{-16pt}
%	\subfigure[t=1]{\includegraphics[scale=.25]{example3(t=1)}}
%	\subfigure[t=3]{ \includegraphics[scale=.25]{example3(t=3)}}
%	\subfigure[t=5]{ \includegraphics[scale=.25]{example3(t=5)}}
%	\subfigure[t=7]{ \includegraphics[scale=.25]{example3(t=7)}}
%	\subfigure[t=9]{ \includegraphics[scale=.25]{example3(t=9)}}
%	\subfigure[t=11]{ \includegraphics[scale=.25]{example3(t=11)}}
%	\caption{\small Scattering-total field at different time step.}
%	\label{example3fig}
%\end{figure}

In the analysis of  \cite{Li2014wellpose}, a constraint $R_2\geq 2R_1$ is assumed. It was  pointed out  that it was unsure if this constraint is necessary and no numerical results regarding the case $R_2<2R_1$ were presented therein. Here, we shall do some numerical test for the case $R_2<2R_1$. For this purpose, we set $R_1=0.15$, $R_2=0.25<2R_1$ and plot the contours of $D_z$ at different time in Figure \ref{SemC}. It shows that the cloak works as well as in the case $R_2>2R_1$.
\begin{figure}[ht!]
	{~}\hspace*{-16pt}
	\centering
	\subfigure[t=1]{\includegraphics[scale=.2]{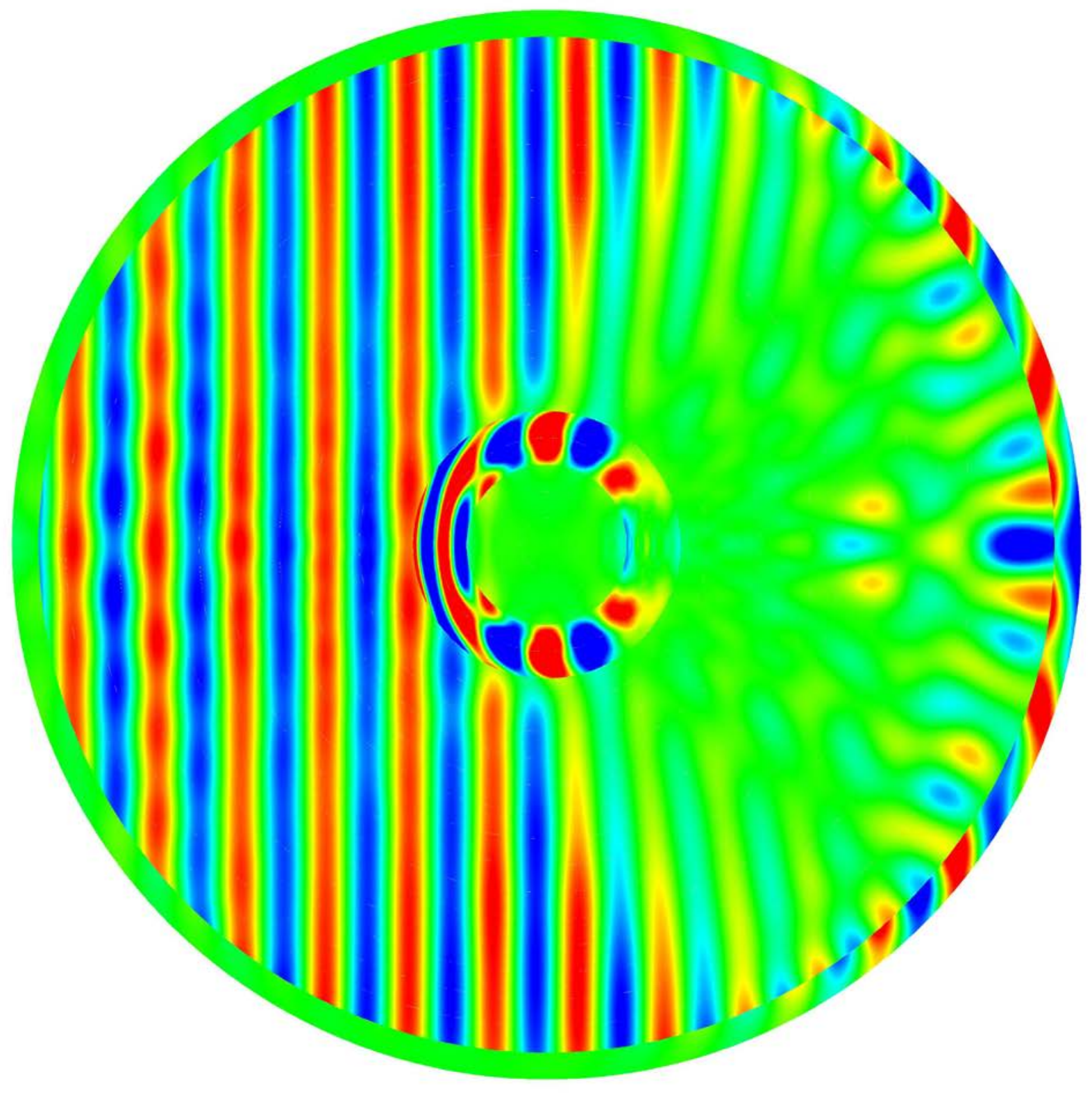}}
	\subfigure[t=3]{ \includegraphics[scale=.2]{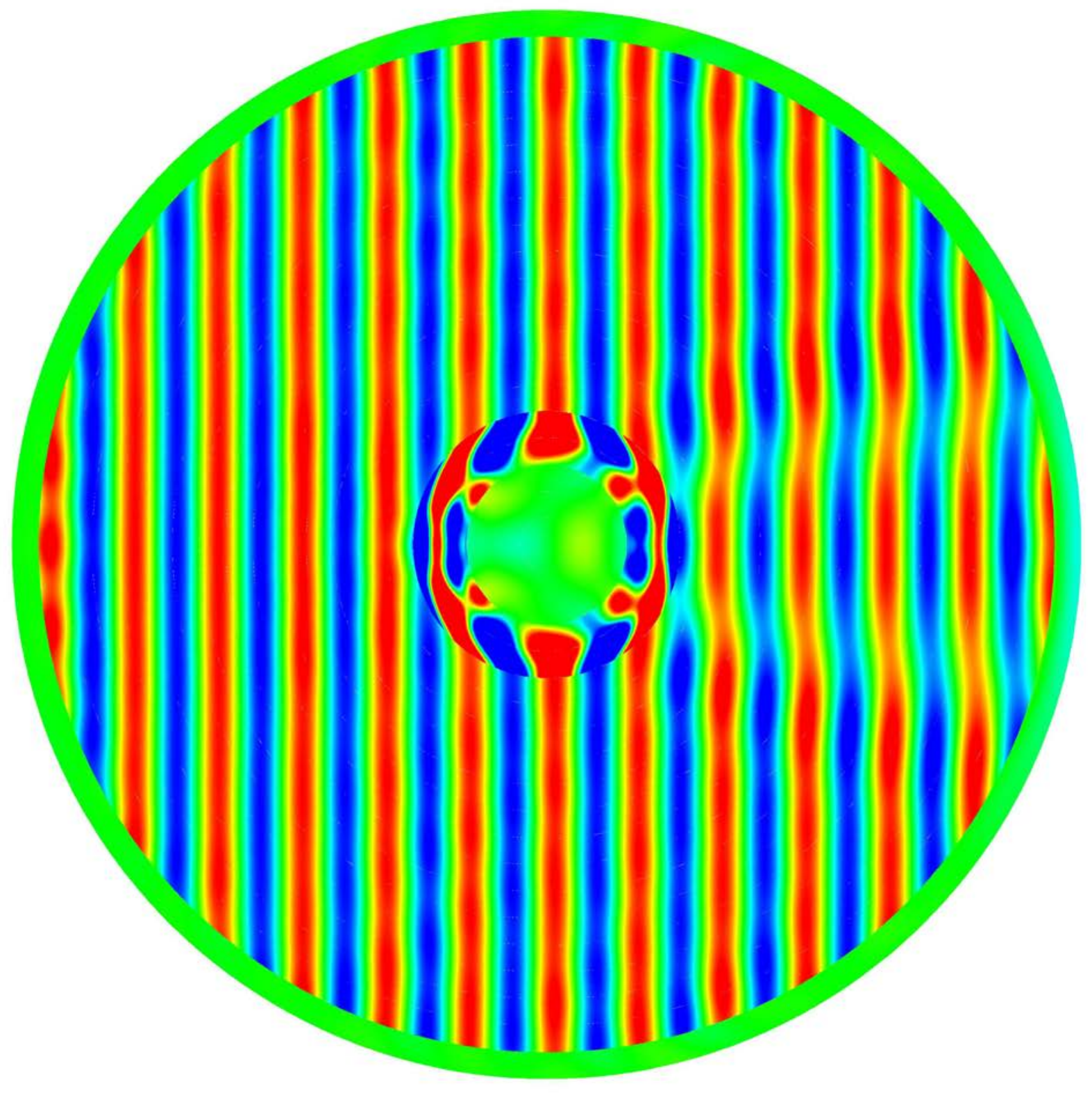}}
	\subfigure[t=5]{ \includegraphics[scale=.2]{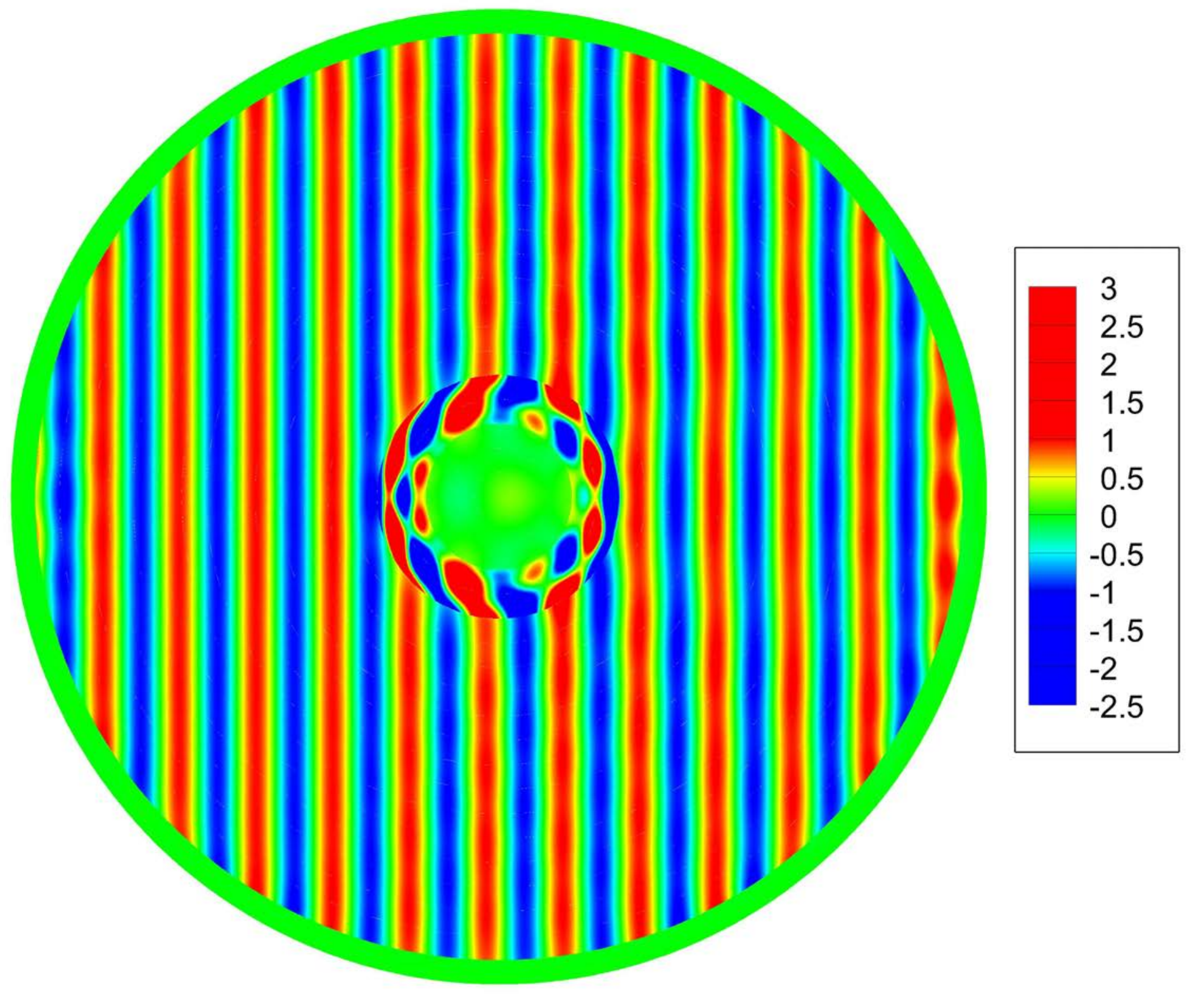}}\\
	\subfigure[t=7]{ \includegraphics[scale=.2]{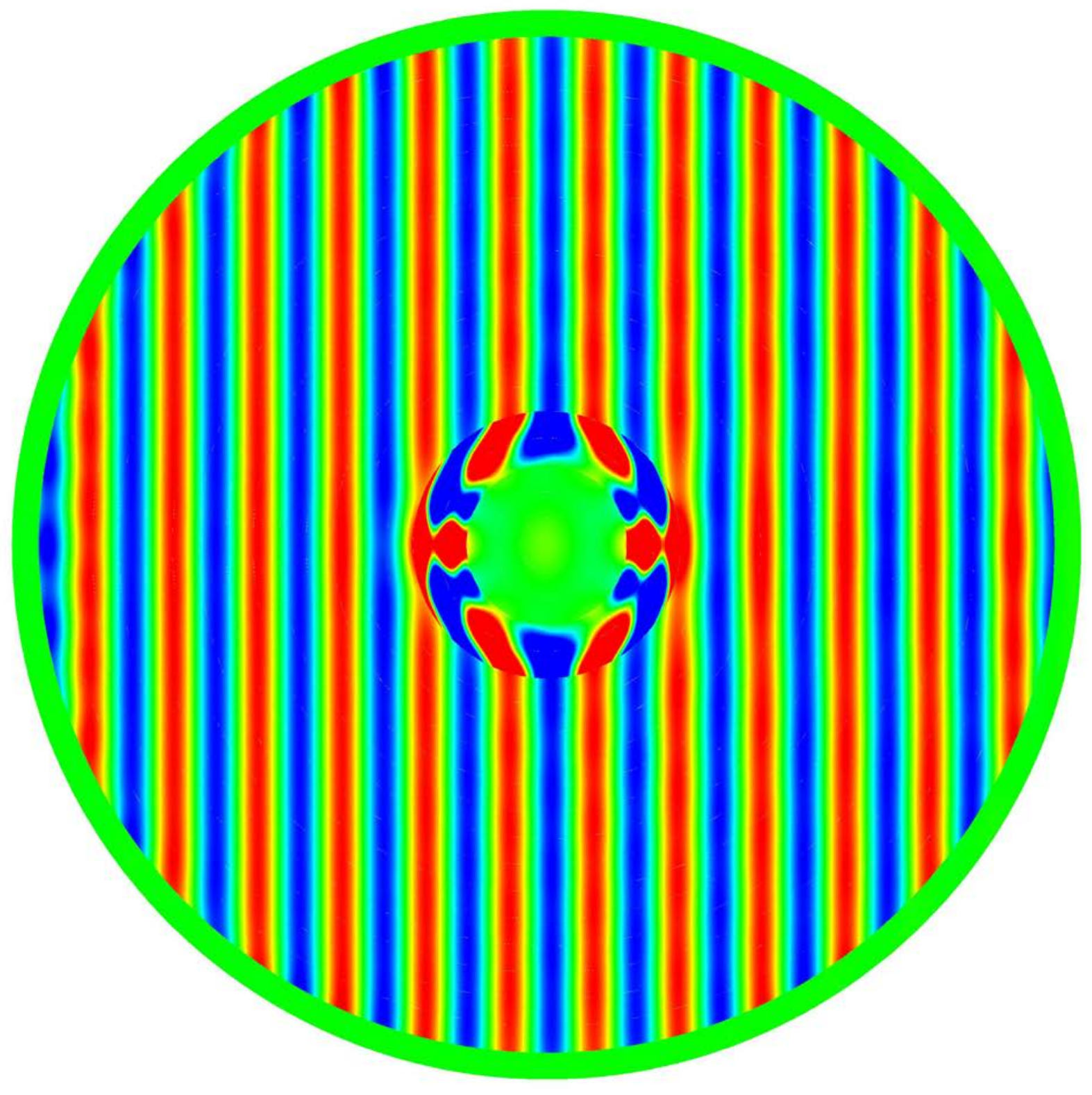}}
	\subfigure[t=9]{ \includegraphics[scale=.2]{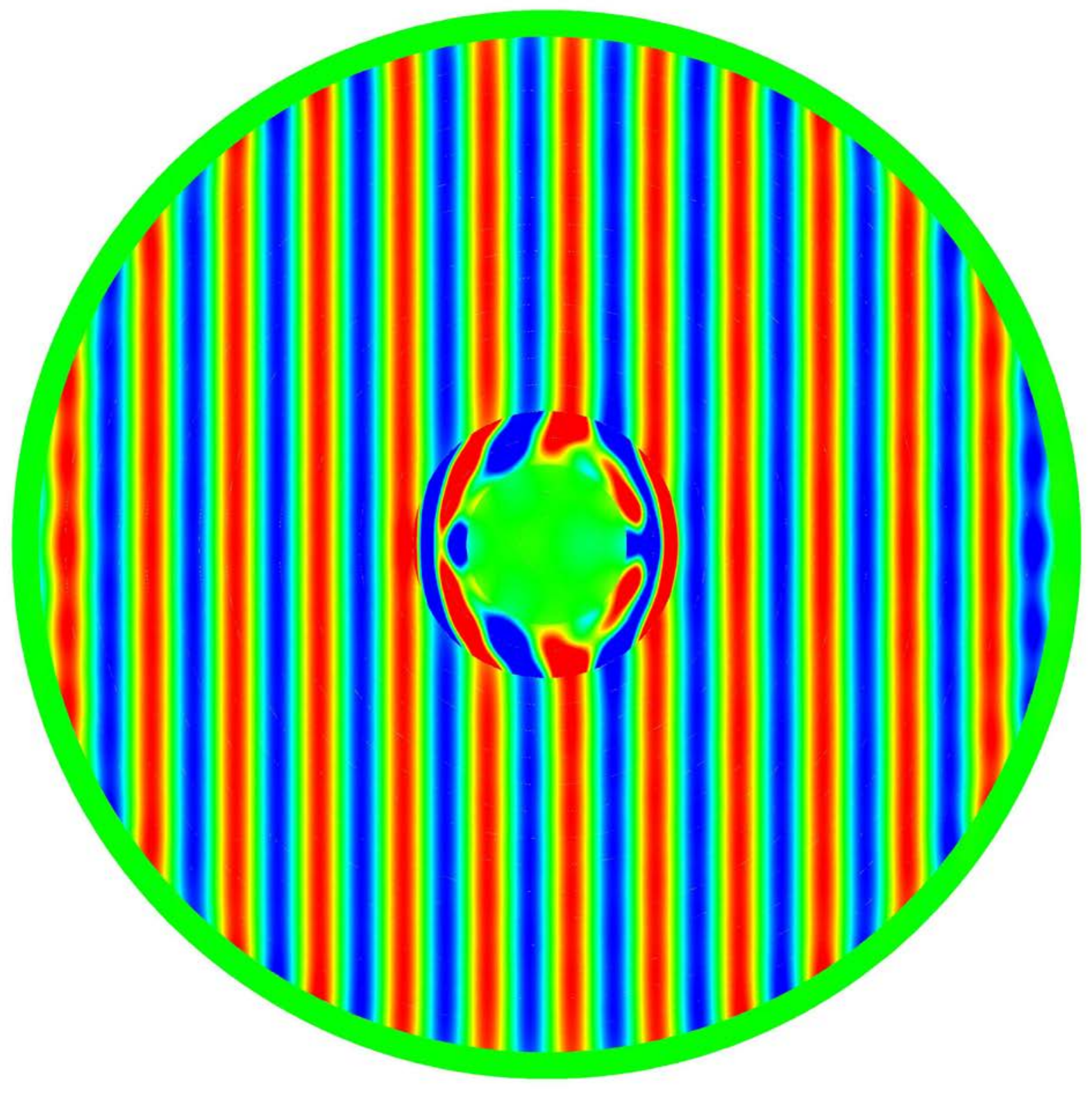}}
	\subfigure[t=11]{ \includegraphics[scale=.2]{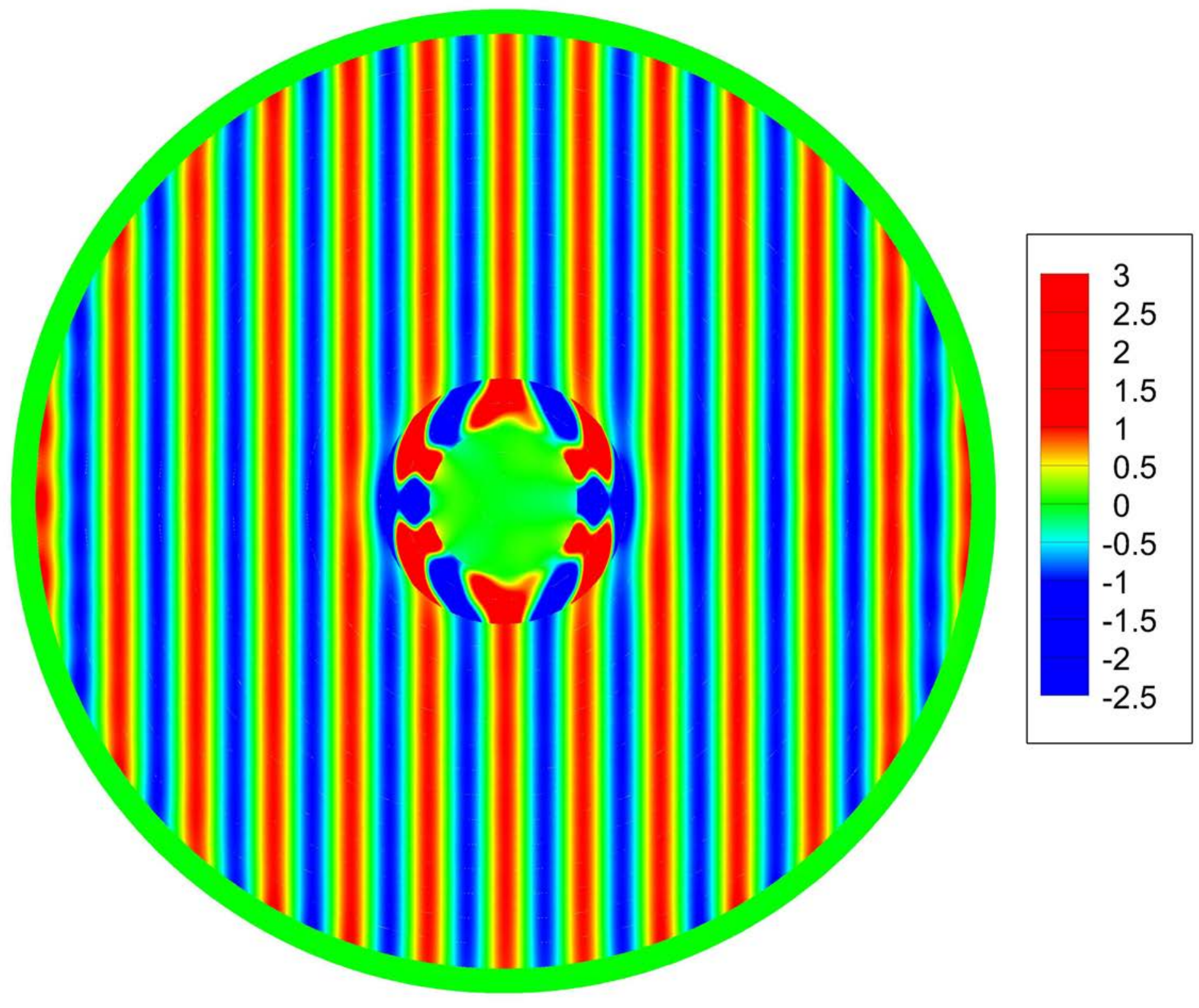}}
	\caption{\small Contours of the approximated $D_z$ in the $XY$ plane at different time steps with $R_2<2R_1$.}
	\label{SemC}
\end{figure}
\begin{figure}[ht!]
	{~}\hspace*{-16pt}
	\centering
	\subfigure[t=1]{\includegraphics[scale=.2]{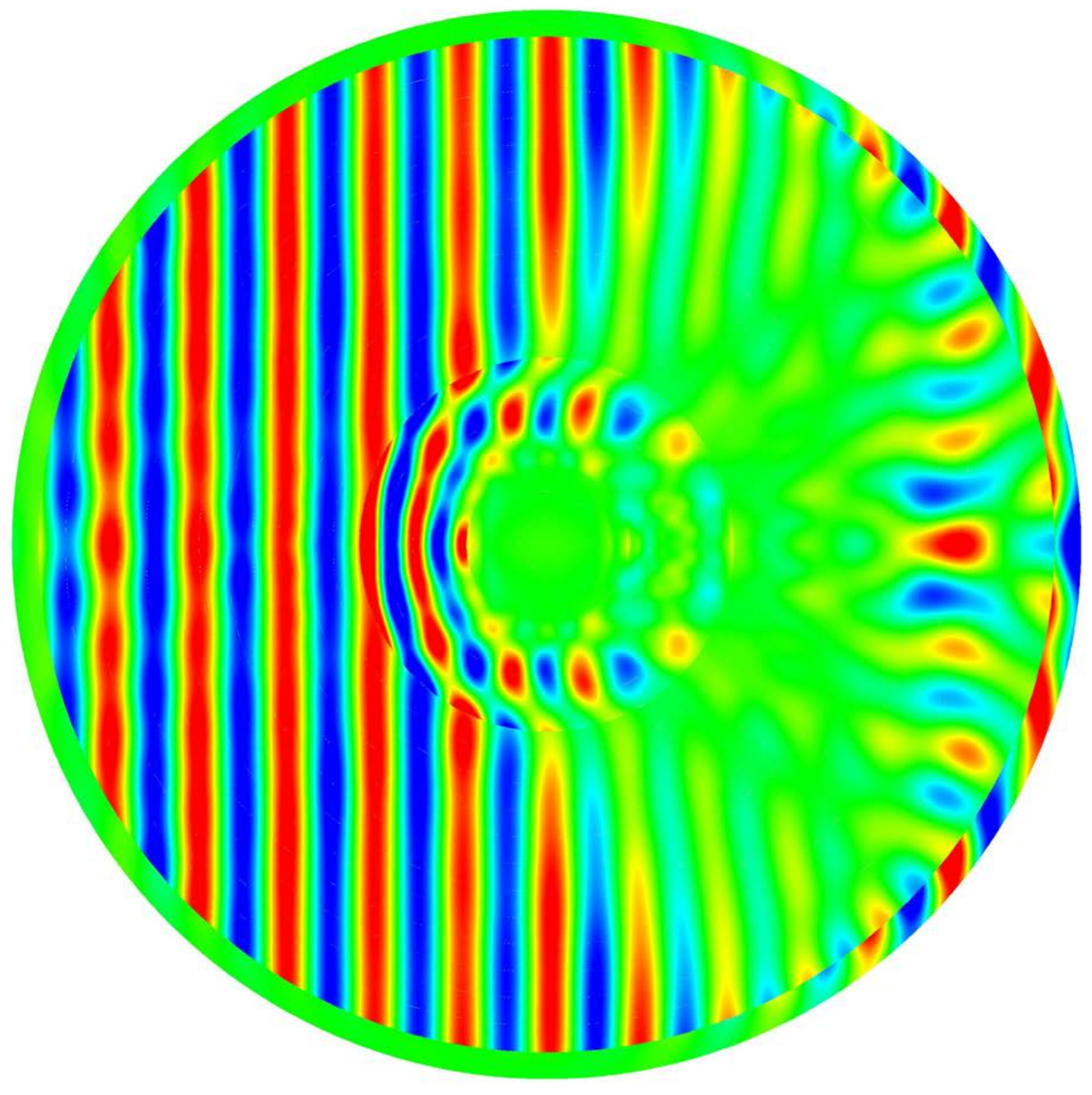}}
	\subfigure[t=3]{ \includegraphics[scale=.2]{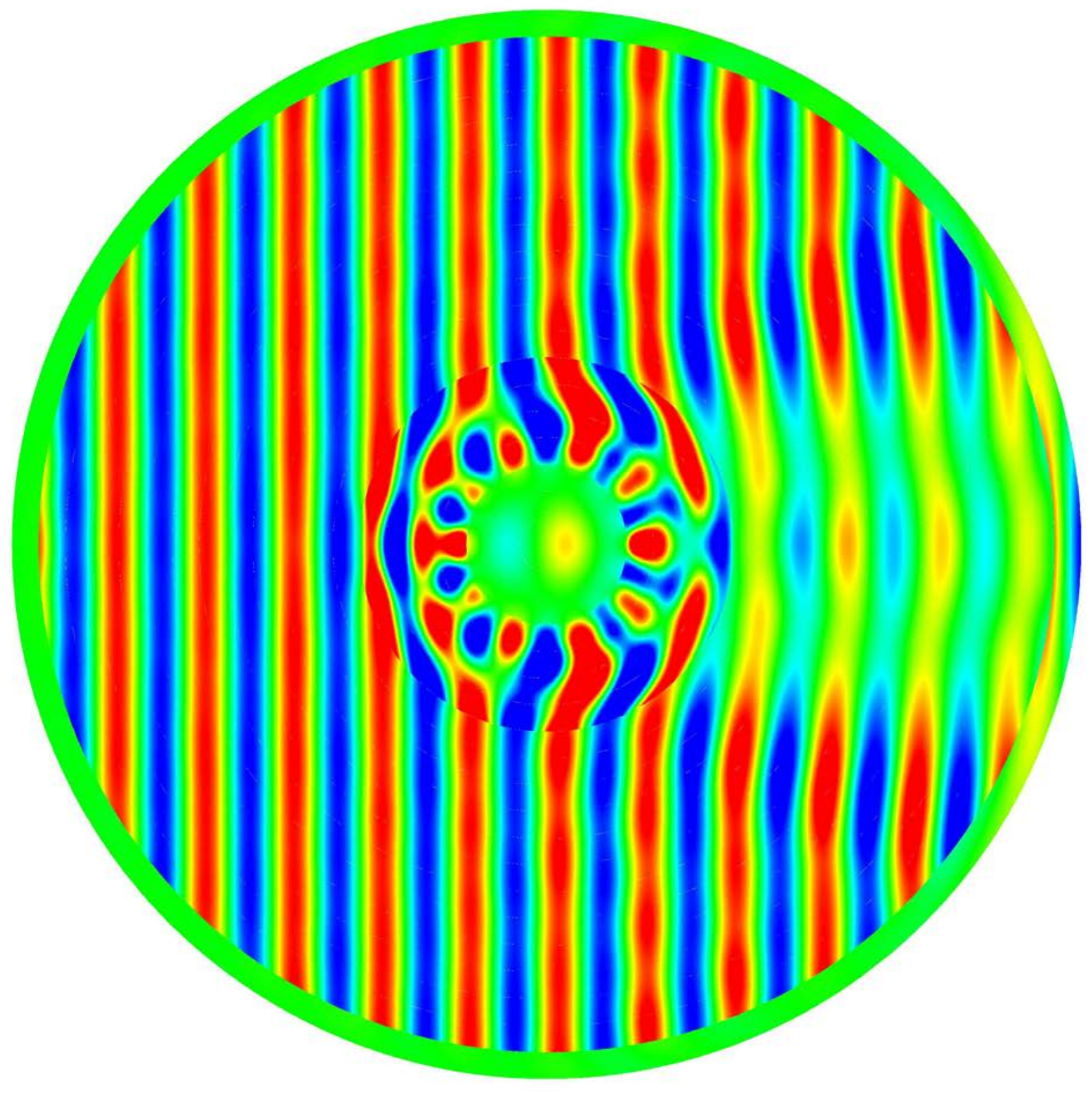}}
	\subfigure[t=5]{ \includegraphics[scale=.2]{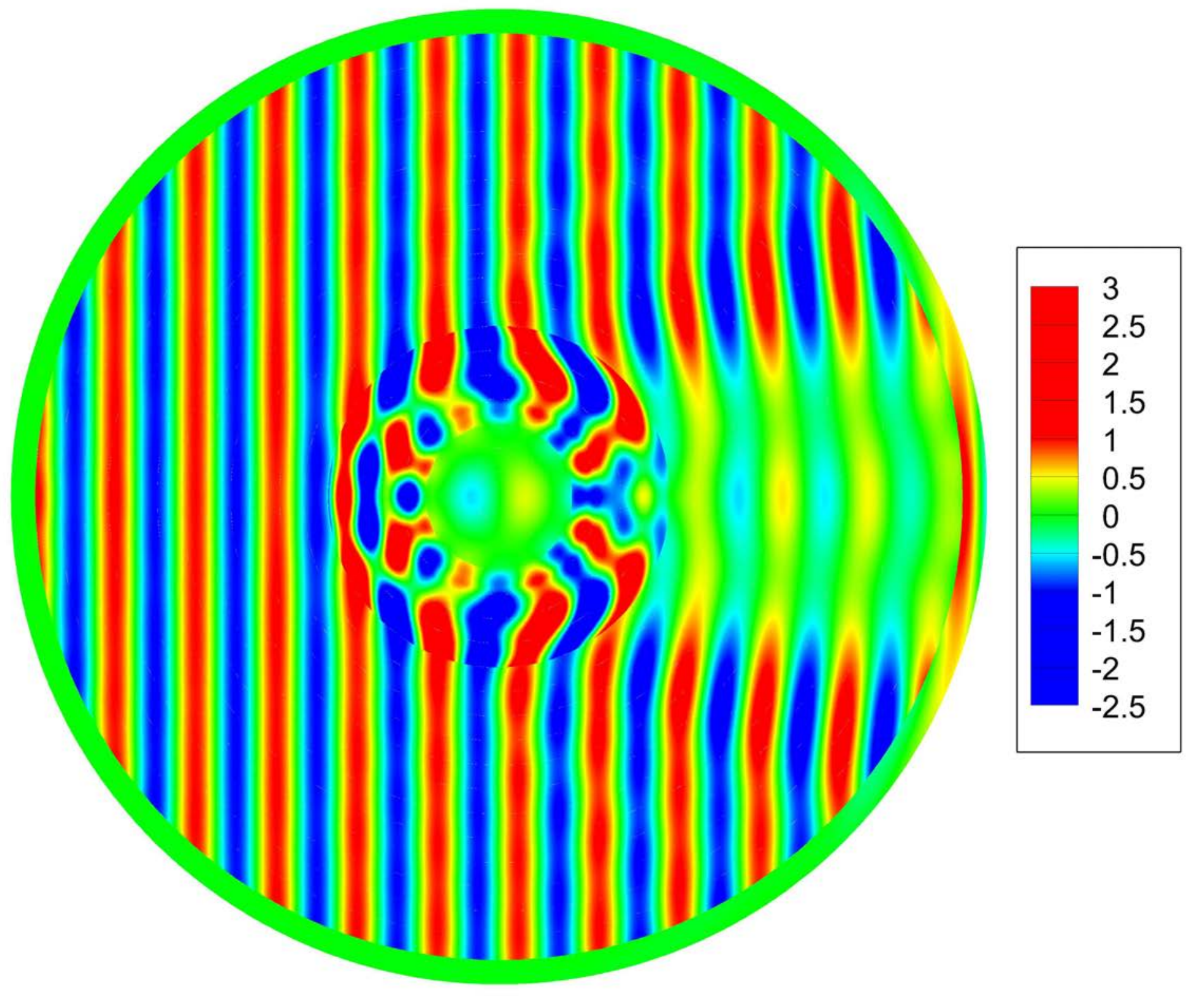}}\\
	\subfigure[t=7]{ \includegraphics[scale=.2]{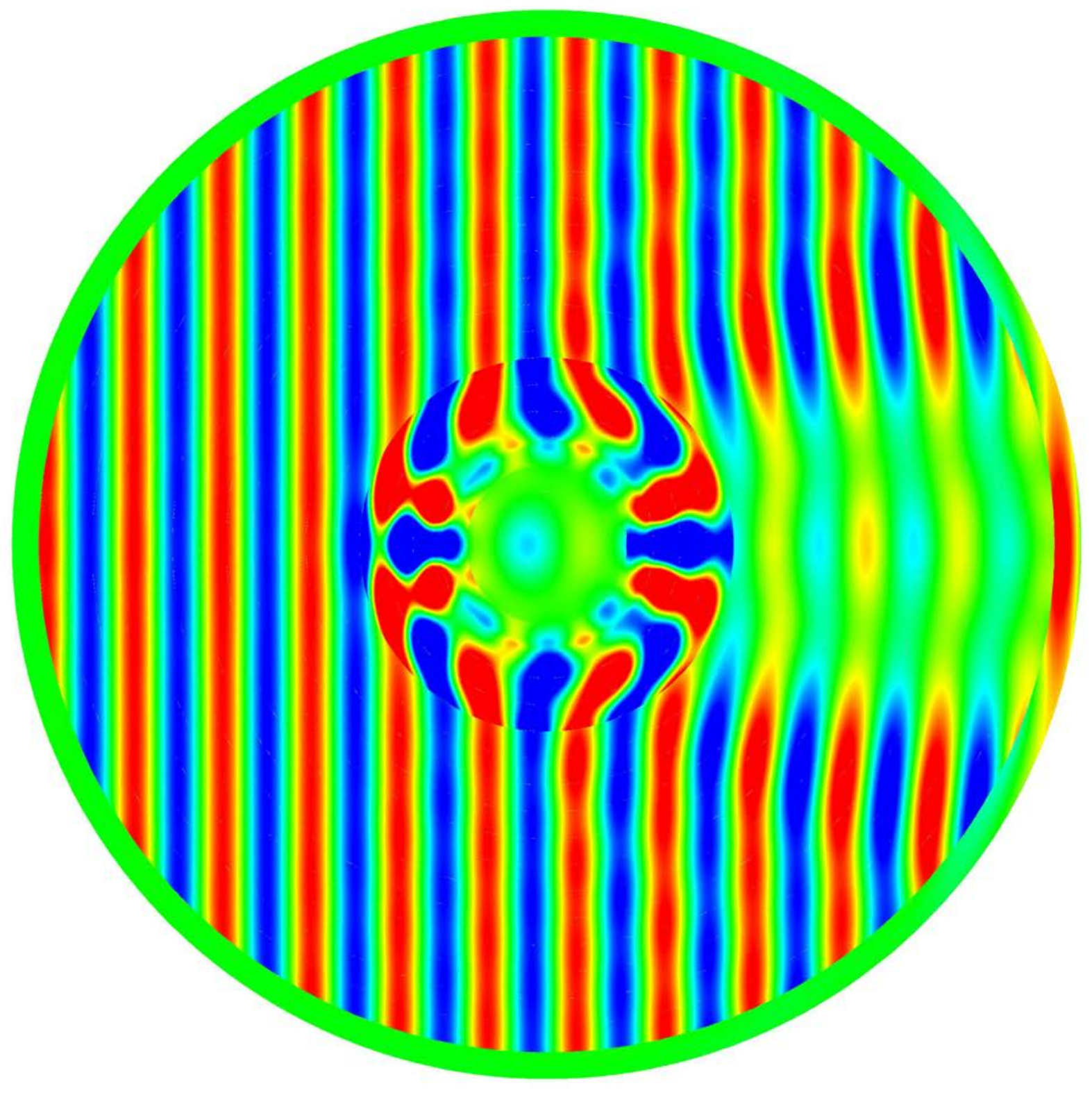}}
	\subfigure[t=9]{ \includegraphics[scale=.2]{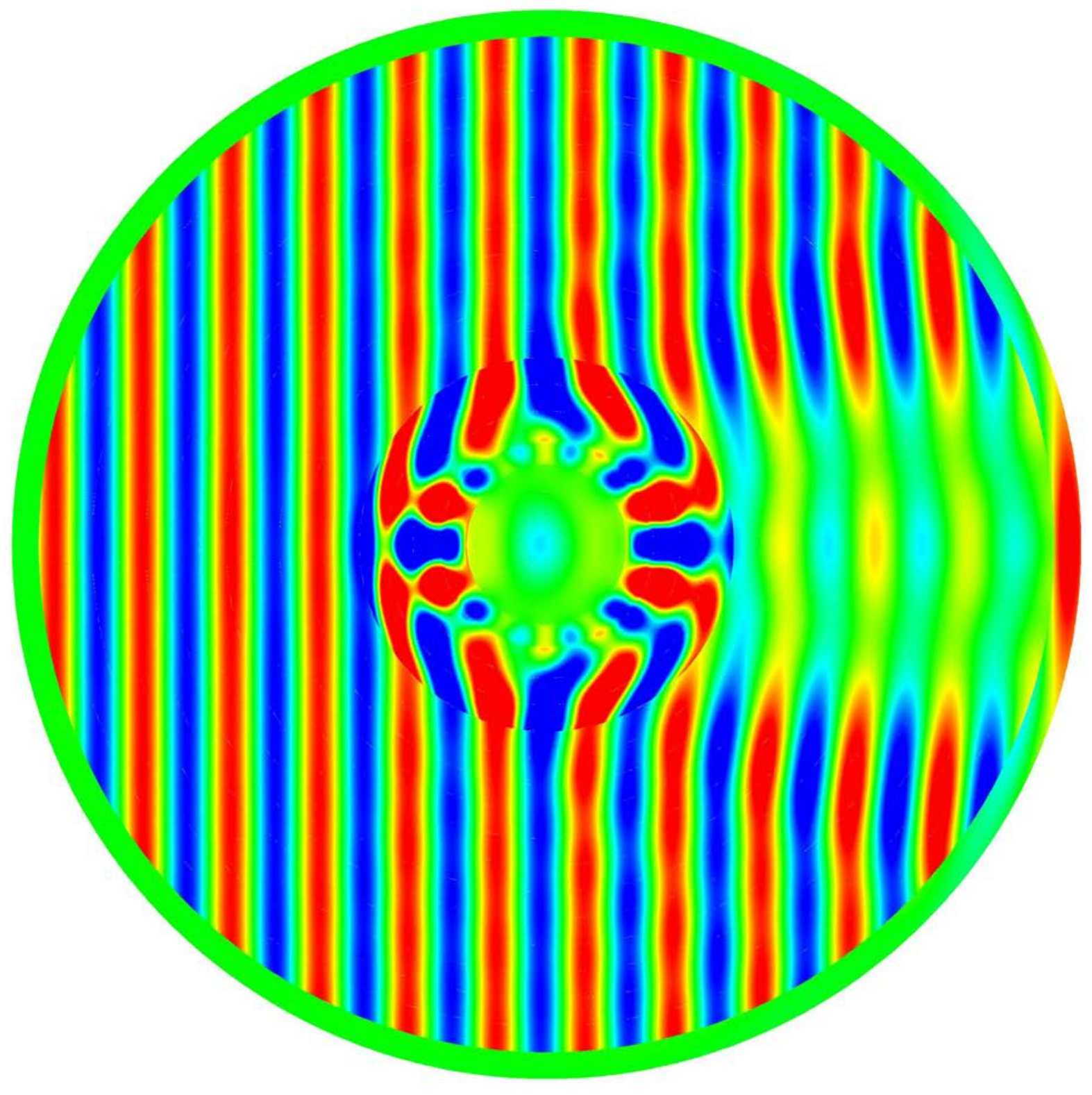}}
	\subfigure[t=11]{ \includegraphics[scale=.2]{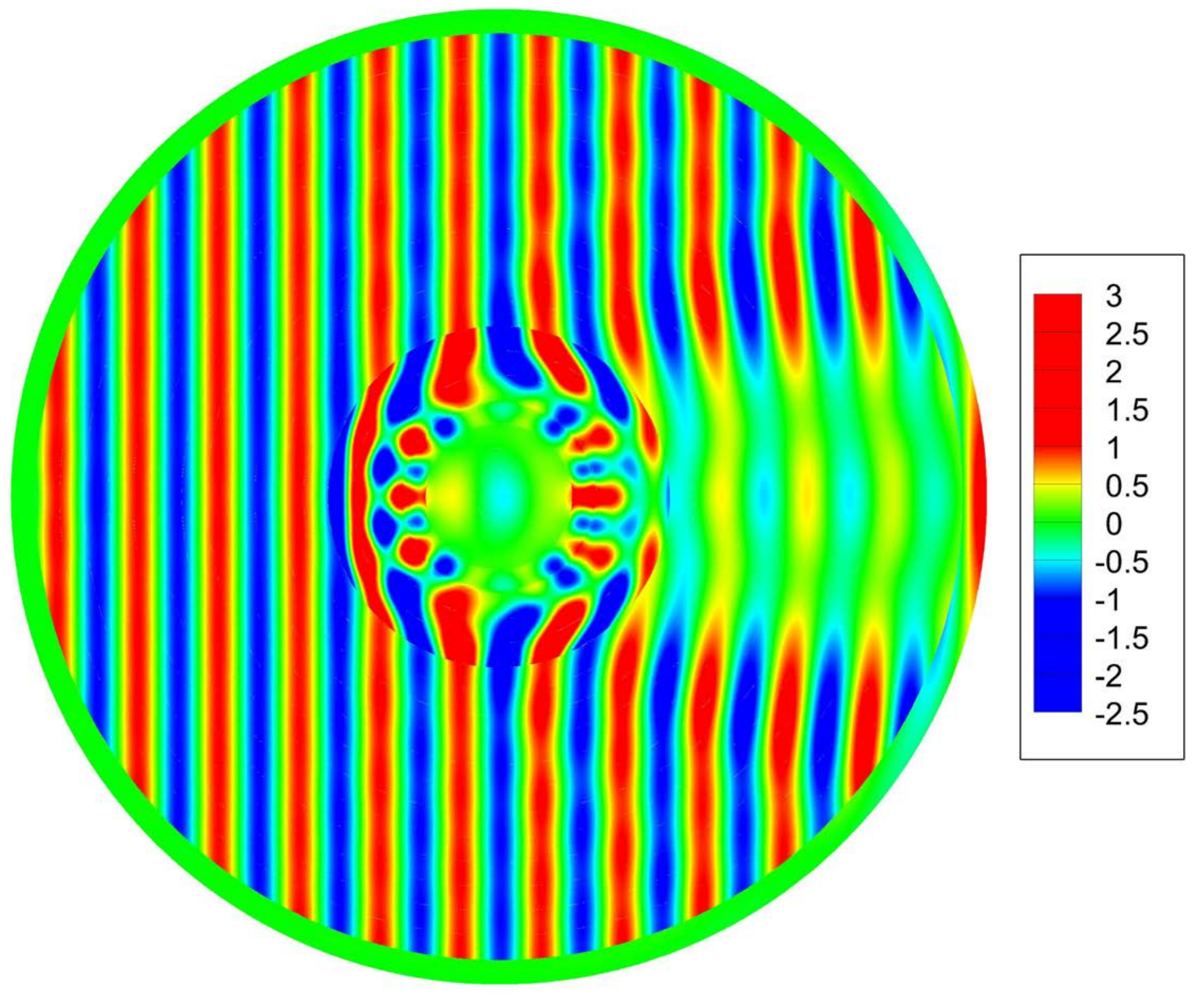}}
	\caption{\small Contours of the approximated $D_z$ in the $XY$ plane at different time steps with $k=38\neq \omega_c$.}
	\label{example4fig}
\end{figure}

As discussed in \cite{zhao2009full}, the cloak is relatively sensitive to the frequency of the incident wave. Here we use the monochromatic incident wave \eqref{inc1} with $k=\omega=38$ to test the cloak device with parameters: $\omega_c=40$, $\gamma_1=\gamma_2=0.001$, $R_1=0.15$, $R_2=0.35$. The contours of the numerical $D_z$ at different time are plotted in Figure \ref{example4fig}. The numerical results show that there are waves propagating inside the cloaking  region.

\subsection{Polychromatic incident wave}
In this example, we use a pulse of plane wave given by
\begin{equation*}
\bs D^{\rm in}=\mathfrak{Re}\{e^{\ri k(x-t)}\}e^{-\frac{(x-t+t_c)^2}{q}}\bs A,\quad \bs A:=\begin{bmatrix}
0 & 0 & A
\end{bmatrix}^{\rm T},
\end{equation*}
as the incident wave with $A=1$, $k=40$, $t_c=4$ and $q=0.5$. We first consider the cloaking device with parameters given by $\omega_c=40$, $\gamma_1=\gamma_2=0.001$, $R_1=0.15$, $R_2=0.35>2R_1$. The contours of $D_z$ at different time are plotted in Figure \ref{planewavepulse1}. Then, the outer radius of the cloak is set to $R_2=0.25<2R_1$ and other parameters remain unchanged. The contours of $D_z$ at different time are plotted in Figure \ref{planewavepulse2}. In these tests, there are polychromatic EM waves interacting with the cloaking devices. We can see from the numerical results that there are waves propagating inside the cloaked region.
\begin{figure}[h!]
	\centering
	{~}\hspace*{-16pt} \subfigure[t=3]{ \includegraphics[scale=.18]{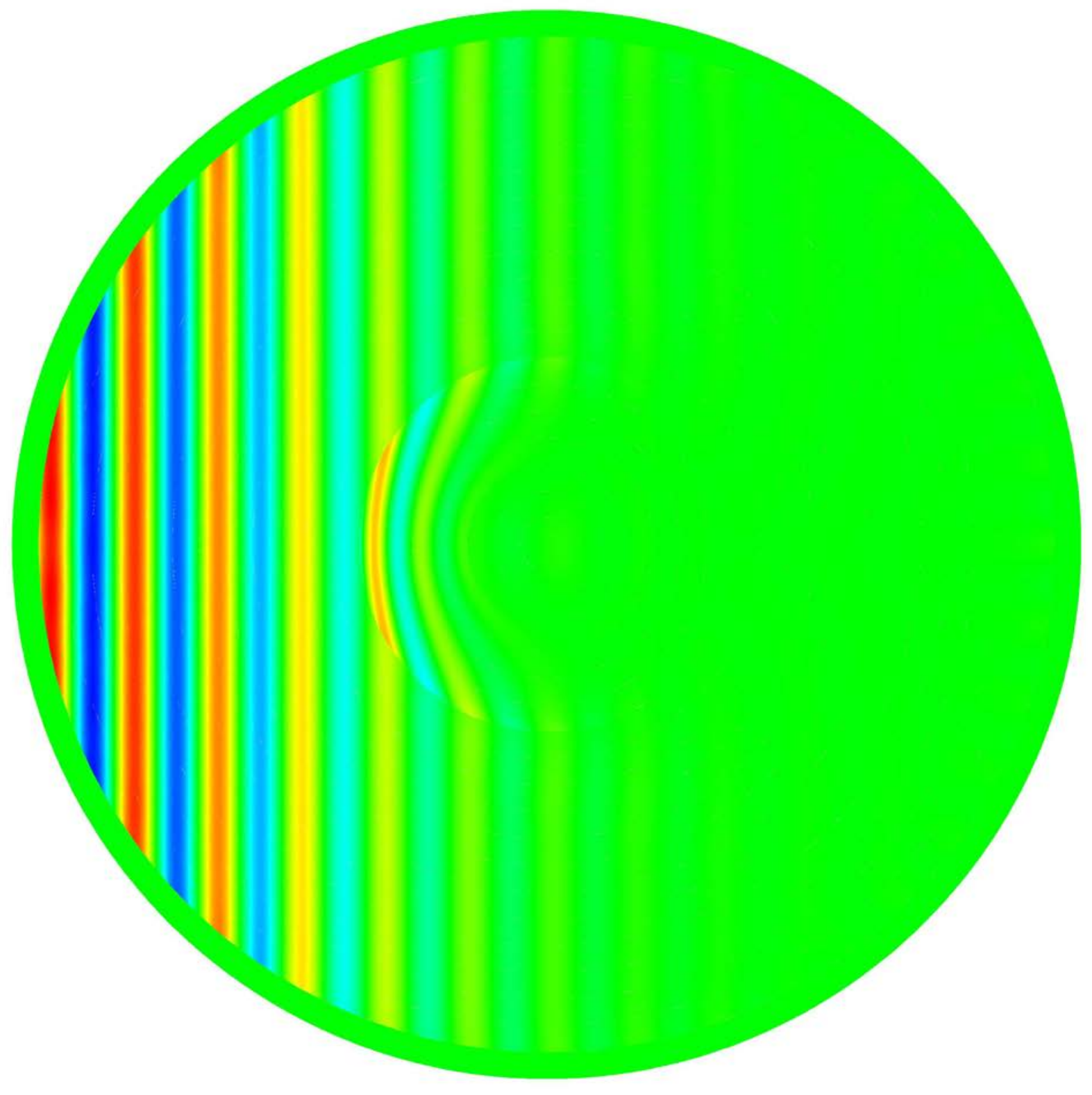}}
	\subfigure[t=3.5]{\includegraphics[scale=.18]{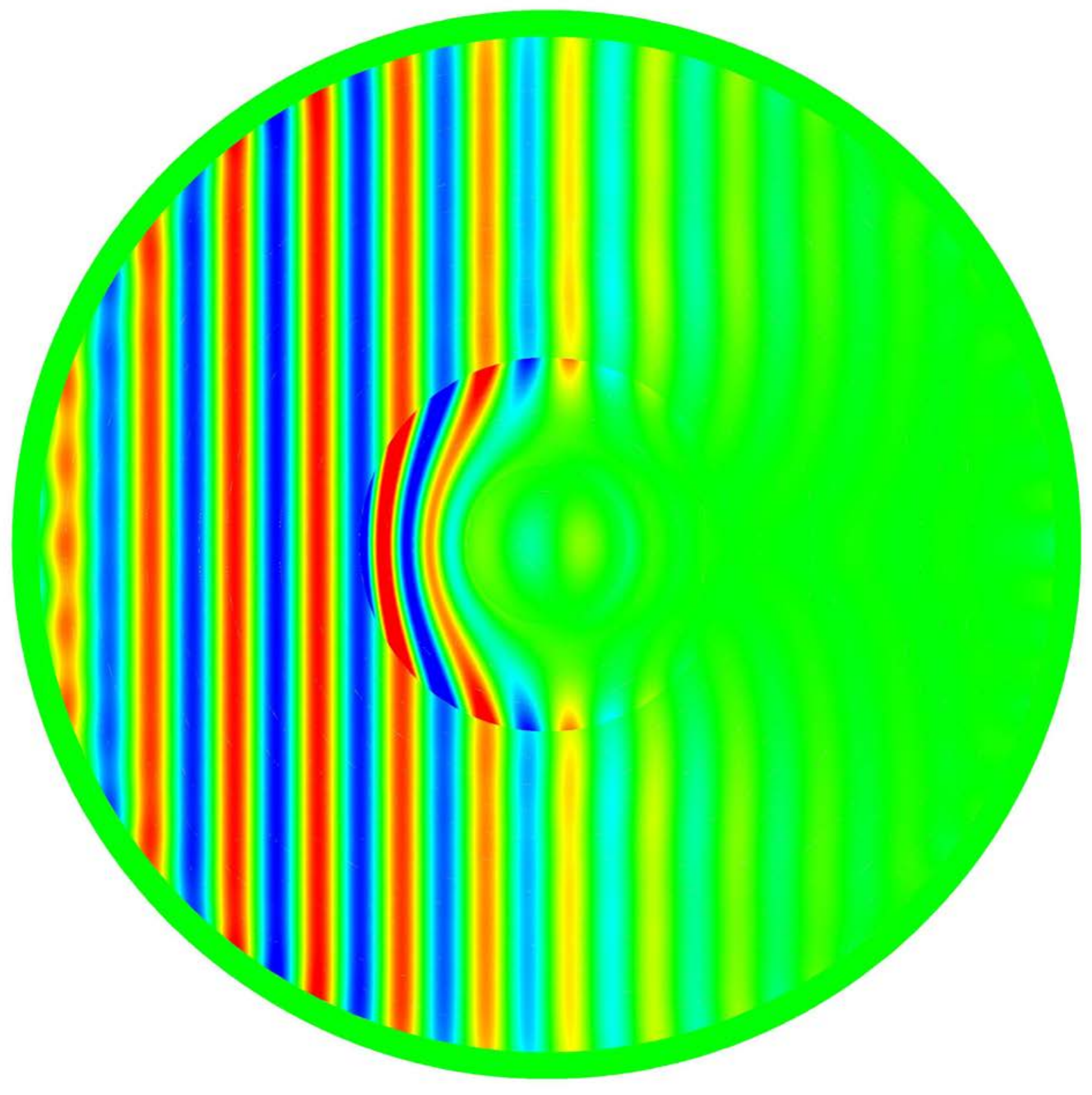}}
	\subfigure[t=4]{ \includegraphics[scale=.18]{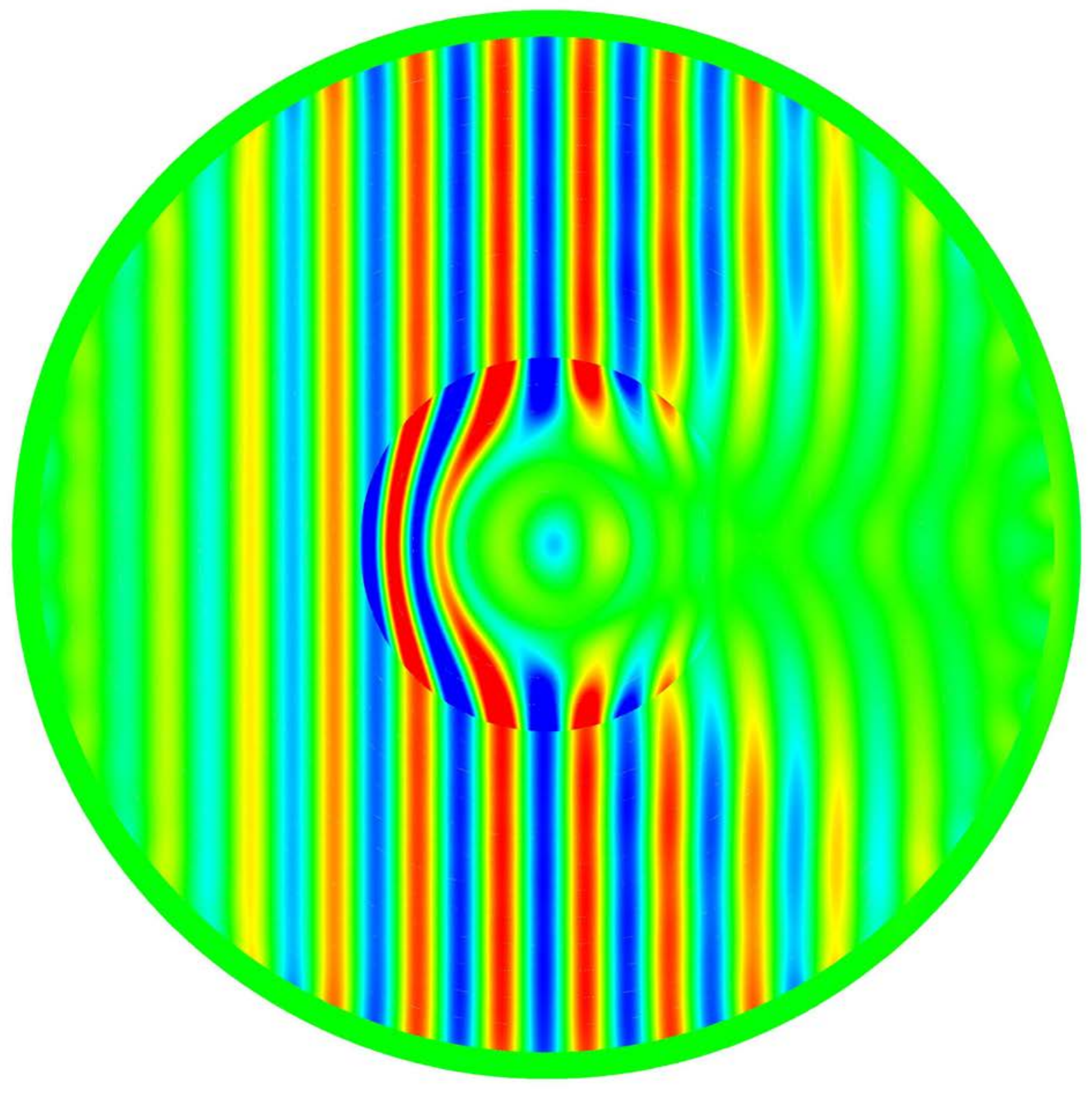}}
	\subfigure[t=4.5]{ \includegraphics[scale=.18]{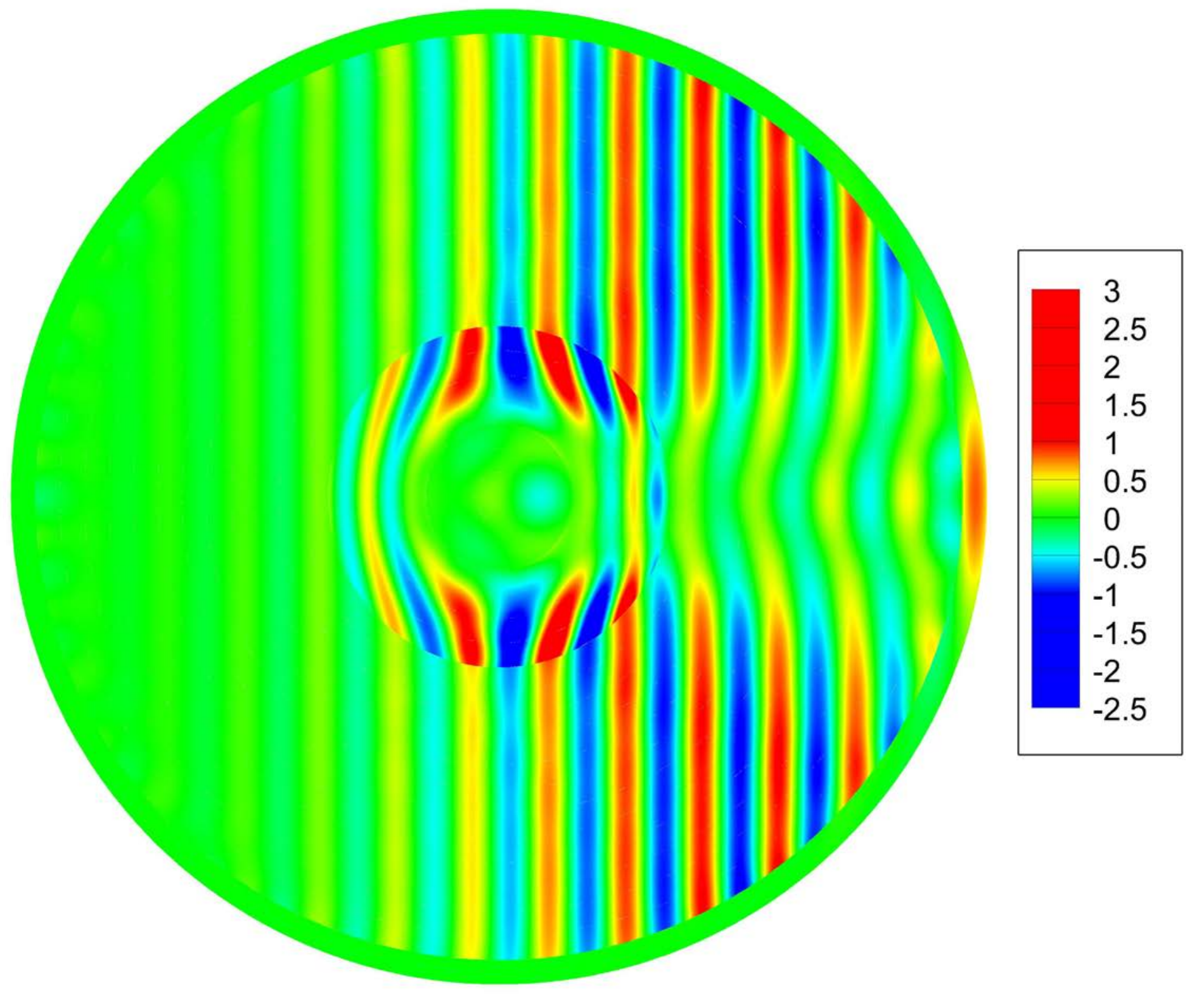}}
	\subfigure[t=5]{ \includegraphics[scale=.18]{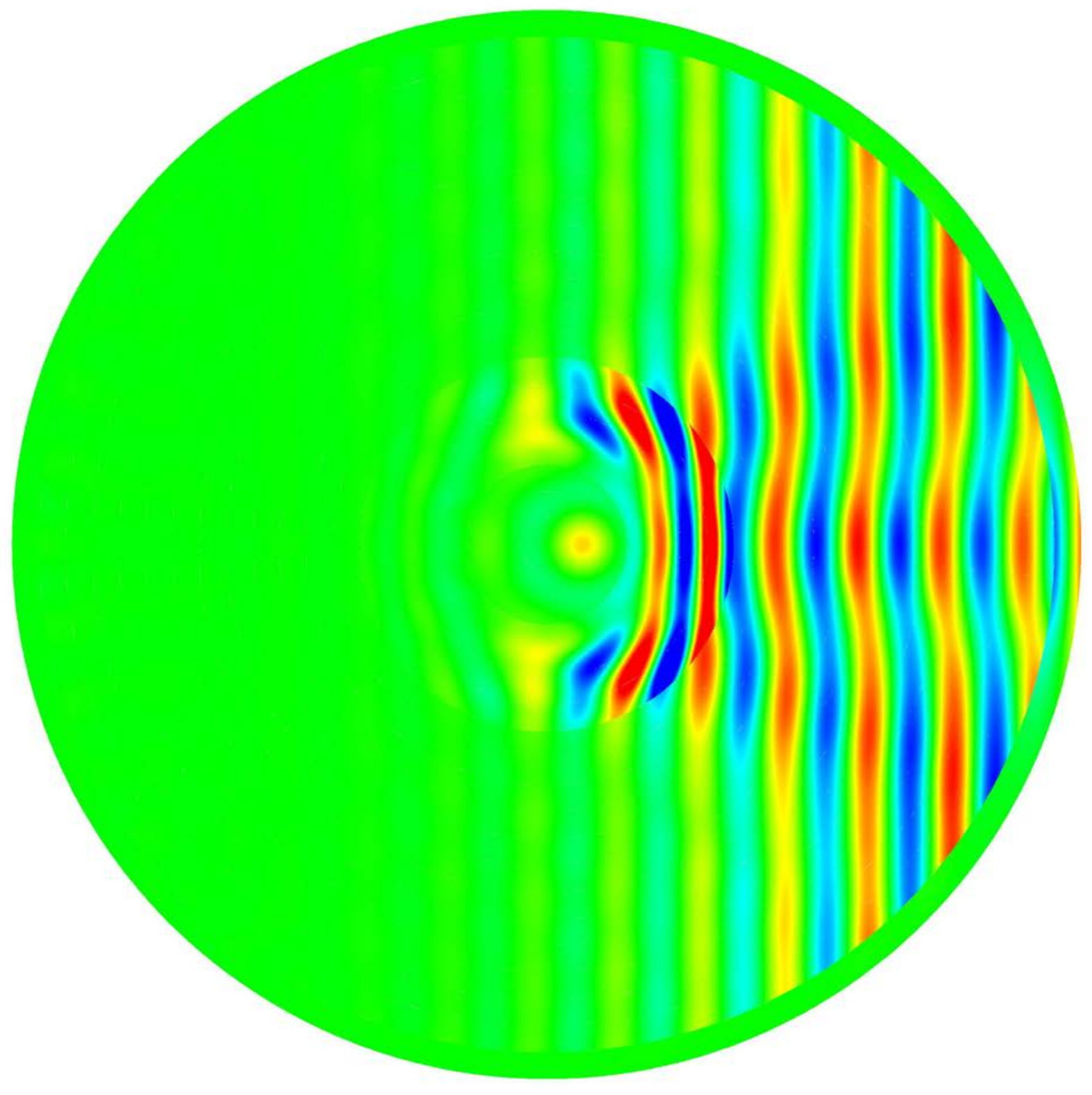}}
	\subfigure[t=5.5]{\includegraphics[scale=.18]{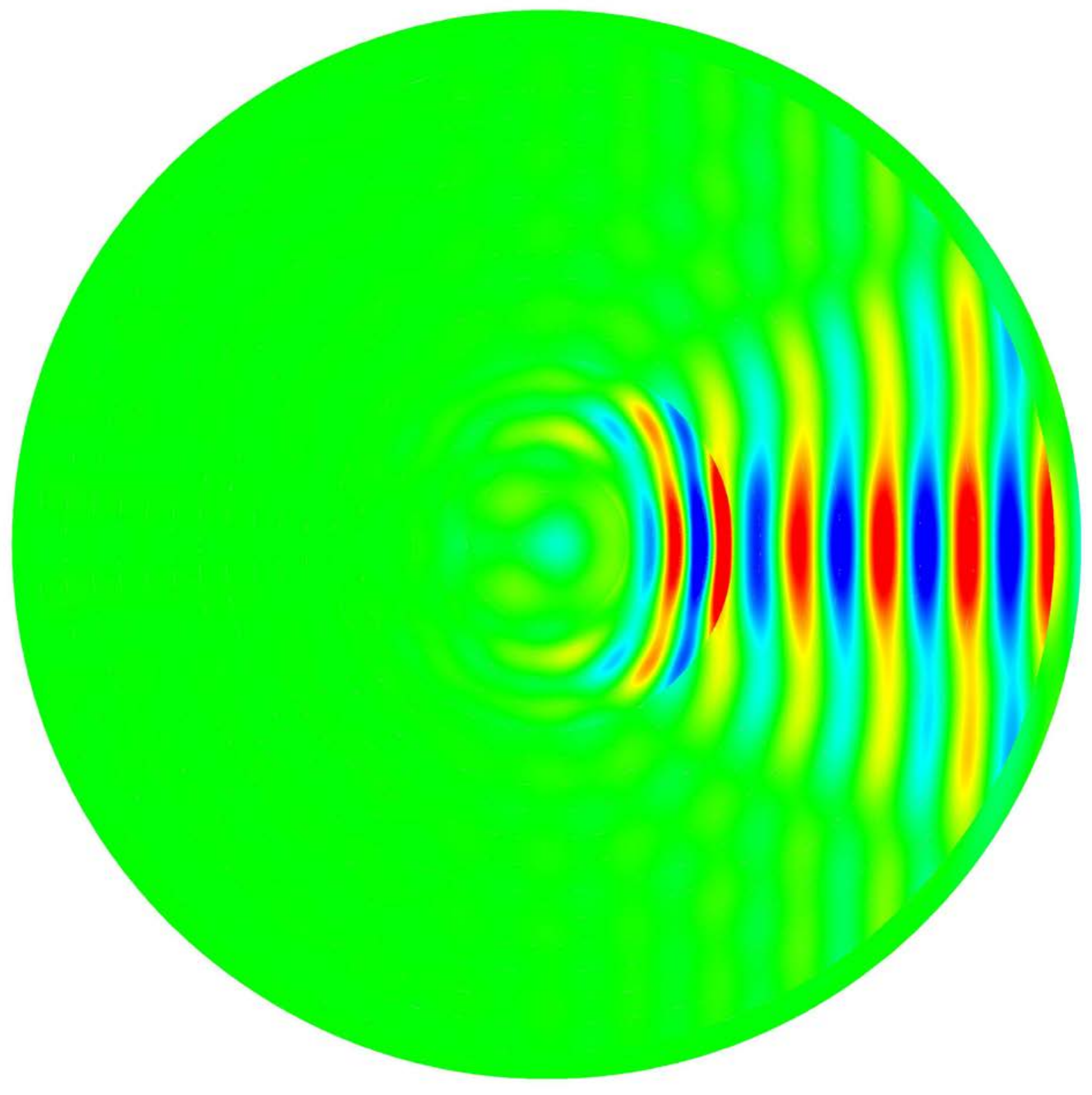}}
	\subfigure[t=6]{ \includegraphics[scale=.18]{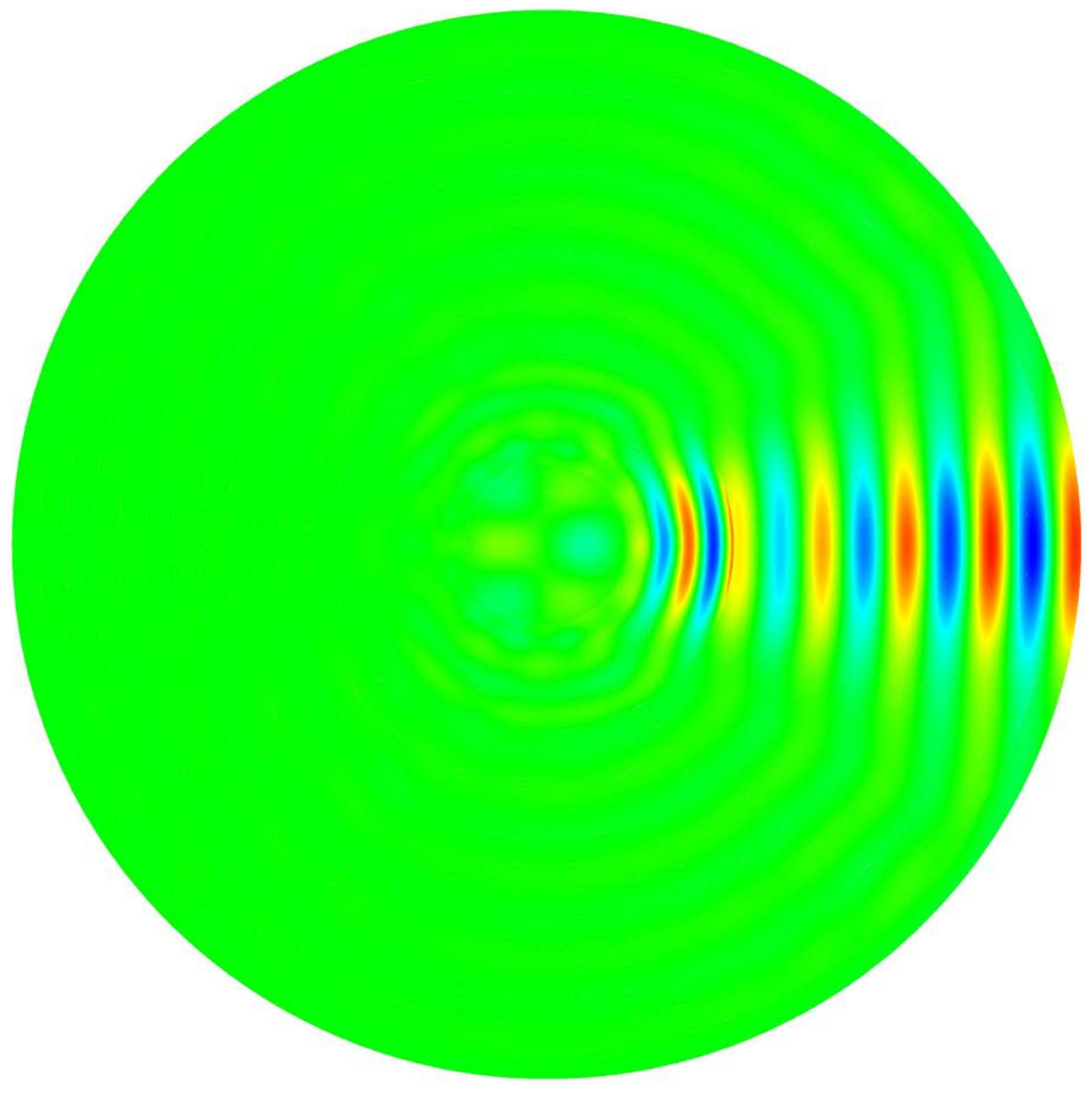}}
	\subfigure[t=6.5]{ \includegraphics[scale=.18]{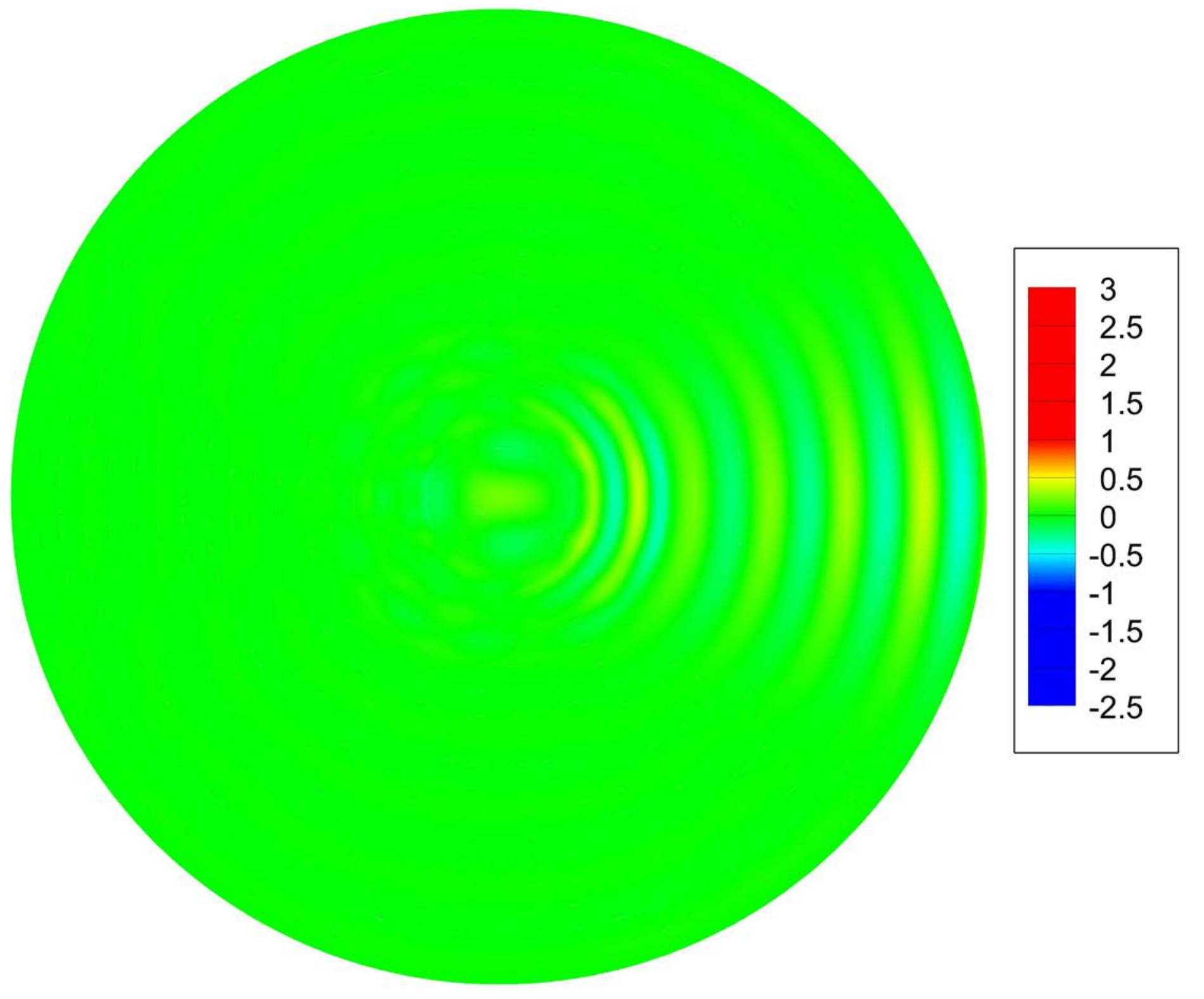}}
	\caption{\small Contours of the approximated $D_z$ in the $XY$ plane at different time steps with $R_2>2R_1$.}
	\label{planewavepulse1}
\end{figure}

\begin{figure}[h!]
	\centering
	{~}\hspace*{-16pt} \subfigure[t=3]{ \includegraphics[scale=.18]{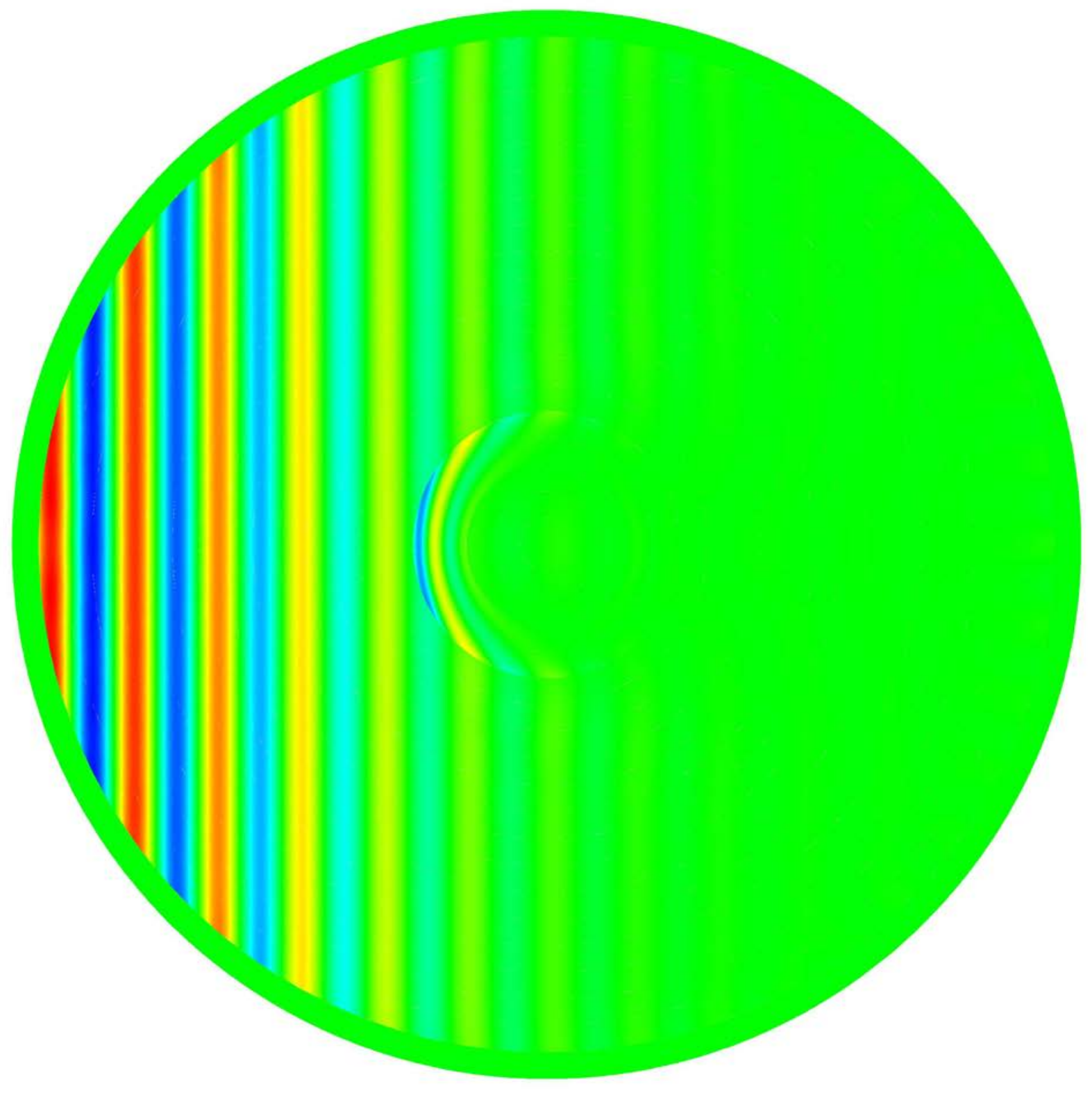}}
	\subfigure[t=3.5]{\includegraphics[scale=.18]{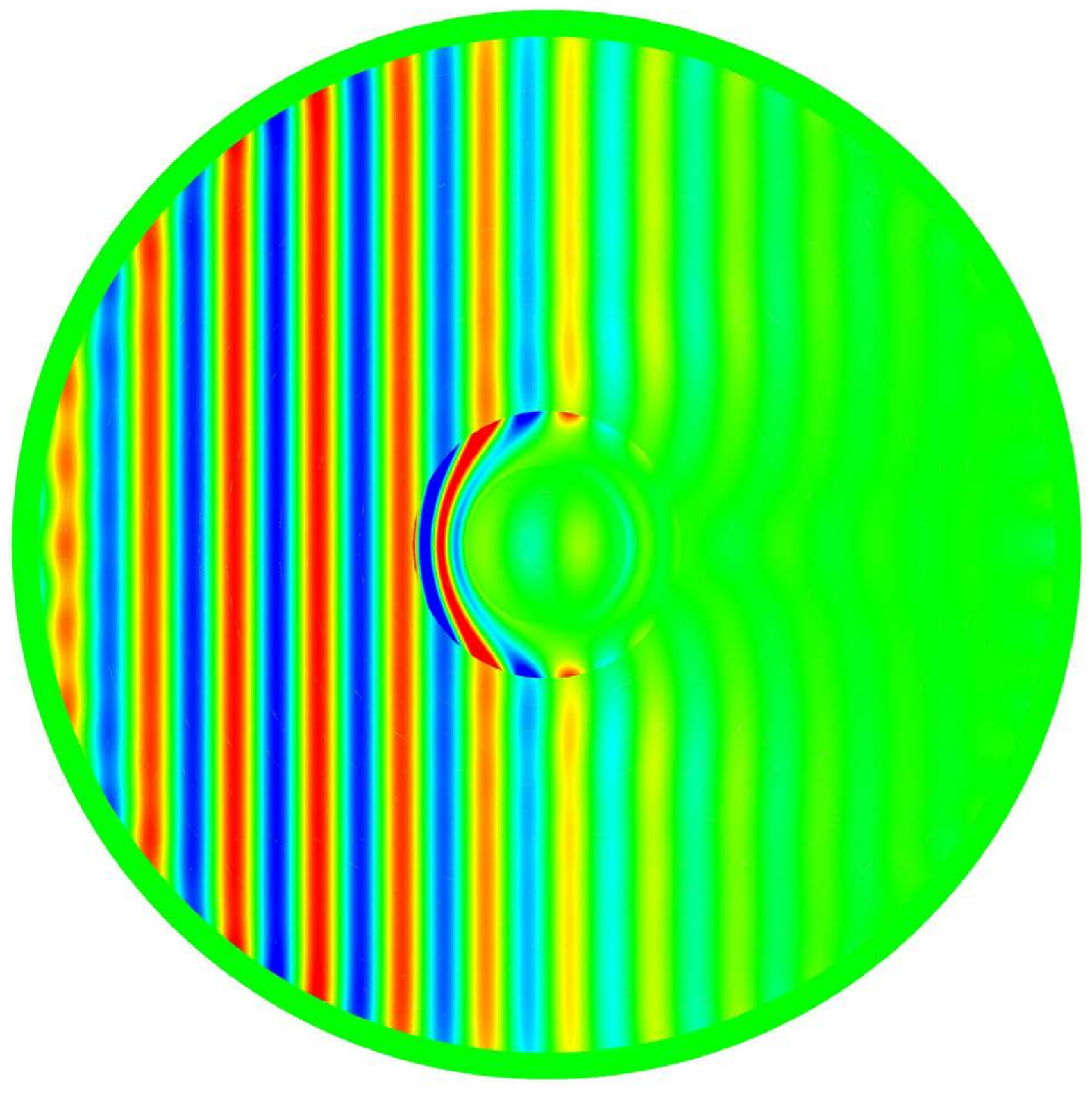}}
	\subfigure[t=4]{ \includegraphics[scale=.18]{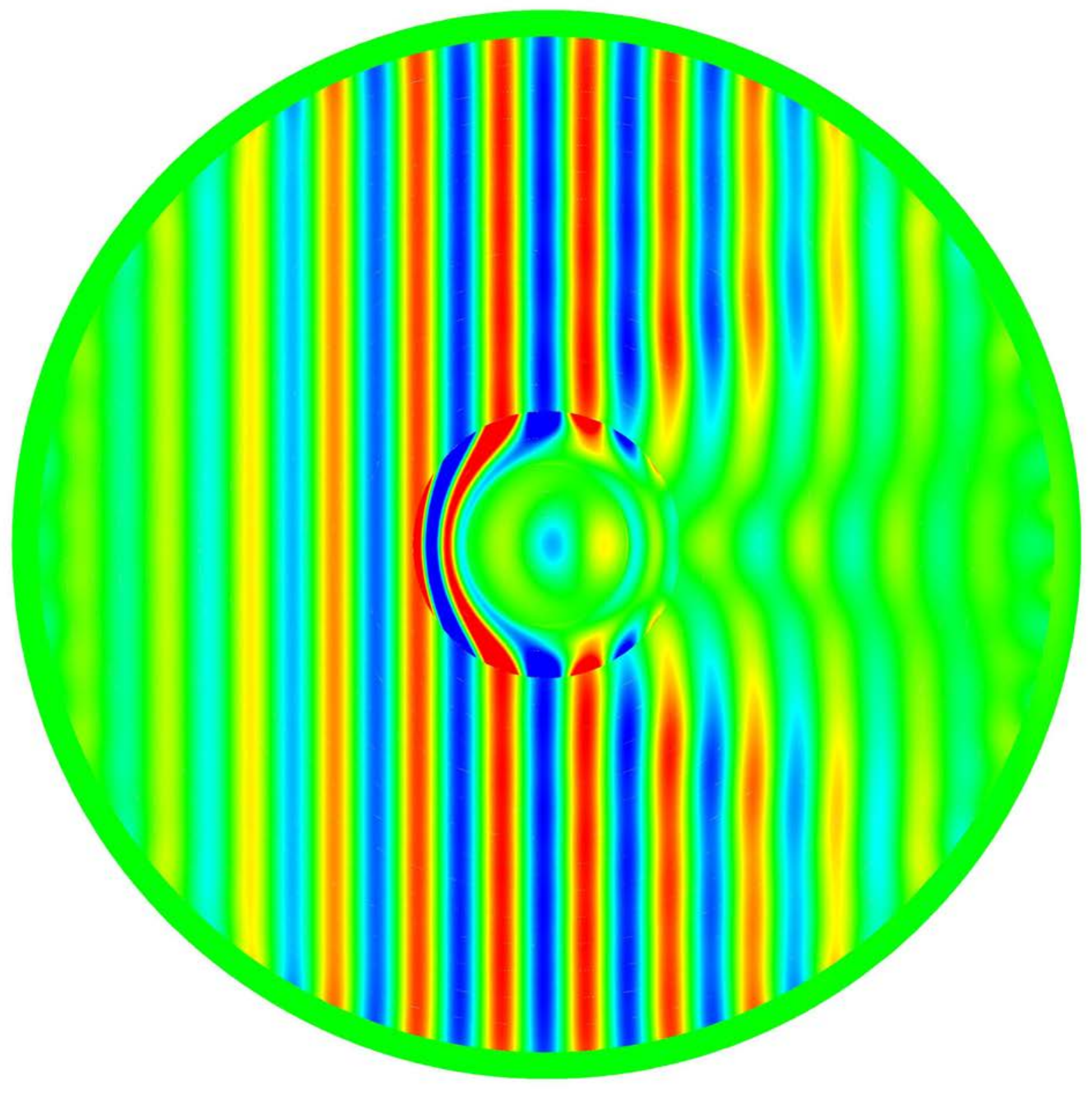}}
	\subfigure[t=4.5]{ \includegraphics[scale=.18]{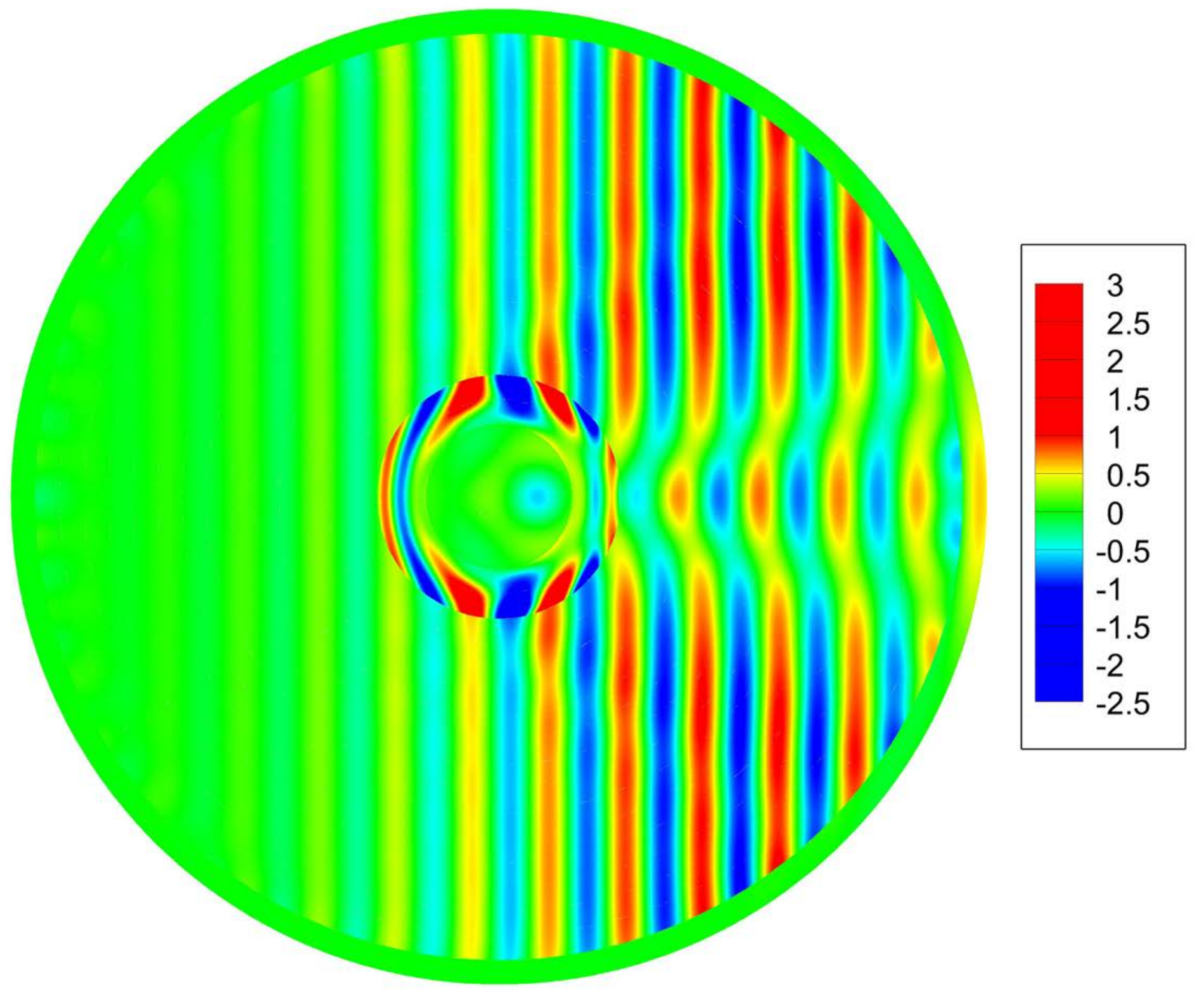}}
	\subfigure[t=5]{ \includegraphics[scale=.18]{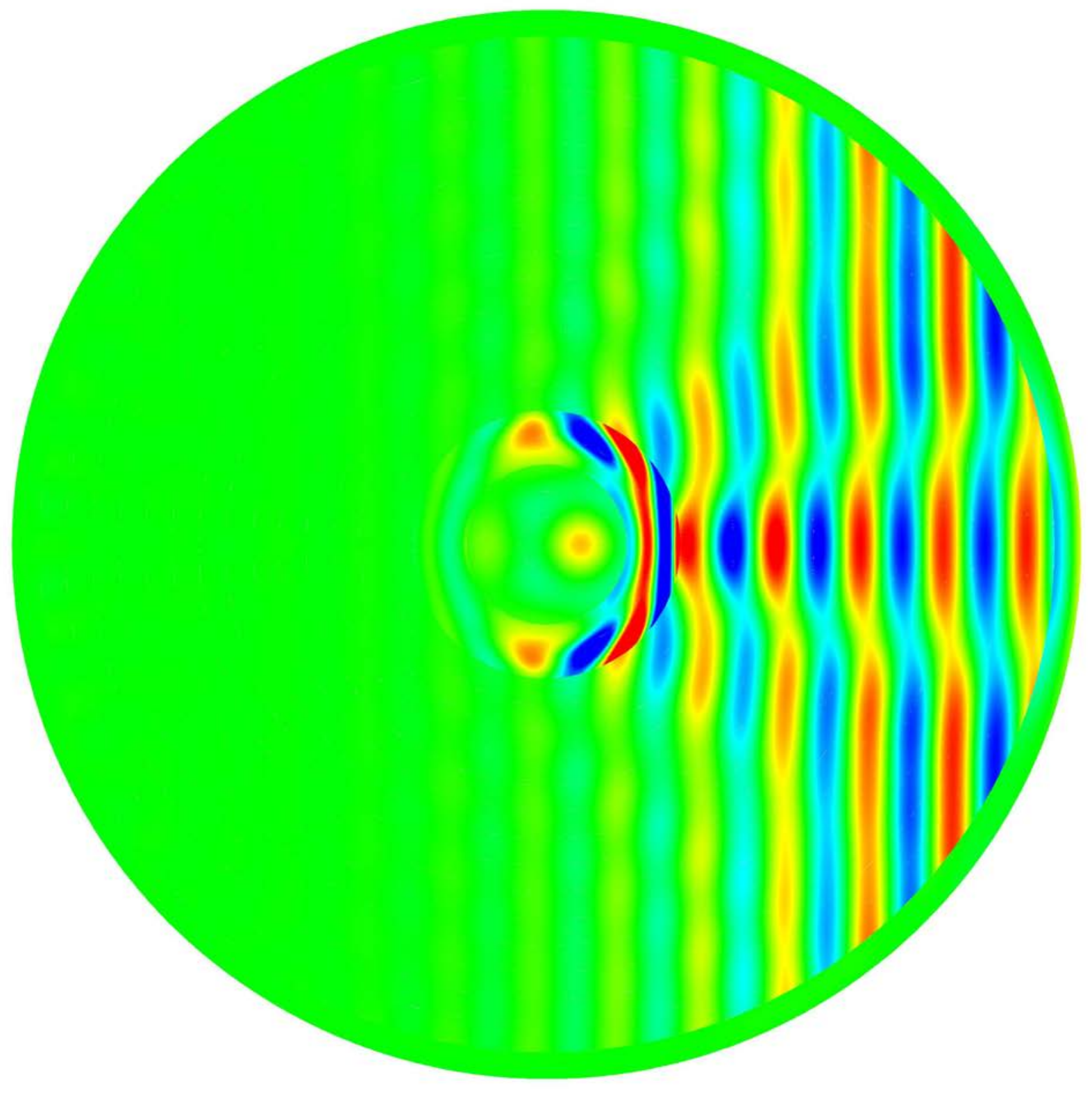}}
	\subfigure[t=5.5]{\includegraphics[scale=.18]{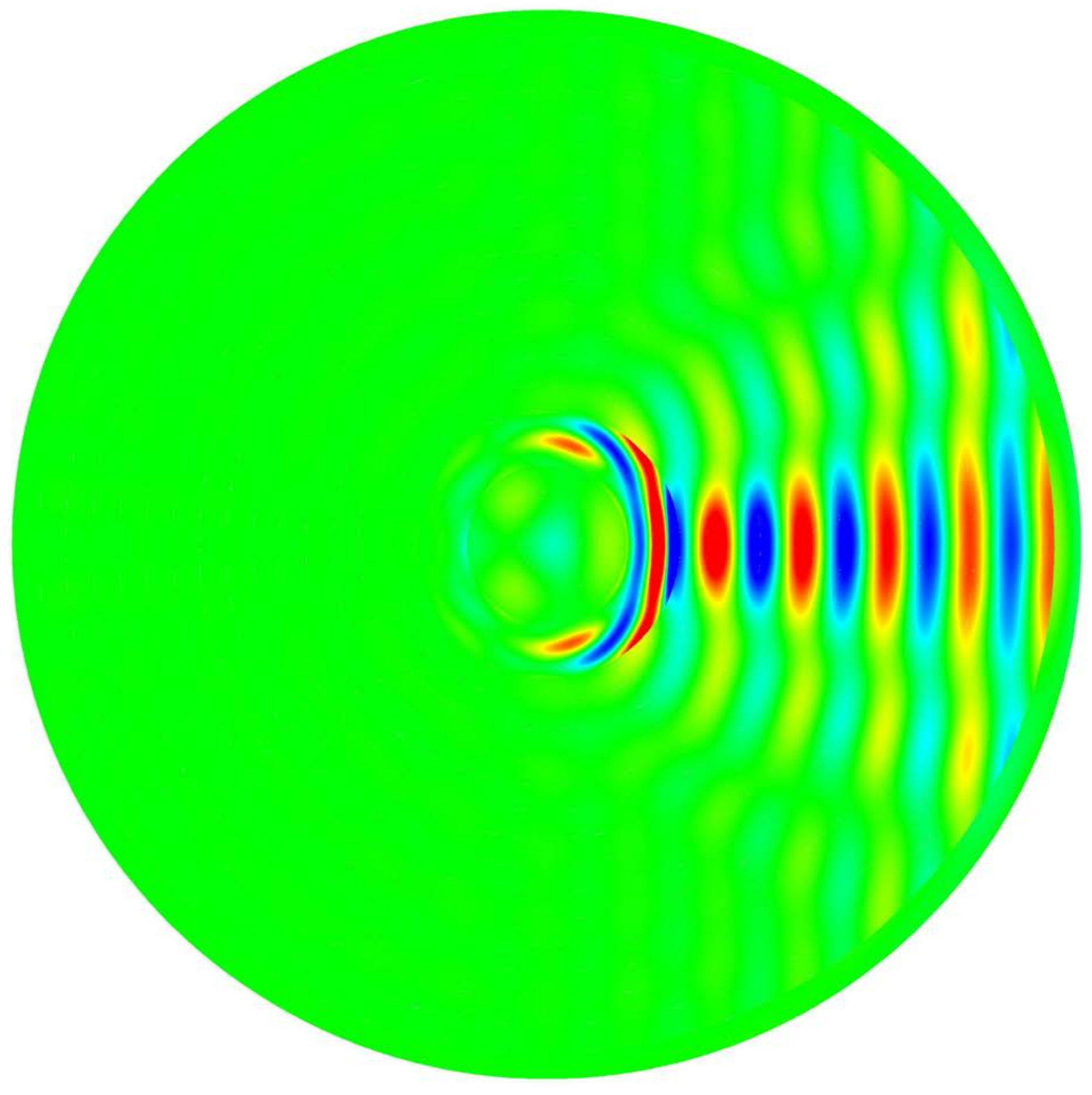}}
	\subfigure[t=6]{ \includegraphics[scale=.18]{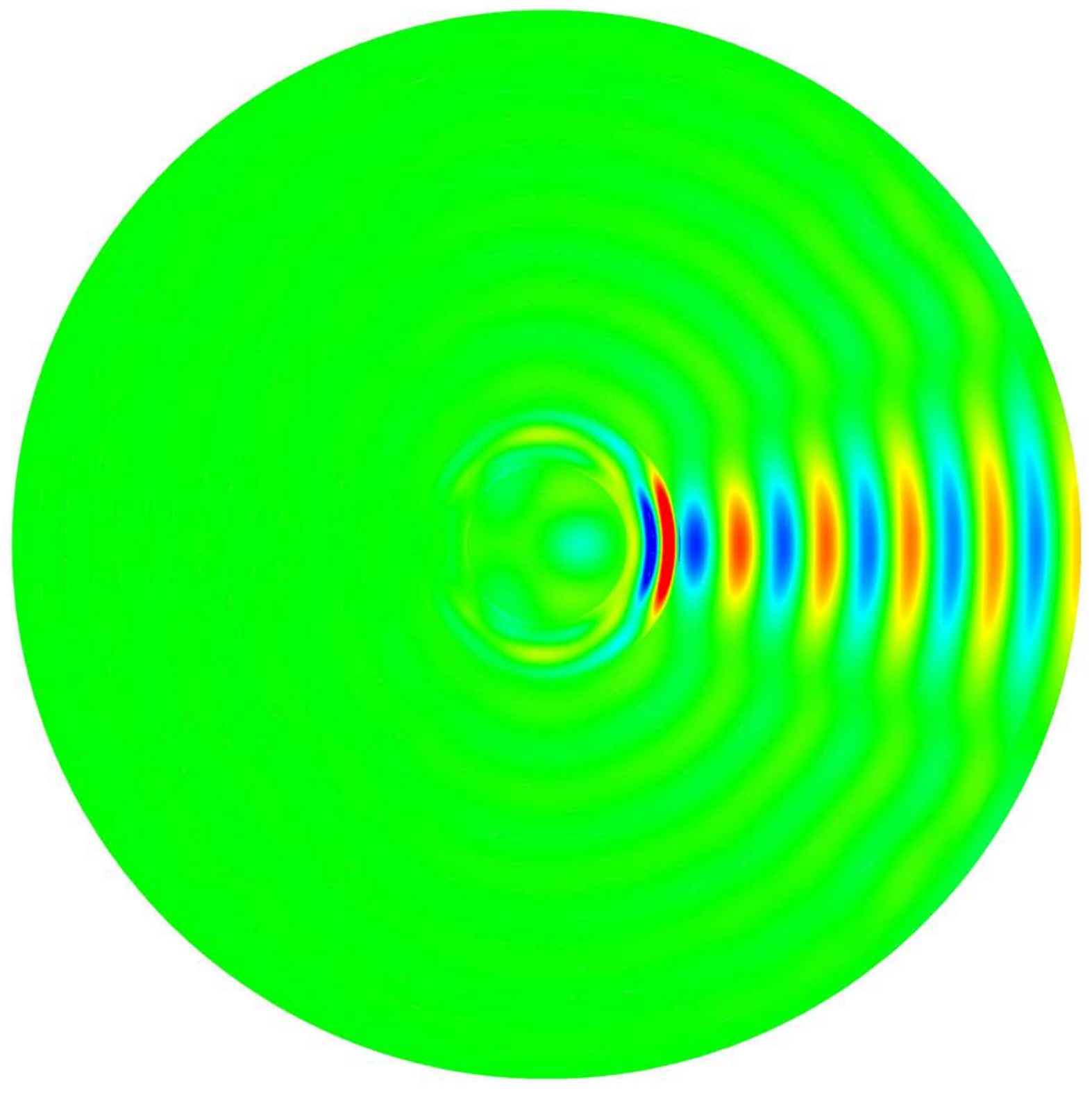}}
	\subfigure[t=6.5]{ \includegraphics[scale=.18]{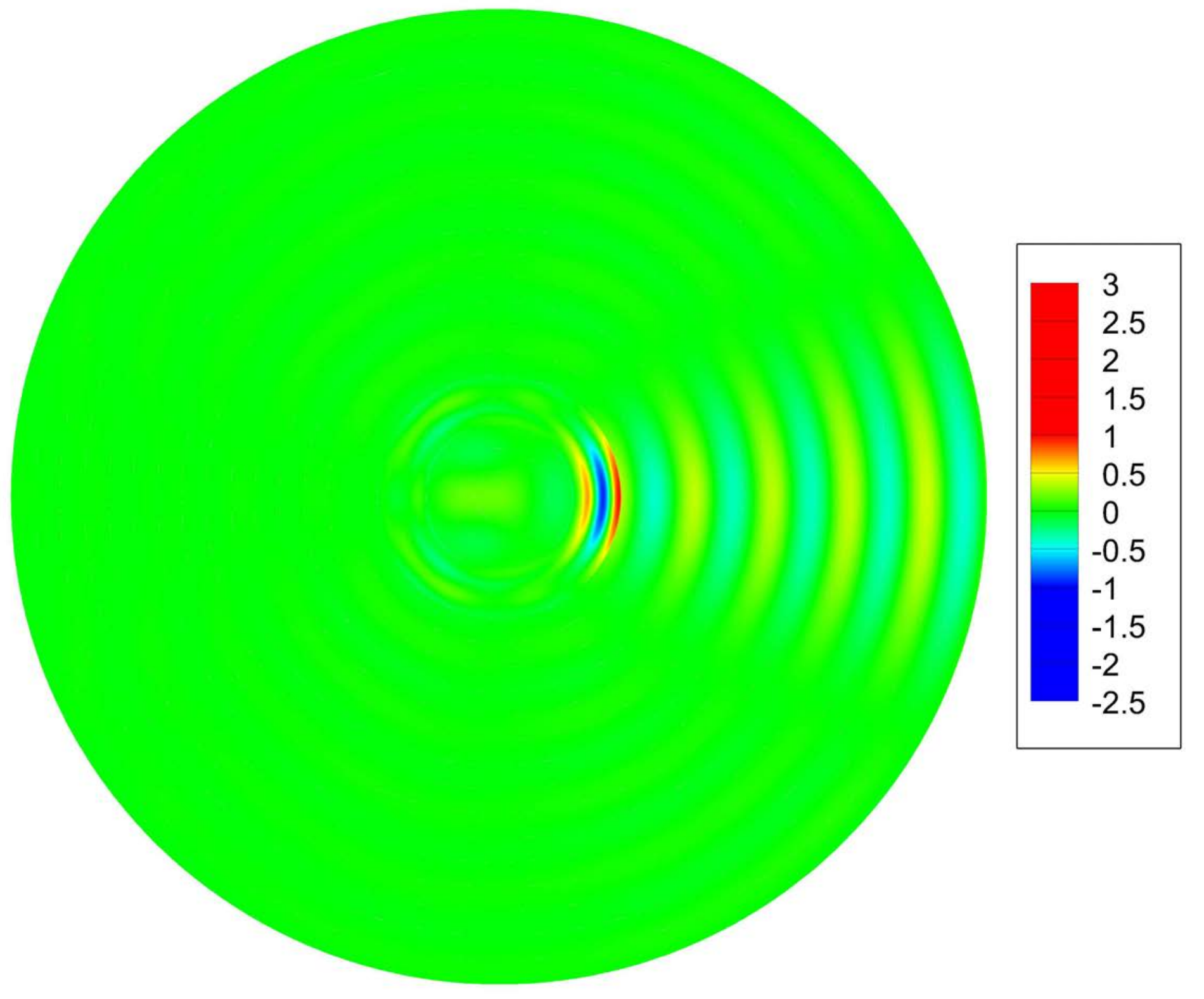}}
	\caption{\small Contours of the approximated $D_z$ in the $XY$ plane at different time steps with $R_2<2R_1$.}
	\label{planewavepulse2}
\end{figure}
%\subsection{Point Source} In this example, we add an external source, compactly supported in $\Omega_3$ as the wave maker, and turn off the incident wave. The source is given by Guassian function
%\begin{equation}
%f(\bs r, t)=\exp\Big(-\frac{(x-x_0)^2+(y-y_0)^2+(z-z_0)^2}{\kappa^2}\Big)\begin{bmatrix}
%-(z-z0)A\\
%0\\
%(x-x0)A
%\end{bmatrix},
%\end{equation}
%where $A$, $(x_0, y_0)$ and $\kappa$ are tuneable constants. Here, we take $A=100$, $(x_0, y_0, z_0)=(0.8, 0.8, 0)$ and $\kappa=0.02$.

\section{Conclusion}
In this paper, we  proposed accurate algorithms for computing the involved temporal convolutions of the NRBCs for the time-dependent Maxwell's equations on a spherical artificial surface.  More precisely, we provided the explicit formulas of the convolution kernel functions defined by inverse Laplace transforms of special modified Bessel functions, and also derived a new formulation of the NRBC capacity operator.  With these at our proposal, the temporal convolutions in the NRBCs can be computed in a fast manner which therefore could offer an accurate way to reduce Maxwell's system in $\mathbb R^3$ to a bounded domain.
%A truncated model for time domain electromagnetic scattering is formulated by introducing artificial interface together with the NRBC. The usage of the artificial interface is to get rid of the global capacity operator $\mathscr T_b$ on the incident wave. New formulas of the time domain NRBC and explicit expression of the NRBK are derived and adopted.
 As a direct application of the truncated model, we considered  the modelling and accurate simulation of the time-domain invisibility cloaks. We derived a new model valid for general cloaking geometry for the design of time-domain full wave invisibility cloaks involving just one unknown field $\bs D$ and seemingly complicated convolution operators that could be evaluated recursively in time again.
%  Based on the new NRBC formulations and Drude model for dispersive media in the cloak, we formulate a model problem which only has one unknown field $\bs D$.
 In this work, we focused on the spherical invisibility cloaks designed in the first, original of Pendry et al (cf. \cite{pendry.2006}).   We proposed an efficient VSH-spectral-element method for numerical simulation. The resulted algorithm could produce accurate numerical solution with far less computation cost compared with the simulations based on FDTD in literature.

\section*{Acknowledgments}
The research of the first author is supported by NSFC (grant 11771137), the Construct Program of the Key Discipline in Hunan Province and a Scientific Research Fund of Hunan Provincial Education Department (No. 16B154).
The research of the third author is supported by the Ministry of Education, Singapore, under its MOE AcRF Tier 2 Grants (MOE2018-T2-1-059 and MOE2017-T2-2-144).

The authors would like to thank Dr. Xiaodan Zhao at the National Heart Centre in Singapore for the initial exploration of this topic when she was a research associate in NTU.

\begin{appendix}

	\section{Vector spherical harmonics} \label{spherehar}
	\renewcommand{\theequation}{A.\arabic{equation}}
	\setcounter{equation}{0}
	%\comm{No need to put it into appendix!!!!}
	We adopt the notation and setting as in N{\'e}d{\'e}lec \cite{Nedelec}.
	The spherical coordinates
	$(r,\theta,\varphi)$ are related to  the Cartesian coordinates
	${\bs r}=(x,y,z)$ via
	\begin{equation}\label{sphco}
	x=r\sin \theta \cos \varphi,\quad y=r\sin\theta \sin
	\varphi,\quad z=r\cos\theta,
	\end{equation}
	where $r\ge 0,\theta\in [0,\pi]$ and $\phi\in [0,2\pi).$ The corresponding  moving  (right-handed) orthonormal coordinate basis $\{\er,
	\et, \ep \}$ is given by
	\begin{equation}\label{localco}
	\begin{split}
	&\er=\hat{\bs r}={\bs r}/{r},\;\; \et=(\cos \theta \cos \varphi, \
	\cos\theta \sin \varphi, \ -\sin\theta),
	\;\;  \ep= (-\sin
	\varphi, \ \cos \varphi, \ 0).
	\end{split}
	\end{equation}
	Let  $\{Y_{l}^m\}$ be the spherical harmonics as normalized  in  \cite{Nedelec}, and let  $S$ be the unit sphere. Recall that
	\begin{equation}\label{nablaY}
	\nabla_S Y_l^m =\frac {\partial Y_l^m}{\partial\theta}\et+\frac 1 {\sin \theta}
	\frac {\partial Y_l^m}{\partial\varphi}\ep.
	\end{equation}

	The VSH family $\big\{ \bs Y_l^m , \bs\Psi_{l}^m,
	\bs \Phi_l^m\big\}:=\big\{ Y_l^m \er, \nabla_SY_{l}^m,
	\nabla_S Y_l^m \wedge \er\big\},$ which  has been used in the Spherepack
	\cite{Swa.S00} (also see \cite{Morse53}) forms a complete orthogonal basis of ${\bs L}^2(S):=(L^2(S))^3$  under the inner product:
	\begin{equation}\label{uvform}
	\langle \bs u, \bs v\rangle_S=\int_S \bs u\cdot \bar{\bs v}\, dS=\int_0^{2\pi}\hspace*{-3pt} \int_0^{\pi} \bs u\cdot \bar{\bs v}\, \sin\theta\, d\theta d\varphi.
	\end{equation}
	Define the subspace  of  ${\bs L}^2(S),$ consisting of the tangent components of the vector fields on $S$:
	\begin{equation}\label{tangcomp}
	{\bs L}_T^2(S)= \big\{\bs u\in {\bs L}^2(S) : \bs u\cdot \hat {\bs x}=0\big\}.
	\end{equation}
	The VSH $\{\bs\Psi_{l}^m, \vt\}$ forms a complete orthogonal basis of ${\bs L}_T^2(S).$ Consequently, the vector field expanded  in terms of VSH has a distinct separation  of tangential and normal components.
	For any vector fields $\bs u\in {\bs L}^2(S)$, we write
	\begin{equation}\label{Eexpansion}
	{\bs u}=u_{00}\bs Y_0^0+\sum_{l=1}^\infty\sum_{|m|=0}^l\big\{u_{lm}^{r}\,  \bs Y_l^m\, +u_{lm}^{(1)}\bs \Psi_{l}^m+u_{lm}^{(2)}\, \vt\big\},
	\end{equation}
	where we denote $\beta_l=l(l+1),$ and have
	\begin{equation}\label{uvwexp}
	u_{00}=\langle\bs u, \bs Y_0^0 \rangle_S,\;\;\;u_{lm}^r= \langle\bs u, \bs Y_l^m \rangle_S, \;\;\;     u_{lm}^{(1)}= {\beta_l^{-1}}\langle\bs u, \bs\Psi_{l}^m \rangle_S, \;\;\; u_{lm}^{(2)}= {\beta_l^{-1}}\langle\bs u, \vt \rangle_S.
	\end{equation}
	It is noteworthy that   given $\bs u$, we   can be computed $\{u_{lm}^r, u_{lm}^{(1)}, u_{lm}^{(2)}\}$ via the discrete VSH-transform using the Spherepack
	\cite{Swa.S00}, and vice versa by the inverse transform. Moreover, the normal component  solely involves the first term while the tangential component $\bs E_T$ of $\bs E$ involves the last two terms in \eqref{Eexpansion}.

	Now, we collect some frequently used vector calculus formulas. Define the differential operators:
	\begin{equation}\label{opsa}
	d_l^{\pm}=\frac d {dr}\pm \frac l r, \quad  \hat \partial_r=\frac d {dr}+\frac 1 r, \quad {\mathcal L}_l=\hat \partial_r^2-\frac{\beta_l}{r^2}=\frac{d^2}{dr^2}+\frac 2 r\frac{d}{dr}
	-\frac{\beta_l}{r^2},
	\end{equation}
	where $\beta_l:=l(l+1)$. For any given $f(r)$, the following properties can be derived from \cite{Hill54}:
	\begin{itemize}
		\item For divergence operator
		\begin{equation}
		{\rm div}\big(f \bs Y_{l}^m\big)=\Big(\frac d {dr}+\frac 2 r\Big) f\,Y_l^m,\quad{\rm div}\big(f \bs\Psi_{l}^m\big)=-\beta_l\frac{f} r \,\bs Y_l^m,\quad {\rm div}\big(f \bs \Phi_{l}^m\big)=0;\label{divformula}
		\end{equation}
		\item For curl operator
		\begin{equation}\label{vt2yy2}
		\nabla\times \big(f \bs Y_l^m \big)=\frac f r \,\bs\Phi_l^m,\quad
		\nabla\times \big(f \bs\Psi_{l}^m\big)=-\hat \partial_r f \,\bs\Phi_l^m, \quad
		\nabla\times \big(f \bs\Phi_l^m\big)=\hat \partial_r f \,\bs\Psi_{l}^m+\beta_l\frac f r \bs Y_l^m;
		\end{equation}
		\item For Laplace operator
		\begin{equation}
		\Delta \big(f\vt\big)={\mathcal L}_{l}(f)\vt.\label{newvect}
		\end{equation}
	\end{itemize}

	\renewcommand{\theequation}{B.\arabic{equation}}
	\setcounter{equation}{0}
	\section{Proof of Proposition \ref{newprob}}\label{proofnewprob}
	\begin{proof}
		Recall that if ${\rm div} {\bs u}=0,$ then
$		\nabla \times \nabla \times {\bs u}=-\Delta {\bs u}. $
		Thus, from \eqref{opsa}-\eqref{newvect}, we derive
		\begin{align*}
		&\nabla \times \nabla \times \big( u\vt  \big)=-\Delta \big(  u \vt   \big)=-\mathcal{L}_l(u)\vt,\\
		&\nabla \times \nabla \times \nabla \times \big( v \vt   \big)=-\nabla \times \big(  \Delta \big( v \vt   \big)  \big)=-\nabla \times \big(  \mathcal{L}_l\big(v\big)\vt   \big).
		\end{align*}
		Therefore, \eqref{HomoEq0} can be reduced to:
		\begin{equation}
		\frac{\partial^2u^i_{lm}}{\partial t^2}-c^2\mathcal{L}_l( u^i_{lm} )=f_{1,l}^{i,m}\quad \frac{\partial^2v^i_{lm}}{\partial t^2}-c^2\mathcal{L}_l( v^i_{lm} )=f_{2,l}^{i,m}, \quad r\in I_i,\quad i=0, 2, 3,
		\end{equation}
		for $|m|\leq l$, $l=1, 2, \cdots$, by using the expansions \eqref{unknownfieldsexp}.
		 %where differential operator $\mathcal{L}_l$ is defined in \eqref{opsa}.
In spherical coordinates (cf. \cite{Abr.I64}):
		\begin{equation}\label{sphericalcurl}
		\nabla \times {\bs v}= \frac 1{r\sin
			\theta}\Big( \frac{\partial \big(\sin\theta v_\varphi\big)}{\partial
			\theta} -\frac{\partial v_\theta}{\partial \varphi}\Big)\er + \frac
		1 r \Big(\frac 1 {\sin\theta} \frac{\partial v_r}{\partial \varphi}-
		\frac{\partial \big(r v_\varphi\big)}{\partial r} \Big)\et+\frac 1 r \Big(\frac{\partial \big(r v_\theta\big)}{\partial
			r}-\frac {\partial v_r}{\partial\theta} \Big)\ep,
		\end{equation}
		for any vector field $\bs v=v_r\er+v_{\theta}\et+v_{\varphi}\ep$. Apparently, we have
\(\nabla\times(u_{00}^i(r, t)\bs Y_0^0)=\bs 0,\)
		as $\bs Y_0^0=\er/\sqrt{4\pi}$. For the coefficient $u_{00}^i$, we then have
		\begin{equation}\label{u00govern}
		\frac{\partial^2u_{00}^i}{\partial t^2}=f_{00}^i,\quad r\in I_i,\quad i=0, 2, 3.
		\end{equation}
		
		We now turn to the governing equation \eqref{CloakingEq} in the cloaking layer $I_1=(R_1,R_2)$. According to \eqref{vt2yy2}, the vector spherical harmonic expansion of $\bs D^1$ can be rewritten as
		\begin{equation}\label{SphereExp}
		{\bs D}^1=u_{00}^1\bs Y_0^0+\sum_{l=1}^\infty\sum_{|m|=0}^l \Big\{   u_{lm}^1\bs\Phi_l^m+ \hat \partial_r v_{lm}^1  \bs{\Psi}_{l}^m+\frac{\beta_l}{r}v_{lm}^1 \bs Y_{l}^m     \Big\}.
		\end{equation}
		%\comm{what does this mean?}
		 Using \eqref{SphereExp} and the fact that $\mathscr{D}_1$ defined in \eqref{DispOperator1} is uniaxial, we have
		\begin{equation}
		\label{D1Di}
		\mathscr{D}_1[{\bs D}^1]=\big(u_{00}^1+\theta_1\ast u_{00}^1\big)\bs Y_{0}^0+\sum_{l=1}^\infty\sum_{|m|=0}^l\Big\{\epsilon^{-1}u_{lm}^1\bs\Phi_l^m+\epsilon^{-1}\hat \partial_r v_{lm}^1 \bs\Psi_{l}^m +\frac{\beta_l}{r}\big(v_{lm}^1+\theta_1\ast v_{lm}^1\big)\bs Y_{l}^m\Big\}.
		\end{equation}
		Using formula \eqref{sphericalcurl}, we have
		\begin{equation}
		\nabla\times\Big(\big(u_{00}^1+\theta_1\ast u_{00}^1\big)\bs Y_{0}^0\Big)=\bs 0.
		\end{equation}
		Then, we calculate from \eqref{D1Di} that
		\begin{equation}
		\label{curlD1D1}
		\begin{split}
		\nabla \times \big(\mathscr{D}_1[{\bs D}^1 ]   \big)=&\sum_{l=1}^\infty\sum_{|m|=0}^l\Big(\frac{\beta_l}{r^2}\big(v_{lm}^1+\theta_1\ast v_{lm}^1\big)-\epsilon^{-1} \hat \partial_r^2 v_{lm}^1\Big)\bs\Phi_l^m\\ &+\sum_{l=1}^\infty\sum_{|m|=0}^l\Big(\epsilon^{-1} \hat \partial_ru_{lm}^1 \bs\Psi_{l}^m+\epsilon^{-1}\frac{\beta_l}{r}u_{lm}^1\bs Y_{l}^m\Big)
		\end{split}
		\end{equation}
		by using formulas \eqref{vt2yy2}.
		Repeating the above calculation and using the definition of $\mathscr D_2$ and \eqref{vt2yy2}, we obtain
		\begin{equation}
		\begin{split}
		\nabla \times \big( \mathscr{D}_2&\big[\nabla \times \big(\mathscr{D}_1[{\bs D}^1] \big)\big]   \big)=\sum_{l=1}^\infty\sum_{|m|=0}^l\epsilon^{-1}\Big( \frac{\beta_l}{r^2}(u_{lm}^1+\theta_2\ast u_{lm}^1)-\epsilon^{-1}\hat \partial_r^2 u_{lm}^1 \Big )\bs\Phi_l^m\\
		&+\sum_{l=1}^\infty\sum_{|m|=0}^l\epsilon^{-1}\nabla \times \Big(\Big( \frac{\beta_l}{r^2}(v_{lm}^1+\theta_1\ast v_{lm}^1)- \epsilon^{-1}\hat \partial_r^2 v_{lm}^1      \Big)\bs\Phi_l^m               \Big),
		\end{split}
		\end{equation}
		Inserting the above equation into \eqref{CloakingEq}, one immediately shows that the expansion coefficients $\{u_{lm}^1$, $v_{lm}^1\}, |m|\leq l, l=1, 2, \cdots$ satisfy the same governing equation \eqref{vcloak} with different convolution kernels $\theta_2$ and $\theta_1$, respectively. As in \eqref{u00govern}, $u_{00}^1$ satisfies the same differential equation.

		According to \eqref{SphereExp} and \eqref{D1Di} and the facts
		\begin{equation}\label{VSHbasistimes}
		\bs{\Psi}_l^m\times\bs e_r=\bs{\Phi}_l^m,\quad \bs{\Phi}_l^m\times\bs e_r=-\bs{\Psi}_l^m,
		\end{equation}
		we have
		\begin{equation}\label{formulatimeser}
		\begin{split}
		&{\bs D}^i\times \bs e_r= \sum_{l=1}^\infty\sum_{|m|=0}^l\big(-u_{lm}^i\bs\Psi_l^m+ \hat \partial_r v_{lm}^i  \bs{\Phi}_{l}^m\big),\quad i=0, 2,3,\\
		&(\nabla\times{\bs D}^i)\times \bs e_r= \sum_{l=1}^\infty\sum_{|m|=0}^l\Big\{
		\Big(\hat \partial_r^2 v_{lm}^i- \frac{\beta_l}{r^2}v_{lm}^i\Big)\bs\Psi_l^m+\hat{\partial}_ru_{lm}^i \bs{\Phi}_{l}^m\Big\},\quad i=0, 2,3,
		\end{split}
		\end{equation}
		and
		\begin{equation*}
		\begin{split}
		&\mathscr{D}_1[{\bs D}^1]\times \bs e_r=\sum_{l=1}^\infty\sum_{|m|=0}^l\big(-\epsilon^{-1}u_{lm}^1\bs\Psi_l^m+\epsilon^{-1}\hat \partial_r v_{lm}^1 \bs\Phi_{l}^m\big),\\
		&\big(\nabla\times(\mathscr{D}_1[{\bs D}^1])\big)\times \bs e_r= \sum_{l=1}^\infty\sum_{|m|=0}^l\Big\{\Big(\epsilon^{-1} \hat \partial_r^2 v_{lm}^1- \frac{\beta_l}{r^2}\big(v_{lm}^1+\theta_1\ast v_{lm}^1\big)\Big)\bs\Psi_l^m +\epsilon^{-1} \hat \partial_ru_{lm}^1 \bs\Phi_{l}^m\Big\}.
		\end{split}
		\end{equation*}
		Substituting the above equations into jump condition \eqref{Interface} and \eqref{Interface1}, we obtain jump conditions
		\begin{equation}\label{d9}
		\begin{split}
		\epsilon u_{lm}^0=u_{lm}^1,\quad \partial_r v_{lm}^1=\epsilon \partial_r v_{lm}^0+(\epsilon-1)r^{-1}v_{lm}^0 \quad{\rm at}\;\;\; r=R_1,\\
		\epsilon u_{lm}^2=u_{lm}^1,\quad \partial_r v_{lm}^1=\epsilon \partial_r v_{lm}^2+(\epsilon-1)r^{-1}v_{lm}^2 \quad{\rm at}\;\;\; r=R_2,
		\end{split}
		\end{equation}
		and
		\begin{eqnarray}
		\partial_r u_{lm}^1=\epsilon^2 \partial_r u_{lm}^0+\epsilon(\epsilon-1)r^{-1}u_{lm}^0 \quad {\rm at}\;\;\; r=R_1, \label{d10}\\
		\partial_r u_{lm}^1=\epsilon^2 \partial_r u_{lm}^2+\epsilon(\epsilon-1)r^{-1}u_{lm}^2 \quad {\rm at}\;\;\; r=R_2,\label{d11}\\
		\hat \partial_r^2 v_{lm}^0- \frac{\beta_l}{r^2}v_{lm}^0=\epsilon^{-2} \hat \partial_r^2 v_{lm}^1- \frac{\beta_l}{\epsilon r^2}\big(v_{lm}^1+\theta_1\ast v_{lm}^1\big) \quad {\rm at}\;\;\; r=R_1, \label{vjumpcond1}\\
		\hat \partial_r^2 v_{lm}^2- \frac{\beta_l}{r^2}v_{lm}^2=\epsilon^{-2} \hat \partial_r^2 v_{lm}^1- \frac{\beta_l}{\epsilon r^2}\big(v_{lm}^1+\theta_1\ast v_{lm}^1\big) \quad {\rm at}\;\;\; r=R_2.\label{vjumpcond2}
		\end{eqnarray}
		Noting that $\hat{\partial}_r^2u=\frac{1}{r^2}\frac{\partial }{\partial r}\big(r^2\frac{\partial u}{\partial r}\big)$, the governing equations \eqref{vhomo0}-\eqref{vcloak} then gives
		\begin{equation}
		\label{jumpeq}
		\begin{split}
		\epsilon^{-2} \hat \partial_r^2 v_{lm}^1- \frac{\beta_l}{\epsilon r^2}\big(v_{lm}^1+\theta_1\ast v_{lm}^1\big)=\frac{1}{c^2}\frac{\partial^2v^1_{lm}}{\partial t^2},\quad\hat \partial_r^2 v_{lm}^i- \frac{\beta_l}{r^2}v_{lm}^i=\frac{1}{c^2}\frac{\partial^2v^i_{lm}}{\partial t^2},
		\end{split}
		\end{equation}
	for $i=0, 2, 3.$	Substituting \eqref{jumpeq} into the jump conditions \eqref{vjumpcond1}-\eqref{vjumpcond2} and integrate w.r.t. $t$ and using homogeneous initial conditions \eqref{initialv}, we derive
		\begin{equation}\label{continuityvlm}
		v^0_{lm}=v^1_{lm}\quad {\rm at}\;\;\; r=R_1; \quad v^2_{lm}=v^1_{lm}\quad {\rm at}\;\;\; r=R_2.
		\end{equation}
		Note that the jump conditions at artificial interface $r=R_3$ are trivial. Thus,
		we consider the boundary condition at $ r=b$.

		Applying expansion \eqref{unknownfieldsexp} in \eqref{DtN} and using identities \eqref{VSHbasistimes}, \eqref{formulatimeser} and formulation \eqref{newEtMdivfree}, we obtain
		\begin{equation}
		\begin{split}
		\sum_{l=1}^\infty\sum_{|m|=0}^l
		\!\Big(\partial_t\hat \partial_r v_{lm}^3+c\Big(\hat \partial_r^2 v_{lm}^3- \frac{\beta_l}{b^2}v_{lm}^3\Big)-\frac{c}{b^2}\omega_l\ast v_{lm}^{3}\Big)\bs\Psi_l^m&\\
		+\sum_{l=1}^\infty\sum_{|m|=0}^l\!\Big(\partial_tu_{lm}^3+c\hat{\partial}_ru_{lm}^3 -\frac c b\sigma_l\ast u_{lm}^{3}\Big)\bs{\Phi}_{l}^m&=\bs 0,
		\end{split}
		\end{equation}
		which implies two boundary conditions
		\begin{align}
		\frac{1}{c}\partial_tu_{lm}^3+\frac{{\partial}u_{lm}^3}{\partial r}+\frac{1}{b}u_{lm}^3 -\frac 1 b\sigma_l\ast u_{lm}^{3}=0\quad {\rm at}\;\; r=b,\label{unrbc}\\
		\frac{\partial}{\partial r}\frac{\partial v_{lm}^3}{\partial t}+\frac{1}{b}\frac{\partial v_{lm}^3}{\partial t}+c\Big(\hat \partial_r^2 v_{lm}^3- \frac{\beta_l}{b^2}v_{lm}^3\Big)-\frac{c}{b^2}\omega_l\ast v_{lm}^{3}=0\quad {\rm at}\;\; r=b.\label{vnrbc}
		\end{align}
		Here, the definition of differential operator $\hat{\partial}$ in \eqref{opsa} is applied. Obviously, the boundary condition for $u_{lm}^3$ is exactly the one we adopted in the model problem \eqref{veqsys}. Next, we will show that the equation \eqref{vnrbc} can be reformulated to the same form as \eqref{unrbc}. Indeed, we can directly calculate
		\begin{equation}
		\frac{c}{b^2}\omega_l(t)\ast v_{lm}^3(b, t)=\frac{1}{b}\bigg(\int_0^t\sigma_l'(t-\tau)v_{lm}^3(b, \tau)\,d\tau+\sigma_l(0)v_{lm}^3(b, t)\bigg)=\frac{1}{b}\partial_t(\sigma_l\ast v_{lm}^3(b,t)),
		\end{equation}
		by using the expression of $\omega_l(t)$ \eqref{kernelomega00}. Using the above equation and \eqref{jumpeq} in \eqref{vnrbc} gives
		\begin{equation}
		\frac{\partial}{\partial t}\Big\{\frac{ \partial v_{lm}^3}{\partial r}+\frac{1}{b}v_{lm}^3+\frac{1}{c}\frac{\partial v_{lm}^3}{\partial t}-\frac{1}{b}\sigma_l\ast v_{lm}^{3}\Big\}=0\quad {\rm at}\;\; r=b.
		\end{equation}
		Consequently, we obtain boundary condition \eqref{DtNv} by the zero initial data assumption.
		% \comm{Here is incomplete?}

		Note that the initial boundary value problems for coefficients $u_{lm}^i$ and $v_{lm}^i$ have almost the same form except the interface conditions \eqref{d9}-\eqref{d11} and \eqref{continuityvlm}. Apparently, by introducing the variable substitution \eqref{varisub}, $\{\widetilde u_{lm}^i\}$ satisfy the same governing equation as $\{u_{lm}^i$ and the same interface and boundary conditions as $v_{lm}^i$.
\end{proof}
	
\end{appendix}

%\bibliographystyle{plain}

%\bibliography{refpapers}

\end{document}